\documentclass[bj,authoryear]{imsart}
\startlocaldefs
\theoremstyle{plain}

\newtheorem{theorem}{Theorem}[section]
\newtheorem{lemma}[theorem]{Lemma}
\newtheorem{proposition}[theorem]{Proposition}
\newtheorem{corollary}[theorem]{Corollary}

\theoremstyle{remark}
\newtheorem{definition}[theorem]{Definition}

\usepackage{smile}

\newcommand{\T}{^\top}

\newcommand*{\TV}{\operatorname{\operatorname{TV}}}
\renewcommand*{\PP}{\mathbb{P}}
\newcommand*{\var}{\operatorname{Var}}

\usepackage{bbm}
\usepackage{mathtools}
\usepackage{mathrsfs} 
\usepackage{subcaption}
\usepackage{cancel} % \cancel{} draw diagonal lines to cancel an item

\endlocaldefs

\begin{document}

\begin{frontmatter}
\title{Non-asymptotic bounds for the $\ell_{\infty}$ estimator in linear regression with uniform noise}
%\title{A sample article title with some additional note\thanksref{t1}}
\runtitle{Non-asymptotic bounds for the $\ell_{\infty}$ estimator}
%\thankstext{T1}{A sample additional note to the title.}

\begin{aug}
\author[B]{\inits{Y.}\fnms{Yufei}~\snm{Yi}\ead[label=e1]{yy544@fb.com}}
\author[B]{\inits{M.}\fnms{Matey}~\snm{Neykov}\ead[label=e2]{mneykov@stat.cmu.edu}}%\orcid{0000-0000-0000-0000}}
%\author[B]{\inits{T.}\fnms{Third}~\snm{Author}\ead[label=e3]{third@somewhere.com}\ead[label=u1,url]{www.foo.com}}
%%%%%%%%%%%%%%%%%%%%%%%%%%%%%%%%%%%%%%%%%%%%%%
%% Addresses                                %%
%%%%%%%%%%%%%%%%%%%%%%%%%%%%%%%%%%%%%%%%%%%%%%
%\address[A]{Department,
%University or Company Name, City, Country\printead[presep={,\ }]{e1}}

\address[B]{Department of Statistics \& Data Science,
Carnegie Mellon University, Pittsburgh, USA\printead[presep={,\ }]{e1,e2}}
\end{aug}

\begin{abstract}
The Chebyshev or $\ell_{\infty}$ estimator is an unconventional alternative to the ordinary least squares in solving linear regressions. It is defined as the minimizer of the $\ell_{\infty}$ objective function 
\begin{align*}
    \hat\bbeta := \argmin_{\bbeta} \|\bY - \Xb\bbeta\|_{\infty}.
\end{align*}
The asymptotic distribution of the Chebyshev estimator under fixed number of covariates was recently studied \citep{knight2017asymptotic}, yet finite-sample guarantees and generalizations to high-dimensional settings remain open. In this paper, we develop non-asymptotic upper bounds on the estimation error $\|\hat\bbeta-\bbeta^*\|$ for a Chebyshev estimator $\hat\bbeta$, in a regression setting with uniformly distributed noise $\varepsilon_i\sim U([-a,a])$ where $a$ is either known or unknown. With relatively mild assumptions on the (random) design matrix $\Xb$, we can bound the error rate by $C_p/n$ with high probability, for some constant $C_p$ depending on the dimension $p$ and the law of the design. Furthermore, we illustrate that there exist designs for which the Chebyshev estimator is (nearly) minimax optimal. On the other hand we also argue that there exist designs for which this estimator behaves sub-optimally in terms of the constant $C_p$'s dependence on $p$. Finally, we show that ``Chebyshev's LASSO'' has advantages over the regular LASSO in high dimensional situations, provided that the noise is uniform. Specifically, we argue that it achieves a much faster rate of estimation under certain assumptions on the growth rate of the sparsity level and the ambient dimension with respect to the sample size.
\end{abstract}


\end{frontmatter}

\section{Introduction}

The goal of this paper is to analyze the non-asymptotic behavior of the Chebyshev estimator (and some of its close relatives) in a linear model with uniformly distributed errors. Concretely, suppose we have $n$ independent and identically distributed (i.i.d.) observations of the following model $Y_i = \bX_i\T \bbeta^* + \varepsilon_i$ where $\bX_i \in \RR^p$ are covariates, and $\varepsilon_i \sim U([-a,a])$ for some $a > 0$ which may be either known or unknown. Throughout the paper we will additionally assume that $\bX_i$ is independent of $\varepsilon_i$. A natural (although unconventional) estimator of $\bbeta^*$ is the Chebyshev (also known as $\ell_{\infty}$ or minimax) estimator which is defined through:
%We begin this section with introducing the Chebyshev estimator, which is an (unconventional) alternative to the least squares when the noise  distribution is known to have bounded support. Suppose that we are given $n$ i.i.d. observations from a linear model $Y_i = \bX_i\T \bbeta^* + \varepsilon_i$, where $\bX_i \in \RR^p$ are covariates, and $\varepsilon_i$ are i.i.d. random variables independent from the predictors $\bX_i$. The Chebyshev (also known as $\ell_{\infty}$ or minimax) estimator is defined through
\begin{align}\label{Chebyshev:estimator:intro}
\hat \bbeta := \argmin_{\bbeta} \max_{i \in \{1,\ldots,n\}} |Y_i - \bX_i\T \bbeta|.\footnotemark
\end{align}
\footnotetext{Formally, we should write $\hat \bbeta \in \argmin_{\bbeta} \max_{i \in \{1,\ldots,n\}} |Y_i - \bX_i\T \bbeta|$, as the minimizer may not be unique. However since our results are valid for any point in this set we abuse the notation slightly.}Compared to ordinary least squares (OLS), the Chebyshev estimator minimizes the $\ell_{\infty}$ rather than the $\ell_2$ norm of the estimated residuals. The motivation of \eqref{Chebyshev:estimator:intro} stems from the fact that this is the MLE when the noise is known to be uniform on a bounded interval $U([-a,a])$ where the value of $a$ is unknown (see also Section \ref{main:section} for this simple calculation). It is easy to see that \eqref{Chebyshev:estimator:intro} can be conveniently solved through a linear program. Alternatively, there exist iteratively reweighted least squares schemes, originally due to Lawson \citep{lawson1961contribution, cline1972rate}, which can be shown to converge to the solution of \eqref{Chebyshev:estimator:intro} at a linear rate. Intuitively, \eqref{Chebyshev:estimator:intro} will be a good estimator of $\bbeta^*$ when the noise has bounded support with non-zero probability mass near the boundary or the noise is concentrated on a bounded interval with very thin tails outside the interval. In contrast, when the noise is of unbounded support or whenever there is a negligible probability of the noise being near the boundary the Chebyshev estimator might have poor performance or may be even inconsistent. Importantly, observe that the $\ell_\infty$ estimator \eqref{Chebyshev:estimator:intro} is not a linear estimator of the observations, and therefore Gauss-Markov's theorem is not applicable --- which leaves the door open for the Chebyshev estimator to dominate OLS on some occasions. We will verify that this is indeed the case. 

Apart from being a cute mathematical problem, regression with uniform errors can be motivated in problems where the error is naturally bounded. For instance if the observations undergo some physical measurement process (such as measuring weight on a scale) it may be natural to assume that the error has bounded support. Although one may argue that uniform distribution is not necessarily the most natural bounded distribution, we find it enlightening to study this model, in part because the uniform distribution is naturally related to the order statistics of any continuous distribution. As we shall see the order statistics play a big role in our non-asymptotic analysis of the performance of the Chebyshev estimator, and therefore we believe the methods we develop here are (much) more broadly applicable. In fact all of our proofs can readily be extended to continuous, symmetric, bounded noise with almost no efforts (see Remark \ref{remark:extension}). Finally, it is also possible to extend the results to cases with asymmetric noise, at the cost of a slightly cumbersome argument on symmetrizing the observations (and critical inequalities) one considers in the proofs. We avoid doing this here in order to keep our exposition as clean as possible. 

\subsection{Related work and contributions}

 Although the Chebyshev estimator is not extensively used in practice, there certainly has been some interest coming from various fields. In particular it has found applications in the physical and environmental sciences \citep{james1983fitting, james1983probability, brenner2002aeroservoelastic, zolghadri2004minimax, bertsch2005fitting, qi2015theoretical}, finance \citep{jaschke1997arbitrage, jaschke2001coherent}, and there is also a considerable literature in signal processing on estimation with bounded noise \cite[see][e.g., and references therein]{milanese1982estimation, tse1993optimal, akccay1996choice, alecu2006wavelet, beck2007regularization}. In addition  there is a lot of literature on Chebychev's estimation, dealing with the computational aspects of these estimators using linear programming and numerical analysis \citep{appa19731, sielken1973two, hand1980using, armstrong1980dual, sklar1982least}. Recent statistical studies of the Chebyshev estimator include \citep{castillo2009combined, knight2017asymptotic, berenguer2019models, du2019optimal}. Of note \cite{du2019optimal}, present a high-dimensional problem in composite fuselage assembly where a regularized Chebyshev estimator is a natural choice (i.e. the $\ell_{\infty}$ loss seems more natural in this problem compared to other standard loss functions such as the $\ell_2$ and $\ell_1$ losses). \cite{du2019optimal} provide some statistical guarantees about a dual version of what we call Chebyshev's LASSO below (see \eqref{chebyshev_lasso_est}), assuming that the noise is sub-Gaussian, but under strong assumptions on the design matrix, and most importantly they fail to recognize that if the noise is uniform this estimator will actually outperform the LASSO in terms of the estimation convergence rate (see Theorem \ref{theorem:l2_bound_chebyshev_lasso}). 

% ("fitting tracks in certain detectors such as multiwire proportional chambers with Chebyshev vs least squares")
%\cite{james1983probability} ("did not find access to this book, but it definitely mentions Chebyshev's estimate")
%\cite{brenner2002aeroservoelastic} ("using Chebyshev's estimator for uncertainty bound estimation from flight data")
%\cite{zolghadri2004minimax} ("statistical modeling of high frequency air pollution data.")
%\cite{bertsch2005fitting} ("evaluating mass fits in the liquid drop model with Chebyshev norm")
%\cite{qi2015theoretical} ("is a follow up to Bertsch")
%\cite{jaschke1997arbitrage} ("")
%\cite{jaschke2001coherent} ("")
%
%\cite{castillo2009combined}

Remarkably, even though the Chebyshev estimator has been around for a long time, partly as folklore knowledge, studies of its behavior have been very limited. We attribute this fact to the relatively complicated form of the loss function which is non-smooth. To our knowledge the first proper attempts at characterizing the rates of convergence of the Chebyshev estimator (in a regression setting) are due to \citep{robbins1986maximum, schechtman1986estimating}. These authors used very clever ideas, to derive the rate of convergence of the Chebyshev estimator in a simple linear regression case. Both papers observed the ``super-efficient'' behavior of the Chebyshev estimator in comparison to the OLS. The case with more than one covariate remained unsolved for nearly 30 years, and in 2020 Keith Knight in a breakthrough preprint (which seems to first have been released in 2010 and later revised in 2017 and 2020) derived the exact asymptotic distribution of the Chebyshev estimator with a fixed (but potentially bigger than one) number of covariates \citep{knight2017asymptotic}. The author showed, that the Chebyshev estimator converges to its target at a $n^{-1}$ rate (in the uniform noise case), which should be contrasted to the much slower $n^{-1/2}$ rate for the OLS. The asymptotic distribution however is complicated, and non-pivotal, which means that it cannot be used to perform inference. % In addition, it shouldn't be surprising that such a fast rate can be achieved since the regression model with bounded errors is not a regular model, and in such situations one can take into consideration the bounded supports to arrive at estimators more efficient than the standard parametric efficiency rate of $n^{-1/2}$. 
While being a landmark, the work of \cite{knight2017asymptotic} left a lot to be desired. For one, the rate of convergence to the asymptotic distribution is unknown. This means that finite sample results which hold with high probability cannot be extracted easily from the main result of \cite{knight2017asymptotic}. In addition, due to the complicated form of the asymptotic distribution, it is not straightforward to derive the dependence on the dimension in the rate of convergence for the estimator. This paper proposes a novel non-asymptotic approach, which is able to derive finite-sample guarantees, and in addition can be used to give a rough upper bound for the dependence on the dimension in the convergence rate. In addition, we formalize and analyze Chebyshev's LASSO, which extends the Chebyshev estimator to high-dimensional settings by incorporating an $\ell_1$-penalty. We demonstrate that Chebyshev's LASSO can be much more efficient than the regular LASSO in models where the noise is uniform under certain assumptions. %We elaborate on this in Section \ref{proj2:extensions:and:questions}. 

\subsection{Organization}

The paper is structured as follows. In Section \ref{main:section} we record our main results on the Chebyshev estimator (and its relatives). Subsection \ref{simple:analysis:section} uses a simple analysis which as we argue captures a multitude of random designs, while Subsection \ref{minimax:bound:section} derives a minimax lower bound for the problem. In Section \ref{LASSO:section} we state our main result for Chebyshev's LASSO, which illustrates that Chebyshev's LASSO can be much more accurate than the regular LASSO under certain assumptions. We provide brief numerical results in support of our theoretical findings in Section \ref{sim:section}. Section \ref{discussion:section} is dedicated to a brief discussion.

%%%%%%%%%%%%%%
\subsection{Notation}

%{\color{red} need to fill that one up}

We use $\lesssim$ and $\gtrsim$ to mean $\leq$ and $\geq$ up to positive universal constants. By convention for any integer $n \in \NN$ we set $[n] = \{1,\ldots, n\}$. We use $\mathbb{B}^{p}_2$ to denote the $p$-dimensional Euclidean ball, while $\mathbb{S}^{p-1}$ to denote the $p$-dimensional Euclidean sphere. For a vector $\vb \in \RR^p$ we denote $\|\vb\|_q = [\sum_{i \in [p]} |\vb^{(i)}|^q]^{1/q}$ its $\ell_q$-th norm (with the usual extension for $q=\infty$), and we use $\|\cdot\|$ as a shorthand for $\|\cdot\|$. For two vectors $\vb, \wb \in \RR^p$ we denote their dot product with either $\vb\T \wb$ or with $\langle \vb,\wb\rangle$. For a matrix $\Ab\in\mathbb{R}^{n\times p}$ we use $\|\Ab\|_{\max}$ to denote the largest absolute value of all its entries, and use $\|\Ab\|_{\infty}$ to denote the largest $\ell_1$ norm of all its rows such that $\|\Ab\|_{\infty} = \max_{i\in n} \sum_{j\in p}|\Ab_{ij}|$. We denote the operator norm of a matrix $\Ab$ with $\|\Ab\|_{\operatorname{op}}$.

%%%%%%%%%%%%%%%
%% LS with unif noise
%%%%%%%%%%%%%%%%
\section{Linear model with uniform noise}\label{main:section}
% In order to appreciate the intuition of the Chebyshev estimator, we firstly study the properties of the constrained least squares estimator.
Consider a linear model 
\begin{align}
\label{linear_model_unif_error}
    \bY = \Xb\bbeta^* + \bvarepsilon, \quad\varepsilon_i\sim U([-a,a]),
\end{align}
where $a>0$ is a known constant. Here $\bY$ is a vector of $n$ outcome values, $\Xb$ is an $n \times p$ matrix whose rows are the covariates and $\bvarepsilon$  is the vector of the error terms (which we assume is independent of $\Xb$). OLS is probably the most commonly used method to estimate $\bbeta^*$ in linear models, with an estimation rate $\|\hat\bbeta - \bbeta^*\|\sim O(\sqrt{p/n})$ when $p$ is changing with $n$, and $\bX \sim N(0,\Ib)$ \citep[see equation (18)]{mourtada2019exact}.
Since in our problem the noise $\bvarepsilon$ is bounded, by incorporating this information, we expect that the following constrained optimization
\begin{align}
\hat\bbeta = \argmin_{\bbeta} \frac{1}{n}\sum_{i \in [n]} (Y_i - \bX_i\T \bbeta)^2 \mbox{ s.t. } \label{least:squares:under:constraint} \\
 Y_i - a\leq \bX_i\T \bbeta \leq Y_i + a ~\forall i \in [n],\label{main:optimization:problem}
\end{align}
may give a better estimation of $\bbeta^*$ compared to the OLS. Clearly \eqref{least:squares:under:constraint} given \eqref{main:optimization:problem} can be solved via quadratic programming.  In addition, one could consider the best risk equivariant estimator in this problem, which is given by the centroid of the constraint set in \eqref{main:optimization:problem} \cite[see equation (1.5) and also references therein]{jureckova2009minimum}, although it may be hard to calculate it in practice \citep{rademacher2007approximating}. %(https://citeseerx.ist.psu.edu/viewdoc/download?doi=10.1.1.161.7335&rep=rep1&type=pdf, Kakosyan et al. too).
Our analysis will simultaneously cover both estimators considered above. In fact our analysis covers any estimator taking values in the set \eqref{main:optimization:problem}. 

We now consider the situation when $a$ is unknown. In this case, none of the two proposed estimators can be implemented since both of them rely on the knowledge of $a$. A natural approach would be to obtain the MLE. The likelihood function is
\begin{align*}
\mathcal{L}(a, \bbeta) = 1/(2a)^n \prod_{i \in [n]} \mathbbm{1}(|Y_i - \bX_i\T \bbeta| \leq a).
\end{align*}
Hence the MLE of $a$ and $\bbeta^*$ is given by the following linear program (where the inequalities are entrywise):
\begin{align*}
\min a,\\
\bY \leq a + \Xb \bbeta,\\
\Xb \bbeta - a \leq \bY.
% a \geq 0. \mbox{ this constraint is redundant and can be dropped}
\end{align*}
Clearly, this is equivalent to minimizing the loss function $\|\bY-\Xb\bbeta\|_{\infty}$. Thus the MLE of $\bbeta^*$ is given by
\begin{align}
\label{chebyshev_est}
    \hat\bbeta := \argmin_{\bbeta} \|\bY-\Xb\bbeta\|_{\infty},
\end{align}
which is also called the $\ell_{\infty}$ estimator or the Chebyshev estimator. Consequently, $a$ can be estimated by
\begin{align*}
    \hat a = \|\bY-\Xb\hat\bbeta\|_{\infty}.
\end{align*}
Observe that trivially we must have $\hat a \leq a$. This implies that when $\hat \bbeta$ is the Chebyshev estimator, it also satisfies \eqref{main:optimization:problem} even though $a$ is unknown. As we mentioned previously, all results below will be valid for any estimator $\hat \bbeta$ which takes values in the set \eqref{main:optimization:problem}, hence they are automatically valid for the Chebyshev estimator as well.

Let $\eta_i = -\sign(\varepsilon_i),\,\,i\in[n]$ be independent Rademacher random variables which are also independent from $\bX_i$ and $|\varepsilon_i|$. Let
$$\tilde\bX_{i} = \eta_i\bX_{i}.$$
We will now introduce the concept of a \textit{critical inequality} given in \eqref{critical_ineq_idv}. 

\begin{lemma}
\label{lemma:critical_ineq}
From \eqref{main:optimization:problem} one can deduct the inequality
\begin{align}\label{critical_ineq_idv}
\tilde\bX_{i}\T (\hat\bbeta - \bbeta^*) \leq a - |\varepsilon_{i}|.
\end{align}
\end{lemma}
Although we use the term \emph{critical inequality} to refer to any inequality of the type \eqref{critical_ineq_idv}, we will actually only use these inequalities for which the $|\varepsilon_i|$ value happens to be close to $a$. This justifies the term critical, as the right hand side of such an inequality is very close to $0$. Hence, if we are lucky enough and $\tilde \bX_i\T (\hat \bbeta - \bbeta^*)$ is not too small, a critical inequality will yield that $\tilde \bX_i\T (\hat \bbeta - \bbeta^*) \approx 0$. If we have many such approximate identities, it should be the case that $\hat \bbeta \approx \bbeta^*$. While this is not exactly how our analysis proceeds, we hope this gives a good intuition why critical inequalities may be useful. It is also worth stressing the fact that $\tilde\bX_{i}$ is sign symmetric regardless of the distribution of $\bX_i$. In the next section, we will present a simple way of obtaining bounds on the $\ell_2$ estimation error $\|\hat \bbeta - \bbeta^*\|$. Although initially it may seem that the restrictions we impose on the random design are somewhat severe, contrarily, through examples we show that there is a multitude of designs which obey these assumptions.

\subsection{A simple non-asymptotic analysis of estimators taking values in \eqref{main:optimization:problem}}\label{simple:analysis:section}

The high level intuition of the analysis we give in this section is very simple. First, note that due to the nature of the uniform distribution, there will be a significant proportion of critical inequalities whose right hand side will be close to $0$. Suppose now that we are able to establish that there exists a ``reasonably large'' $\mathbf{0}$-centered $\ell_2$-ball inside the convex hull of the $\tilde \bX_i$, for indices $i$ which correspond to the critical inequalities which are close to $0$. This will automatically mean that the $\|\hat \bbeta - \bbeta^*\|$ has to be bounded by the largest deviation from $0$ in the considered critical inequalities. Formally, we have:

\begin{theorem}\label{ball:in:convex:hull:follows:rate} Suppose that the design $\bX_i \in \RR^p$ is random and is independent of the noise $\varepsilon_i$. Let $f(p,\gamma)$ be a known function of the dimension and the scalar $\gamma > 0$. Assume that the design is such that for any integer $m \geq f(p,\gamma)$, and an i.i.d. sample $\{\bX_i\}_{i \in [m]}$ from the design we have
%In addition, suppose that the distribution of the design --- $\cL(\bX)$ --- has the property that for any integer $m \geq f(p,c)$, where $f$ is some known function, given an $m$ sample from $\cL(\bX)$ we have
\begin{align*}%\label{ball:in:convex:hull}
\PP(\xi \mathbb{B}^{p}_2 \not \subset \operatorname{conv}(\eta_1\bX_1, \ldots, \eta_m\bX_m)) \leq \gamma,
\end{align*}
for some $\xi > 0$, where $\eta_i$ are i.i.d. Rademacher random variables which are also independent from the design. Then, for any estimator $\hat \bbeta$ taking values in the set \eqref{main:optimization:problem}, we have that for any $L > 0$
\begin{align*}
    \|\hat \bbeta - \bbeta^*\| \leq \frac{ a (L + 1) \lceil f(p,\gamma) \rceil}{\xi n},
\end{align*}
with probability at least $1-\gamma - \exp\bigg(\frac{- L^2}{\frac{8}{3}L + 2}\bigg)$.
\end{theorem}

\begin{remark} One can see that when the constant $L$ is fixed we do obtain constant probability bounds (which decay exponentially with $L$). Perhaps with slight abuse of terminology, throughout the paper we refer to this type of bound as a ``high probability'' bound, even though it does not decay to $0$ as $n$ goes to $\infty$. This is similar in spirit to how one can only obtain constant confidence bounds for the expression $|\sum_{i \in [n]} X_i/n  - \mu| < C/\sqrt{n}$ for any constant $C > 0$, where $X_i \sim N(\mu,1)$ (since the variable $\sqrt{n}(\sum_{i \in [n]} X_i/n  - \mu) \sim N(0,1)$). 
\end{remark}

\begin{proof} Sort the absolute values of the errors $|\varepsilon_i| \sim U([0,a])$ in a decreasing manner $|\bar \varepsilon_{(i)}|$, so that $a \geq |\bar \varepsilon_{(1)}| \geq\ldots \geq |\bar \varepsilon_{(n)}| \geq 0$. Take the first $\lceil f(p,\gamma) \rceil$ many of them.
%, and consider $|\bar \varepsilon_{(\lceil f(p,c) \rceil)}|/a \sim Beta(n - \lceil f(p,c) \rceil + 1, \lceil f(p,c) \rceil)$. Using Chebyshev's inequality for a constant $\kappa >0$ we can obtain
%\begin{align*}
%    &\PP\Big[ \frac{|\bar \varepsilon_{(\lceil f(p,c) \rceil)}|}{a} < \frac{n - \lceil f(p,c) \rceil + 1}{n+1} - \kappa \frac{(n - \lceil f(p,c) \rceil + 1)\lceil f(p,c) \rceil}{(n+1)^2(n+2)}\Big] \leq \frac{1}{\kappa^2}
%\end{align*}
%
%Now, set $\kappa = n + 2$, to obtain
%\begin{align*}
%    \frac{n - \lceil f(p,c) \rceil + 1}{n+1} - \kappa \frac{(n - \lceil f(p,c) \rceil + 1)\lceil f(p,c) \rceil}{(n+1)^2(n+2)} & = \frac{n - \lceil f(p,c) \rceil + 1}{n+1} \bigg (1- \frac{\lceil f(p,c) \rceil\kappa}{(n+1)(n+2)}\bigg) \\
%& > 1 - \frac{2\lceil f(p,c) \rceil}{n+1}.%> (1-\frac{1}{n^{1-\alpha}})(1-\frac{c}{n+2}),
%\end{align*}
%
By Lemma \ref{lemma:uniform:random:variables:concentration} we know that:
\begin{align*}
    &\PP\bigg( \frac{|\bar \varepsilon_{(\lceil f(p,\gamma) \rceil)}|}{a} < 1 - \frac{(L + 1)\lceil f(p,\gamma) \rceil}{n}\bigg) \leq \exp\bigg(\frac{- L^2/2}{\frac{4}{3}L + 1}\bigg)\end{align*}
Let $\cE$ be the complement of the event in the probability above. Now, by Lemma \ref{lemma:critical_ineq}, on the event $\cE$ we have:
\begin{align}\label{critical:inequalities:proof:thm22}
    \eta_i \bX_i\T (\hat \bbeta - \bbeta^*) \leq a - a\bigg(1 - \frac{(L + 1)\lceil f(p,\gamma) \rceil}{n}\bigg) = \frac{a (L + 1)\lceil f(p,\gamma) \rceil}{n},
\end{align}
for all $i$ corresponding to the $\lceil f(p,\gamma) \rceil$ largest in magnitude $\varepsilon_i$'s (denote this index set by $S$). Since with probability at least $1 - \gamma$, we have the $\xi \frac{(\hat \bbeta - \bbeta^*)}{\|\hat \bbeta - \bbeta^*\|} \in \operatorname{conv}(\{\eta_i \bX_i\}_{i \in S})$ we can write
\begin{align*}
\xi \frac{(\hat \bbeta - \bbeta^*)}{\|\hat \bbeta - \bbeta^*\|} = \sum_{i = 1}^n \alpha_i \eta_i \bX_i,
\end{align*}
 where $\sum_{i = 1}^n \alpha_i = 1$ and $\alpha_i \geq 0$. We can now multiply the inequalities \eqref{critical:inequalities:proof:thm22} by $\alpha_i$ and sum them up to obtain the desired conclusion upon rearranging terms, and using the union bound.
 %we complete the proof by a union bound.    
\end{proof}

\begin{remark}\label{remark:extension} The above theorem can be readily generalized to settings where the noise is continuous, symmetric and bounded on an interval $[-a,a]$ but is not necessarily uniform. All that needs to be done is to replace the application of Lemma \ref{lemma:uniform:random:variables:concentration} with Lemma \ref{extension:nonuniform:noise}. In fact, all of our results can be extended to cover this more general case with almost no efforts. We do not pursue this further here to keep the exposition simple. It should be noted however, that while the upper bound results can be extended to the more general setting of symmetric bounded noise, the optimality of the Chebyshev estimator in such a setting is less clear. \\%In fact, one would not expect it to be a rate optimal estimator in this setting since it is no longer the MLE (or the best equivariant estimator). }
\end{remark}

\begin{example}\label{stupid:design:optimality}
We will now exhibit a simple example of a random design which satisfies the condition imposed in Theorem \ref{ball:in:convex:hull:follows:rate}. Although this example may appear contrived at this point, it is an important example for assessing the difficulty of estimation of $\bbeta^*$, as we will see later when we discuss a minimax lower bound. More natural design examples will follow below. Take the random design $\bX_i \sim U(\sqrt{p}\{\vb_1, \ldots, \vb_p\})$, where $\vb_i, i \in [p]$ denote vectors from any orthonormal basis. We therefore have $\tilde \bX_i \sim U(\sqrt{p}\{\pm\vb_1, \ldots, \pm\vb_p\}) $. 
%In addition, it can be seen that this example is a special case of a much more general class of examples considered in Example \ref{most:important:example} below. However due to its importance, we will analyze this example here directly.
 
First we will show that if all vectors $\{\pm \sqrt{p} \vb_j\}_{j \in [p]}$ are present within the $m$ considered samples we have a $\mathbf{0}$-centered $\ell_2$-ball inside. Take any point on $\bx \in \mathbb{B}^{p}_2$, and write it as $\bx = \sum_{j \in [p]} a_j \vb_j$. We have that $\sum a_j^2 \leq 1$, and hence $\sum_{j \in [p]} |a_j| \leq \sqrt{p}$. This means that we can represent $\bx = \sum_j \alpha_j (\sign(a_j) \sqrt{p} \vb_j)$, where $\sum_j \alpha_j \leq 1, \alpha_j \geq 0$, where $\alpha_j = |a_j|/\sqrt{p}$. On the other hand since clearly $\mathbf{0} \in \operatorname{conv}(\{\pm \sqrt{p} \vb_j\}_{j \in [p]})$ this implies that $\bx  \in \operatorname{conv}(\{\pm \sqrt{p} \vb_j\}_{j \in [p]})$, and since $\bx$ was arbitrary $\mathbb{B}^{p}_2 \subset \operatorname{conv}(\{\pm \sqrt{p} \vb_j\}_{j \in [p]})$. 

Now, it suffices to show that with high probability the set $\{\eta_i \bX_i\}_{i \in [m]}$ contains all vectors from the set $\{\pm \sqrt{p} \vb_j\}_{j \in [p]}$. The probability that a specific vector is not in this set is $(1 - 1/(2p))^m$, hence by a union bound we obtain an upper bound $2p (1 - 1/(2p))^m$. Hence since $2p (1 - 1/(2p))^m \leq 2p \exp(-m/(2p))$, for $m \geq 2p \log ( 2 \gamma^{-1} p)$ we have this probability is bounded by $\gamma$. Therefore by Theorem \ref{ball:in:convex:hull:follows:rate} we can conclude that with probability at least $1 -\gamma - \exp(- L^2/(8L/3 + 2))$ we have $\|\hat \bbeta - \bbeta^*\| \leq a(L + 1) (2p \log ( 2 \gamma^{-1} p) + 1)/n$.

%The above design also contains an $\ell_2$ ball of radius $1$. WLOG we can assume the basis is standard (since rotations won't change the desired inclusion). Hence we can use  $\sqrt{\sum_{i \in [p]} a_i^2}\sqrt{p} \geq \sum_{i \in [p]} |a_i|$, once we have $p\log p$ observations whp since ($p (1-1/p)^{p c \log p}$ can be made arbitrarily small).
\end{example}

Next, we will formalize a sufficient condition under which the design must contain a large $\ell_2$ ball. 

\begin{theorem}
\label{lemma:wendel_v3}
Let $\tilde \bX_1, \ldots, \tilde\bX_m$ be i.i.d. random points in $ \RR^p$, whose distribution is symmetric about $\mathbf{0}$.  If the distribution of $\tilde\bX$ satisfies
\begin{align*}
   \rho := \sup_{\vb \in \mathbb{S}^{p-1}}\frac{\PP\big( |\langle \vb, \tilde \bX \rangle| \leq 2\xi\big)}{2} + \PP(\|\tilde \bX\| \geq \Upsilon) < \frac{1}{2},
\end{align*}
% {\color{red} The above should be with the $L_1$-sphere.}
for some $\xi, \Upsilon > 0$, then
\begin{align*}%\label{prob:bound:l2:ball}
    \PP(\xi \mathbb{B}^{p}_2 \not \subset \operatorname{conv}(\tilde\bX_1, \ldots, \tilde\bX_m)) \leq \bigg(1 + \frac{2\Upsilon}{\xi}\bigg)^p \big(\frac{1}{2} + \rho\big)^m.
    % f(\xi) e^{1/2}(1-\varrho)^n + f(\xi) (1-\varrho)^{n-1-4(p-1)},
\end{align*}
where $\mathbb{B}^{p}_2$ is the $\ell_{2}$ ball centered at $\mathbf{0}$.% \begin{align*}
%     f(\xi) = \exp\bigg(2p\log(1 + 1/(2\xi^2p)) \wedge \frac{1}{\xi^2}\log(1 + 2p\xi^2)\bigg)
% \end{align*}
\end{theorem}
%
%\begin{remark} The original result can be recovered by taking $\phi$ to be the identity map and $a = c \EE \|\tilde \bX\|$.
%\end{remark}
%
\begin{remark} By the extended Markov's inequality condition \eqref{strange:condition} is satisfied if for some monotonically increasing positive function $\phi$, assuming $\EE \phi(\|\tilde \bX\|) < \infty$, we have
\begin{align*}%\label{strange:condition}
   \rho \leq \sup_{\vb \in \mathbb{S}^{p-1}}\frac{\PP\big( |\langle \vb, \tilde \bX \rangle| \leq 2\xi\big)}{2} + \frac{\EE \phi(\|\tilde \bX\|)}{\phi(\Upsilon)}.
\end{align*}
Therefore the theorem statement continues to hold with 
\begin{align*}
\rho :=  \sup_{\vb \in \mathbb{S}^{p-1}}\frac{\PP\big( |\langle \vb, \tilde \bX \rangle| \leq 2\xi\big)}{2} + \frac{\EE \phi(\|\tilde \bX\|)}{\phi(\Upsilon)}.
\end{align*} One simple instance that we will be using throughout the paper is when $\phi(x) = x$. Assuming that $\EE \|\tilde \bX\| < \infty$ and setting $\Upsilon = c \EE \|\tilde \bX\|$ in the definition of $\rho$ above we obtain that if
\begin{align}\label{strange:condition}
\rho := \sup_{\vb \in \mathbb{S}^{p-1}}\frac{\PP\big( |\langle \vb, \tilde \bX \rangle| \leq 2\xi\big)}{2} + c^{-1} \leq \frac{1}{2},
\end{align}
then 
\begin{align}\label{prob:bound:l2:ball}
    \PP(\xi \mathbb{B}^{p}_2 \not \subset \operatorname{conv}(\tilde\bX_1, \ldots, \tilde\bX_m)) \leq \bigg(1 + \frac{2c  \EE \|\tilde \bX\|}{\xi}\bigg)^p \big(\frac{1}{2} + \rho\big)^m.
\end{align}

The proof of Theorem \ref{lemma:wendel_v3} is elementary and is based on a covering argument.  Furthermore, the proof can be extended to any $\ell_q, q \geq 1$ norm ball. We do not pursue this here in order to simplify the presentation, and since it is not very useful for our purposes (which are to derive bounds on $\|\hat \bbeta - \bbeta^*\|$). In passing we would also like to mention a recent reference \citep{guedon2022geometry} which studies the geometry of the absolute convex hull of $n$ i.i.d. observations $\bX_1,\ldots, \bX_n$, i.e., they study the geometry of $\operatorname{conv}\{\pm\bX_1,\ldots,\pm\bX_n\}$, and show that this set contains a deterministic set associated with the law of the random vectors $\bX_i$. This is result is related to but is of different nature compared to Theorem \ref{lemma:wendel_v3}. %In addition the requirement $\EE \|\tilde \bX\| < \infty$ can be relaxed, by using the ``generalized'' Markov inequality instead of Markov's inequality. For completeness we state and prove this result in the appendix under Theorem \ref{lemma:wendel_v3}. %Additionally the sufficient condition in the second part of the theorem works when 
%\begin{align*}
%R := c^{-1} + \frac{1-(1-\theta)^2\inf_{\vb \in \mathbb{S}^{p-1}}\frac{(\EE |\langle \vb, \tilde\bX\rangle|)^{q/(q-1)}}{[\EE |\langle \vb, \tilde\bX\rangle|^q]^{1/(q-1)}}}{2} < \frac{1}{2},
%\end{align*} 
%with $\rho = R$ for any $q > 1$. 
\end{remark}

\begin{proof}[Proof of Theorem \ref{lemma:wendel_v3}]
Let $\wb \in \xi \mathbb{B}^{p}_2$ be an arbitrary vector such that $\|\wb\| \leq \xi$. We are interested when is the point $ - \wb$ in $\operatorname{conv}(\tilde\bX_1, \ldots,\tilde \bX_m)$, which is equivalent to $\mathbf{0}$ belonging to the convex hull $\operatorname{conv}(\tilde\bX_1 + \wb, \ldots, \tilde\bX_m + \wb)$. Note that this happens when there does not exist a $\vb$ ($\vb \neq \mathbf{0}$) such that for all $i\in[n]$
\begin{align*}
\langle \vb, \tilde\bX_i + \wb \rangle\geq0 \quad\Rightarrow 
\langle \vb,  \tilde\bX_{i} \rangle \geq -\langle \vb, \wb\rangle \geq -\|\vb\| \| \wb\| \geq -\xi \|\vb\|.
\end{align*}
So if such a $\vb \in \mathbb{S}^{p-1}$ satisfying $\langle \vb,  \tilde \bX_{i} \rangle \geq -\xi $  for all $i$ does not exist, then we are guaranteed to have  $ -\wb\in\operatorname{conv}(\tilde\bX_1, \ldots,\tilde \bX_m)$. Since $\wb$ is arbitrary it will follow that $\xi \mathbb{B}^{p}_2 \subset \operatorname{conv}(\tilde\bX_1, \ldots, \tilde\bX_m)$. 

Now consider the probability

\begin{align*}
    \PP(\exists \vb \in \mathbb{S}^{p-1}: \inf_{i \in [n]}\langle \vb,  \tilde\bX_{i} \rangle \geq -\xi ).
\end{align*}
Construct a minimum $\delta$-cover $\cN_\delta$ on $\mathbb{S}^{p-1}$ such that for each $\vb\in\mathbb{S}^{p-1}$, there exists $\vb'\in\cN_\delta$ such that $\|\vb-\vb'\|\leq\delta$, and $\cN_\delta$ contains as few points as possible. %The minimum covering number is bounded by the metric entropy as $\log|\cN_\delta|\leq p \log (1 + 2/\delta).$

% $2p\log(1 + 1/(2\delta^2p)) \wedge \frac{1}{\delta^2}\log(1 + 2p\delta^2)$. (\href{http://www.stat.yale.edu/~yw562/teaching/it-stats.pdf}{http://www.stat.yale.edu/~yw562/teaching/it-stats.pdf} see top of page 96).
If $\exists \vb \in \mathbb{S}^{p-1}: -\xi \leq \langle \vb, \tilde\bX_i \rangle$, then for the closest-to-$\vb$ point $\vb'$ in the $\delta$-cover set $\cN_\delta$ we have
\begin{align*}
%|\langle \vb', \bX_i\rangle| \leq |\langle \vb, \bX_i\rangle| + \|\vb - \vb'\|\|\bX_i\| \leq \xi + \delta\|\bX_i\|.
\langle \vb, \tilde \bX_i \rangle = \langle \vb - \vb', \tilde\bX_i \rangle + \langle \vb', \tilde\bX_i \rangle \leq \langle \vb',\tilde \bX_i \rangle + \delta\|\tilde\bX_i\|.
\end{align*}
Hence it follows that 
\begin{align*}
\PP(\exists \vb \in \mathbb{S}^{p-1}: -\xi \leq \langle \vb, \tilde\bX_i \rangle, \forall i) & \leq \PP(\exists \vb' \in \cN_\delta: \langle \vb', \tilde\bX_i \rangle \geq -\xi - \delta\|\tilde\bX_i\|, \forall i)\\
& \leq |\cN_{\delta}| (\sup_{\vb \in \mathbb{S}^{p-1}}\PP\big( \langle \vb,\tilde \bX \rangle \geq -\xi - \Upsilon\delta\big) + \PP(\|\tilde\bX\| \geq \Upsilon))^{m},
\end{align*}
for any $\Upsilon > 0$. Set $\delta = \xi/\Upsilon$, to obtain
\begin{align*}
\PP(\exists \vb \in \mathbb{S}^{p-1}: -\xi \leq \langle \vb, \tilde \bX_i \rangle \forall i) & \leq (1 + 2\Upsilon/\xi)^p (\sup_{\vb \in \mathbb{S}^{p-1}}\PP\big( \langle \vb, \tilde\bX \rangle \geq -2\xi\big) + \PP(\|\tilde\bX\| \geq \Upsilon))^{m},
% &{\color{red} \leq |\cN_{\delta}| \sup_{\vb \in \mathbb{S}^{p-1}}\PP\big( |\langle \vb, \bX \rangle| \leq \xi + \delta\big)^{|S|},}
\end{align*}
where we used that by a standard volumetric argument we have $|\cN_\delta| \leq (1 + 2/\delta)^p$. Now we observe that by sign symmetry for any $\vb$: $\PP\big( \langle \vb, \tilde\bX \rangle \geq -2\xi\big) = 1/2 + \PP( |\langle \vb, \tilde\bX \rangle| \leq 2\xi)/2$. Hence since $\rho = \sup_{\vb \in \mathbb{S}^{p-1}}\PP\big( |\langle \vb, \tilde\bX \rangle| \leq 2\xi\big)/2 + \PP(\|\tilde\bX\| \geq \Upsilon) < 1/2$ we concude:
\begin{align*}
\PP(\exists \vb \in \mathbb{S}^{p-1}: -\xi \leq \langle \vb, \tilde\bX_i \rangle \forall i) & \leq \bigg(1 + \frac{2c  \EE \|\tilde \bX\|}{\xi}\bigg)^p \big(\frac{1}{2} + \rho\big)^m,
\end{align*}
which is what we wanted to show. %(note that by the ``generalized'' Markov's inequality $\rho' < \rho$). 
\end{proof}

We will now give a simple Corollary to Theorem \ref{lemma:wendel_v3} which is easy to use, as it only relies on certain moment calculations.

\begin{corollary}\label{cor:paley:zygmund} Suppose $\EE \|\tilde \bX\| < \infty$. For a fixed $\theta \in [0,1)$ and $\alpha > 0, q > 1$ define 
\begin{align*}
\rho := c^{-1} +  \frac{1}{2}\bigg(1-\inf_{\vb \in \mathbb{S}^{p-1}}\frac{((1-\theta)\EE |\langle \vb, \tilde\bX \rangle|^{\alpha})^{\frac{q}{q-1}}}{(\EE |\langle \vb, \tilde\bX \rangle|^{q\alpha})^{\frac{1}{q-1}}}\bigg), \mbox{    and   } \xi := (\theta \inf_{\vb \in \mathbb{S}^{p-1}}\EE |\langle \vb, \tilde\bX\rangle|^{\alpha})^{1/\alpha}/2.
\end{align*} 
If $\rho < 1/2$, then \eqref{prob:bound:l2:ball} continues to hold with this choice of $\rho$ and $\xi$.
\end{corollary}

\begin{remark}
In what follows, we will mostly use Corollary \ref{cor:paley:zygmund} over Theorem \ref{lemma:wendel_v3}, and we will be setting $\alpha = 1$ or $2$ and $q = 2$. %For the sake of a clear presentation we will state the two results below. 
\end{remark}

\begin{proof}[Proof of Corollary \ref{cor:paley:zygmund}]
To prove the corollary we note that 
\begin{align*}
\PP\big( |\langle \vb, \tilde\bX \rangle| \leq 2\xi\big) = \PP\big( |\langle \vb, \tilde\bX \rangle|^{\alpha} \leq (2\xi)^{\alpha}\big),
\end{align*}
for any $\alpha > 0$. By the generalized Paley-Zygmund's inequality \citep[see equation (12)][where we instantiate it with $r = 1, s = q$]{petrov2007lower} we have that for any $q > 1$
\begin{align*}
\PP\big( |\langle \vb, \tilde\bX \rangle|^{\alpha} \leq \theta \EE  |\langle \vb, \tilde\bX \rangle|^{\alpha}) \leq 1 - \frac{((1-\theta)\EE |\langle \vb, \tilde\bX \rangle|^{\alpha})^{\frac{q}{q-1}}}{(\EE |\langle \vb, \tilde\bX \rangle|^{q\alpha})^{\frac{1}{q-1}}}.
\end{align*}
It follows that when we set $\xi = (\theta \inf_{\vb \in \mathbb{S}^{p-1}}\EE |\vb\T \tilde\bX|^{\alpha})^{1/\alpha}/2$, \begin{align*}
    \rho \leq \PP(\|\tilde\bX\| \geq \Upsilon) + \frac{1}{2}\bigg(1-\inf_{\vb \in \mathbb{S}^{p-1}}\frac{((1-\theta)\EE |\langle \vb, \tilde\bX \rangle|^{\alpha})^{\frac{q}{q-1}}}{(\EE |\langle \vb, \tilde\bX \rangle|^{q\alpha})^{\frac{1}{q-1}}}\bigg),
\end{align*}
where $\rho$ is as defined in Theorem \ref{lemma:wendel_v3}. This completes the proof after an application of Markov's inequality with $\Upsilon = c \EE\|\tilde\bX\|$ as in the remark after Theorem \ref{lemma:wendel_v3}.
\end{proof}

We will proceed by giving multiple examples applying Theorem \ref{lemma:wendel_v3} and Corollary \ref{cor:paley:zygmund}. We will start with a narrow set of examples which consider popular distributions, and move towards more abstract conditions on the design. We hope to convince the reader that there is a surprising variety of designs which satisfy the condition imposed by Theorem \ref{ball:in:convex:hull:follows:rate}. Below we present only the final results of the application of Theorem \ref{lemma:wendel_v3} and Corollary \ref{cor:paley:zygmund} to the different designs that we consider, and defer the explicit constant calculations to Appendix \ref{appendix:examples}.\\

%Suppose now that condition \ref{strange:condition} holds with $\xi < c$ (for some small enough $c$), which I think holds when $\bX$ is Gaussian. 
\begin{example}\label{Gaussian:example:main:text}
The first application of the above result with $\alpha =1$ and $q = 2$ is for Gaussian design. Suppose $\bX_i \sim N(0, \bSigma)$, where $\bSigma$ has smallest eigenvalue $\lambda_{\min} > 0$. It follows that $\tilde \bX_i = \eta_i \bX_i \sim N(0,\bSigma)$. %Since, $\vb\T \tilde\bX \sim N(0, \vb\T \bSigma \vb)$ we have 
%\begin{align*}
%\EE |\vb\T \tilde\bX| = \sqrt{2/\pi}\sqrt{\vb\T \bSigma \vb} \geq \sqrt{2\lambda_{\min}/\pi}, ~~~ \EE (\vb\T \tilde\bX)^2 = \vb\T \bSigma \vb.
%\end{align*}
% Set $\theta = 1/2$ (so $\xi = \sqrt{\frac{\lambda_{\min}}{8\pi}}$) to obtain $\rho = c^{-1} + (1 - 1/(2\pi))/2$. Set $c = 8\pi$ to obtain $\rho = 1/2 - 1/(8\pi)$. We have
%\begin{align*}
%    \PP(\xi \mathbb{B}^{p}_2 \not \subset \operatorname{conv}(\tilde\bX_1, \ldots, \tilde\bX_m)) & \leq (1 + 32\sqrt{2}\pi^{3/2}\sqrt{\operatorname{tr}(\bSigma)}/\sqrt{\lambda_{\min}})^p (1 - 1/(8\pi))^m \\
%& \leq  (1 + 32\sqrt{2}\pi^{3/2} \sqrt{\operatorname{tr}(\bSigma)}/\sqrt{\lambda_{\min} })^p \exp(-m/(8\pi)),
%\end{align*}
%where we used that by Jensen's inequality $\EE \|\tilde\bX\| \leq \sqrt{\operatorname{tr}(\bSigma)}$. Hence the design contains a sphere of constant radius with probability at least $1 - \gamma$, so long as $m > 8\pi p \log (1 + 32\sqrt{2}\pi^{3/2} \sqrt{\operatorname{tr}(\bSigma)}/\sqrt{\lambda_{\min} }) + 8\pi \log \gamma^{-1}$.
It can be argued using  Theorem \ref{ball:in:convex:hull:follows:rate} that with probability at least $1 -\gamma - \exp(- L^2/(8L/3 + 2))$
\begin{align*}
\|\hat \bbeta - \bbeta^*\| \leq \frac{a (L + 1)(C p \log (1 + C'\sqrt{\operatorname{tr}(\bSigma)}/\sqrt{\lambda_{\min} }) + C \log \gamma^{-1} + 1)}{\xi n},
\end{align*}
where $\xi = \sqrt{\lambda_{\min}/(8\pi)}$ and $C$ and $C'$ are absolute constants. For more details see Appendix \ref{appendix:examples}. We would like to stress the fact that this bound is nearly optimal when $\bSigma = \Ib$ as we show in Theorem \ref{Chebyshev:lower:bound:thm} in the supplement. There we argue that in the isotropic case, with constant probability we have $\|\hat \bbeta - \bbeta^*\| \gtrsim a p/(n (\log n)^{3/2})$. As we discuss later, there exists a different (computationally expensive) estimator which achieves a better dimension dependence in the Gaussian case (for sufficiently large $p$, e.g., $p = n^{\alpha}$), which implies that the Chebyshev estimator is sub-optimal.\\
\end{example}

\begin{example} Our next application includes applying Corollary \ref{cor:paley:zygmund} with $\alpha = 1$ and $q = 2$ to Rademacher design. Let $\bX_{ij}$ be i.i.d. Rademacher random variables. In this example, the first variable can also optionally be an intercept. In any case, it follows that $\tilde \bX_i = \eta_i \bX_{i}$ are Rademacher vectors. It can be shown with the help of Theorem \ref{ball:in:convex:hull:follows:rate} that:
\begin{align*}
\|\hat \bbeta - \bbeta^*\| \leq \frac{a(L + 1)(C (p \log (1 + C'\sqrt{p}) + \log \gamma^{-1}) + 1)}{\xi n},
\end{align*}
with probability at least $1 -\gamma - \exp(- L^2/(8L/3 + 2))$, where $C,C',\xi$ are absolute constants. For the precise constants see Appendix \ref{appendix:examples}.\\
\end{example}

\begin{example} Let $\bX_i$ have a uniform distribution on the unit sphere. Then $\tilde \bX_i \stackrel{d}{=} \bX_i$. Let $\gb$ be a standard Gaussian random vector, and observe that $\tilde \bX_i \stackrel{d}{=} \gb/\|\gb\|$. Using Theorem \ref{ball:in:convex:hull:follows:rate} one can show that
\begin{align*}
\|\hat \bbeta - \bbeta^*\| \leq \frac{a(L + 1)(C(p^{3/2} \log (1 + 2\xi^{-1}C\sqrt{p}) + \log \gamma^{-1}) + 1)}{\xi n},
\end{align*}
with probability at least $1 -\gamma - \exp(- L^2/(8L/3 + 2))$, where $C,\xi$ are absolute constants. For more details see Appendix \ref{appendix:examples}.\\

 \end{example}

\begin{example} In this example we analyze a centered elliptical distribution $\bX$. This generalizes two of our previous examples where we considered Gaussian and uniform on the unit sphere distributions. By a stochastic representation theorem for centered elliptical distributions \cite[see Proposition 4.1.2 of][e.g.]{tong2012multivariate} we know that one can generate a centered elliptical random variable as $\bX \stackrel{d}{=} R \Ab \bU$, where $R \geq 0$ is a non-negative random variable independent of $\bU$, $\bU \stackrel{d}{=} \gb/\|\gb\|$  is distributed uniformly over the unit sphere $\mathbb{S}^{p-1}$, and $\Ab \in \RR^{p \times p}$ is a constant matrix. Suppose $\bSigma = \Ab \Ab\T$ has smallest eigenvalue $\lambda_{\min}$ bounded away from $0$ and largest eigenvalue $\lambda_{\max}$ being bounded. We have $\tilde \bX \stackrel{d}{=} \bX$. In what follows we also assume $\EE R > 0$ and $\EE R^2 < \infty$. \\% For more details see Appendix \ref{appendix:examples}.

By Theorem \ref{ball:in:convex:hull:follows:rate} it can be shown that we have that with probability at least $1 -\gamma - \exp(- L^2/(8L/3 + 2))$:
\begin{align*}
\|\hat \bbeta - \bbeta^*\| \leq \frac{a(L + 1) \bigg(\frac{8 \pi \EE R^2}{(\EE R)^2} \bigg(p \log \bigg(1 + \frac{C' \sqrt{p} (\EE R^2)^{3/2} \sqrt{\lambda_{\max}}}{(\EE R)^{3}\sqrt{\lambda_{\min}}}\bigg) + \log \gamma^{-1}\bigg) + 1\bigg)}{\xi n},
\end{align*}
where $\xi = \EE R\lambda_{\min}^{1/2}\sqrt{\frac{2}{\pi}}/(4\sqrt{p})$. For more details see Appendix \ref{appendix:examples}.\\
%{\color{red} rotationally invariant distr, or spherically symmetric one?}
\end{example}

\begin{example} \label{most:important:example} We now give a general example which only assumes that $\inf_{\vb \in \mathbb{S}^{p-1}}\EE \vb\T \bX\bX\T \vb = \lambda_{\min} > 0$ and $\sup_{\vb\in \mathbb{S}^{p-1}}\EE (\vb\T \bX)^4 \leq C < \infty$. The latter happens in the case when the variables $\bX$ are sub-Gaussian e.g. (in other words we assume that $\EE \exp(t^{-2} (\vb\T \bX)^2) \leq 2$ for some $t \in \RR^+$ for any $\vb \in \mathbb{S}^{p-1}$ (see also Definition \ref{sub-Gaussian:variable:def} in Section \ref{LASSO:section} for a formal definition)). Indeed, this is so by Lemma 5.5 of \cite{Vershynin2012Introduction}. 

Clearly, under these assumptions $\inf_{\vb \in \mathbb{S}^{p-1}}\EE \vb\T \tilde \bX\tilde \bX\T \vb = \lambda_{\min} > 0$ and $\sup_{\vb\in \mathbb{S}^{p-1}}\EE (\vb\T\tilde  \bX)^4 \leq C < \infty$.  By Theorem \ref{ball:in:convex:hull:follows:rate} one can argue that 
\begin{align*}
\|\hat \bbeta - \bbeta^*\| \leq \frac{a(L + 1)\bigg(C'\lambda^{-2}_{\min}\bigg(p \log\bigg(1 + C''\lambda_{\min}^{-5/2}p^{1/2}\bigg) + \log \gamma^{-1}\bigg) + 1\bigg)}{\xi n},
\end{align*}
with probability at least $1 -\gamma - \exp(- L^2/(8L/3 + 2))$, where $C',C''$ are constants that depend on $C$ and $\xi = \lambda^{1/2}_{\min}/(2 \sqrt{2})$.

One can of course assume even less assumptions in which case the bounds will worsen a bit. For instance, instead of assuming $\sup_{\vb\in \mathbb{S}^{p-1}}\EE (\vb\T \bX)^4 \leq C < \infty$ one can simply assume that the coordinates $\bX^{(j)}$ for $j \in [p]$ have bounded $4$-th moments by some constant $C_0$.  Finally, if one is bothered by $4$-th moment assumptions, this too can be relaxed. One needs to use Corollary \ref{cor:paley:zygmund} with $\alpha = 2$ and $q = 1 + \epsilon/2$ (so that $q\alpha = 2 + \epsilon$) for some $\epsilon > 0$. In this way, it suffices to assume that $\sup_{\vb \in \mathbb{S}^{p-1}}\EE |\vb\T \bX|^{2 + \epsilon} < \infty$ which is even weaker than a 4-th moment assumption. For more details see Appendix \ref{appendix:examples}.\\
\end{example}

\begin{example} In our final example we will not impose moment assumptions on the variables (except $\EE \phi(\|\bX\|) < \infty$ for some increasing and positive $\phi$), but we will impose assumptions on the densities of the variables $\vb\T \bX$ for any $\vb \in \mathbb{S}^{p-1}$. To this end we will be applying Theorem \ref{lemma:wendel_v3} directly rather than its corollary. Before we do that we state a lemma. 

\begin{lemma}\label{density:lemma} Suppose that for any $\vb \in \mathbb{S}^{p-1}$ the variables $\vb\T \bX$ have density with respect to the Lebesgue measure (denoted by $f_{\vb}$), and in addition for some $q > 1$ we have $\sup_{\vb \in \mathbb{S}^{p-1}}  (\int f_{\vb}^q(t) dt)^{1/q} \leq C < \infty$. Then for 
\begin{align*}\xi := \frac{1}{2}\bigg[\frac{q}{\pi (q-1) (c_0^{-1}eC)^{\frac{2q}{q-1}}}\bigg]^{\frac{1}{2}},
\end{align*}
we have $\sup_{\vb \in \mathbb{S}^{p-1}}\PP(|\vb\T \bX| \leq 2\xi) \leq c_0.$
\end{lemma}

Under the assumptions of Lemma \ref{density:lemma} with $c_0$, say $c_0 = 1/4$, we can directly apply Theorem \ref{lemma:wendel_v3} with $\rho = \EE \phi(\|\tilde \bX\|)/\phi(\Upsilon) + 1/8 < 1/2$ for $\Upsilon > \phi^{-1}(8/3\EE \phi(\|\tilde \bX\|))$ (notice here that $\sup_{\vb \in \mathbb{S}^{p-1}}\PP(|\vb\T \tilde \bX| \leq 2\xi) = \sup_{\vb \in \mathbb{S}^{p-1}}\PP(|\vb\T \bX| \leq 2\xi)$). Set $\Upsilon = \phi^{-1}(8\EE \phi(\|\tilde \bX\|))$ so that $\rho = 1/4$. Assuming that $\EE \phi(\|\tilde \bX\|) = \EE \phi(\|\bX\|) < \infty$ we have that 
\begin{align*}
    \PP(\xi \mathbb{B}^{p}_2 \not \subset \operatorname{conv}(\tilde\bX_1, \ldots, \tilde\bX_m)) \leq \bigg(1 + \frac{2\Upsilon}{\xi}\bigg)^p \bigg(1 -  \frac{3}{4}\bigg)^m,
\end{align*}
for $\xi$ as in Lemma \ref{density:lemma}.  Hence for $m \geq 4/3(p \log(1 + 2\Upsilon/\xi) + \log \gamma^{-1})$ we have
\begin{align*}
    \PP(\xi \mathbb{B}^{p}_2 \not \subset \operatorname{conv}(\tilde\bX_1, \ldots, \tilde\bX_m)) \leq  \gamma. 
\end{align*}
Using Theorem \ref{ball:in:convex:hull:follows:rate} we can conclude that 
\begin{align*}
\|\hat \bbeta - \bbeta^*\| \leq \frac{a(L + 1)(\frac{4}{3}(p \log(1 + 2\Upsilon/\xi) + \log \gamma^{-1}) + 1)}{\xi n},
\end{align*}
with probability $1 -\gamma - \exp(- L^2/(8L/3 + 2))$. 

We now move on to provide a realistic instance when the assumptions above can be met. Suppose that the covariates $\bX = \bSigma^{\frac{1}{2}} \bZ$, where $\bZ$ is a vector whose entries are independent variables with densities in $L_2 := L_2(\RR)$, such that $\max_{j \in [p]} [\int f^2_{\bZ^{(j)}}(t) dt]^{1/2} < U$ for some fixed $U < \infty$, and $\bSigma^{1/2}$ is a positive semi-definite symmetric matrix whose minimum and maximum eigenvalues $\lambda_{\min}$ and $\lambda_{\max}$ are bounded away from $0$ and $\infty$. Additionally, assume that $\EE \phi(\lambda_{\max} \|\bZ\|) \leq C(p)$ for some constant $C(p)$ which potentially depends on the dimension $p$. 

We will now argue that the densities $f_{\vb}$ of the variables $\vb\T \bX$ for a unit vector $\vb$ exist and are in $L_2$. To this end let $\wb := \vb\T \bSigma^{1/2}$ (for a unit vector $\vb$) and let $\ell$ be the index such that $|\wb^{(\ell)}| = \|\wb\|_{\infty} \geq \|\wb\|_{2}/\sqrt{p} \geq \lambda_{\min}/\sqrt{p}$. Next, we will control the following integral, involving the characteristic function of the variable $\wb\T \bZ = \vb\T \bX$:
\begin{align*}
 \frac{1}{2\pi} \int |\EE e^{i t \wb\T \bZ}|^2 dt & = \frac{1}{2\pi} \int \prod_{j \in [p]}|\EE  e^{i t \wb^{(j)} \bZ^{(j)}}|^2 dt \leq \frac{1}{2\pi} \int |\EE  e^{i t \wb^{(\ell)} \bZ^{(\ell)}}|^2 dt \\
& =  \frac{1}{|\wb^{(\ell)}|2\pi} \int |\EE  e^{i y \bZ^{(\ell)}}|^2 dy = \frac{1}{\|\wb\|_{\infty}}\int f^2_{\bZ^{(\ell)}}(y) dy  < \frac{U^2 \sqrt{p}}{\lambda_{\min}},
\end{align*}
where we applied Plancharel's theorem in the next to last identity. By Lemma 1.1 of \cite{fournier2010absolute}, we know that the above implies that the variable $\vb\T \bX$ has density with respect to the Lebesgue measure. Denote, as in Lemma \ref{density:lemma}, that density with $f_{\vb}$. We will now argue that $f_{\vb}$ is in $L_2$ and satisfies $\int f^2_{\vb}(t) dt < U^2 \sqrt{p}/\lambda_{\min}$. By another application of Plancharel's theorem we have 
\begin{align*}
\int f^2_{\vb}(t) dt =  \frac{1}{2\pi} \int |\EE e^{i t \wb\T \bZ}|^2 dt  <  \frac{U^2 \sqrt{p}}{\lambda_{\min}}.
\end{align*}
It is also easy to verify that $\EE \phi(\|\tilde \bX\|) = \EE \phi(\| \bX\|) \leq \EE \phi(\lambda_{\max } \|\bZ\|) \leq C(p)$.

We end this example with a concrete instance of variables which do not possess moments, yet the above discussion is applicable. Suppose $\bSigma^{\frac{1}{2}} = \Ib$, and $\bX^{(j)} = \bZ^{(j)} \sim Cauchy(0,1)$ for all $j \in [p]$. Clearly $\bZ^{(j)}$ do not even posses a first moment, yet it is easy to see that their densities $f_{\bZ^{(j)}}(t) = 1/(\pi(1 + t^2))$ belong to $L_2$. Coupled with the fact that $\EE \sqrt{\|\bZ\|} \leq \EE \sum_{j \in [p]} |\bZ^{(j)}|^{1/2} = \sqrt{2}p < \infty$ shows that our results can be applied even to Cauchy random variables (with $\phi(x) = \sqrt{x}$). In the last inequalities we used $\sqrt[4]{\sum_{j \in [p]} x_j^2} \leq \sum_{j \in [p]} |x_j|^{1/2}$, and the fact that $\EE \sqrt{|\bZ^{(j)}|} = \sqrt{2}$.
 
%where the second identity holds by the same argument as \eqref{bounding:4th:moment}. 
\end{example}

We will conclude this section with a result for the known $a$ case, which shows that if one fits least squares \eqref{least:squares:under:constraint}, subject to the constraint \eqref{main:optimization:problem}, one attains ``the best of both worlds'' type of behavior, which will at worst have a standard risk of the least squares. We have the following result:

\begin{proposition}
\label{theorem:alternative_bound}
Suppose $\varepsilon_i\sim U([-a,a])$ where $a>0$ is a known constant, and $\bX^{(j)}_i$ has bounded $4$-th moment for each coordinate $j$.
Denote with $\bSigma := \EE \bX\bX\T$. If 
\begin{align}\label{annoying:condition}
p^6\|\bSigma^{-1}\|_{\operatorname{op}}^2/n=o(1),
\end{align} 
for $\hat\bbeta$ obtained via \eqref{least:squares:under:constraint} and \eqref{main:optimization:problem}, with probability at least $1- C^{-2} - o(1)$ we have
\begin{align}\label{prop:alt_bound:guarantee}
    \|\hat\bbeta - \bbeta^*\| \lesssim C \sqrt{\frac{p}{n}}\|\bSigma^{-1}\|_{\operatorname{op}}.
\end{align}
In addition, if $\sup_{\vb \in \mathbb{S}^{p-1}} \EE |\vb\T \bX|^{2 + \alpha} < \infty$ for some $\alpha \in (0,2]$,  and instead of \eqref{annoying:condition} we have $n > C_{\alpha}' p$ for a sufficiently large constant $C_\alpha'$ depending only on $\alpha$, with probability at least $1 - \exp(-p) - C^{-2}$ \eqref{prop:alt_bound:guarantee} continues to hold.
\end{proposition}
\begin{remark} An unsatisfactory artifact of the first half of Proposition \ref{theorem:alternative_bound} is that it requires $p^6\|\bSigma^{-1}\|_{\operatorname{op}}^2/n=o(1)$. This is because of the proof strategy, which aims to lower bound $\lambda_{\min}(n^{-1} \allowbreak\sum_{i \in [n]} \bX_i\bX_i^T)$, under the minimal constraint that $\bX^{(j)}_i$ has bounded $4$-th moment. This is known as the ``hard edge'' problem in random matrix theory \citep[see, e.g.]{rudelson2010non, vershynin2011spectral, mendelson2010empirical}, and to the best of our knowledge there are currently no reasonable bounds available under such general conditions. One example of a general condition that can be used to lower bound the eigenvalue is $\sup_{\vb \in \mathbb{S}^{p-1}} \EE |\vb\T \bX|^{2 + \alpha} < \infty$ as observed by \cite{srivastava2013covariance, yaskov2014lower, koltchinskii2015bounding}. We are using their result in the second part of this proposition to obtain a much better dependence on $n$ and $p$. %Finally, we would like to mention that the same remark is valid for the proof of Theorem \ref{theorem:l2bound_general} which also lower bounds the minimum eigenvalue of a matrix using a similar argument to the first part of Proposition \ref{theorem:alternative_bound} (which relies on Gershgorin's disk theorem). This method is likely sub-optimal under more stringent design assumptions. However, in that case, the distribution in question is much more challenging and one cannot directly assume that a condition like $\sup_{\vb \in \mathbb{S}^{p-1}} \EE |\vb\T \bX|^{2 + \alpha} < \infty$ holds for the data (the reader may recall the variables $\tilde \bX_{\ell^*}$ from \eqref{double:sided:ineq}).
\end{remark}

%%%%%%%%%%%%%%%%
\subsection{Minimax lower bound}\label{minimax:bound:section}
To complement the upper bounds derived in the previous section, we derive a minimax lower bound of the estimation error in a uniform noise setting. The minimax lower bound is derived based on Assouad's Lemma \citep{yu1997assouad}. We add a small extension to this standard method in order to also arrive at bounds in probability and not only in expectation. We do so since throughout the paper we focus on probability bounds, hence this is the more relevant object to us.

\begin{theorem}
\label{theorem:lower_bound}
Suppose $\varepsilon_i\sim U([-a,a])$ where $a>0$ is a constant, and $\bX_i$ is any random design independent of the errors. Let 
\begin{align}\label{minimax:risk:def}
    \cR := \frac{a^2 p}{16(\inf_{\Rb \in \cO} \max_{j \in [p]} \EE_{\Xb} \sum_{i \in [n]} |(\bX_i\T \Rb)_j|)^2},
\end{align}
where $\cO$ is the set of all orthogonal matrices. Then the following inequalities hold:
\begin{align*}
    \inf_{\hat \bbeta} \sup_{\bbeta^* \in \RR^p} \EE_{\bbeta^*} \|\hat \bbeta - \bbeta^*\|^2 \geq \cR,
\end{align*}
and in addition
\begin{align*}
    \inf_{\hat \bbeta} \sup_{\bbeta^* \in \RR^p} \PP_{\bbeta^*} (\|\hat \bbeta - \bbeta^*\| \geq \sqrt{\cR/2}) \geq \frac{1}{2^8}.
\end{align*}
\end{theorem}

We will now look into the specific design we considered in Example \ref{stupid:design:optimality}. \\

\begin{corollary}
Take the random design $\bX_i \sim U(\sqrt{p}\{ \vb_1, \ldots, \vb_p\})$, where $\vb_i, i \in [p]$ denote vectors from any orthonormal basis. Then $\cR$ from \eqref{minimax:risk:def} is 
\begin{align*}
    \cR = \frac{a^2 p^2}{16 n^2}.
\end{align*}
\end{corollary}

\begin{proof}
Take $\Rb = [\vb_1;\ldots;\vb_p]$ so as to rotate the basis to a standard basis, and observe that $\EE |(\bX_{i}\T\Rb)_j| = \sqrt{p} /p = 1/\sqrt{p}$. From here the claim follows.
\end{proof}

The above example, coupled with the results of Example \ref{stupid:design:optimality} and Theorem \ref{ball:in:convex:hull:follows:rate} illustrate that there exist designs under which the Chebyshev estimator is (nearly) optimal. In the most natural case of standard Gaussian design however, Theorem \ref{theorem:lower_bound} yields a lower bound of the order of $a\sqrt{p}/n$, while the Chebyshev estimator has a guarantee of the form $a p \log p/n$ by Example \ref{Gaussian:example:main:text}. In Appendix \ref{otimal:upper:bound:for:gaussian:design:case} we argue that the lower bound is sharp in this case. There exists an estimator (although non-computationally tractable one) whose rate of estimation is upper bounded by $a \sqrt{p}/n$ in the known $a$ case under standard Gaussian design when $p^3  (\log p )^4 \ll n$. On the other hand, Theorem \ref{Chebyshev:lower:bound:thm} in the supplement argues that in the Gaussian design case with isotropic covariance, with at least constant probability, the Chebyshev estimator makes error $\|\hat \bbeta - \bbeta^*\| \gtrsim a p/(n (\log n)^{3/2})$. Moreover, both results extend to the case where the design $\Xb$ consists of i.i.d. centered sub-Gaussian variables with unit variance, which shows that the Chebyshev estimator is sub-optimal in such situations. This fact also shows that, a general analysis of estimators taking values in the set \eqref{main:optimization:problem} is going to produce sub-optimal results in terms of the dimension dependence in the (sub-)Gaussian case (since the Chebyshev estimator also takes values in the set \eqref{main:optimization:problem}). One may wonder what prevents the Chebyshev estimator from being optimal. Our intuition is that it overfits. As the proof of Theorem \ref{Chebyshev:lower:bound:thm} shows, the value of $\hat a$ is much smaller than the true value of $a$ which is indicative of overfitting. Another related reason in addition to overfitting could be that it is not exploiting the knowledge of $a$ properly, and perhaps one can show that the best risk equivariant estimator (which is the centroid of \eqref{main:optimization:problem}) may work optimally, although this appears difficult to prove. One way of proving such a result could be to follow calculations of \cite{ibragimov2013statistical} which provide a general theory for Bayesian estimators (and the best risk equivariant estimator is generalized Bayesian with an improper prior), specifically their Theorem 5.2. That result however, does not track the dimension dependence and we failed to prove an optimal result for the best risk equivariant estimator or for other Bayesian estimators using their method. There are of course many other examples of high-dimensional settings where the MLE fails to be minimax optimal. One such recent example is given by \cite{neykov2022minimax} where it is argued that in general the MLE is suboptimal for the Gaussian sequence model with convex constraint, but there exist different minimax optimal estimators.

%One drawback of the above theorem is the fact that the dimension dependence appears sub-optimal in some of the other examples that we considered. We leave tightening the dimension dependence in this bound as an important direction of future studies. 

The astute reader would notice that in almost all of our upper bounds examples we assumed the quantity $\lambda_{\min}(\EE \bX \bX\T)$ is bounded from below. This quantity does not explicitly appear in our lower bound above. Below we will show a separate lower bound based on the proof of Theorem \ref{theorem:lower_bound}, which illustrates that the quantity $\lambda_{\min}(\EE \bX \bX\T)$ cannot be too small if one wants to attain reasonable bounds on the estimation error $\|\hat \bbeta - \bbeta^*\|$. 

\begin{proposition}\label{modded:lower:bound:theorem}Assume the same setting as in Theorem \ref{theorem:lower_bound}. Let 
\begin{align*}%\label{minimax:risk:def}
    \cR := \frac{a^2}{16 (\inf_{\vb \in \mathbb{S}^{p-1}} \EE |\bX\T \vb|)^2} \geq \frac{a^2}{16 \lambda_{\min}(\EE \bX\bX\T)},
\end{align*}
Then the following inequalities hold:
\begin{align*}
    \inf_{\hat \bbeta} \sup_{\bbeta^* \in \RR^p} \EE_{\bbeta^*} \|\hat \bbeta - \bbeta^*\|^2 \geq \cR,
\end{align*}
and in addition
\begin{align*}
    \inf_{\hat \bbeta} \sup_{\bbeta^* \in \RR^p} \PP_{\bbeta^*} (\|\hat \bbeta - \bbeta^*\| \geq \sqrt{\cR/2}) \geq \frac{1}{2^8 p^2}.
\end{align*}
\end{proposition}

\begin{remark}Our result above is not entirely satisfactory, since it does not capture any dimension dependence. Furthermore, in some examples such as Example \ref{most:important:example} the quantity $\lambda_{\min}(\EE \bX\bX\T)$ appears to the power of $5/2$ in the denominator which is not matched by the lower bound above. The latter can be remedied by imposing a lower bound on $\inf_{\vb \in \mathbb{S}^{p-1}} \EE |\bX\T\vb|$ in place of $\lambda_{\min}(\EE \bX\bX\T)$ in Example \ref{most:important:example}. We do not pursue that further here however.
\end{remark}

\section{The $\ell_1$ penalized $\ell_{\infty}$ estimator (aka Chebyshev's LASSO)}\label{LASSO:section}

Another problem of interest is whether we can extend the $\ell_{\infty}$ estimator \eqref{chebyshev_est} to high-dimensional situations where $\bbeta^* \in \RR^p$ is $s$-sparse. Consider the program
\begin{align}
\min a + \lambda \|\bbeta\|_1\quad\text{ subject to }|Y_i - \bX_i\T \bbeta| \leq a, \forall i\in[n]
\label{chebyshev_lasso_est}
\end{align}

Luckily one need not write new software to solve problem \eqref{chebyshev_lasso_est} as it is a linear program. A similar but dual version of program \eqref{chebyshev_lasso_est} has recently been considered by \citep{du2019optimal}, where it was argued that the $\ell_\infty$ loss function is the most natural loss for a certain problem in fuselage assembly. \cite{du2019optimal} also provide some theoretical guarantees on their version of the program, however they failed to recognize that this program will converge at much faster rates than the usual LASSO in the case of uniform errors. 

The following Theorem \ref{theorem:l2_bound_chebyshev_lasso} shows that under some conditions on the design matrix $\Xb$ and the growth rate of the sparsity $s$ and the ambient dimension $p$ with respect to the sample size $n$, the estimator obtained via \eqref{chebyshev_lasso_est} achieves a rate faster than the LASSO estimation rate $s \sqrt{\log p/n}$  (for the $\ell_1$ norm) \citep[see Chapter 7]{wainwright2019high}. Before presenting the theorem we need to introduce the Restricted Eigenvalue (RE) condition \citep{bickel2009simultaneous}, which is the least restrictive eigenvalue condition imposed on the population covariance matrix in order to provide good convergence guarantees for $\ell_1$-based methods. First, let us define a set $\cC(S,\gamma)$ which is relevant to the RE condition.
\begin{definition}
\label{re_pre}
For a given subset $S \subset [p]$ and a constant $\gamma\geq1$, the set $\mathcal{C}(S, \gamma)$ is defined as
\begin{align*}
    \mathcal{C}(S, \gamma) := \{\vb\in\mathbb{R}^p: \|\vb_{S^c}\|_1 \leq \gamma \|\vb_{S}\|_1 \}.
\end{align*}
\end{definition}
Next, the Restricted Eigenvalue condition of order $s$ with parameters $\kappa, \gamma$ is denoted as $RE(\kappa,\gamma,s)$ and defined as following.
\begin{definition}
\label{re_condition}
For a constant $\kappa>0$, we say that a symmetric matrix $\Ab$ satisfies the condition $RE(\kappa,\gamma,s)$ if 
\begin{align*}
    \vb\T\Ab\vb \geq \kappa^2 \|\vb\|^2 \quad \text{for all } \vb \in \cC(S,\gamma),
\end{align*}
holds uniformly for all sets $S$ with cardinality $s$.
\end{definition}

Before we state the main result of this section, we will also formally introduce sub-Gaussian and isotropic random variables.
\begin{definition}[Sub-Gaussian and Isotropic Random Vectors]\label{sub-Gaussian:variable:def}
A random vector $\bzeta \in \RR^p$ is called isotropic if $\EE \bzeta\bzeta\T = \Ib$. A random vector $\bzeta \in \RR^p$ is called $\gamma$ sub-Gaussian  if for any $\vb \in \mathbb{S}^{p-1}$
\begin{align*}
\inf\{t: \EE \exp(t^{-2} (\vb\T \bzeta)^2) \leq 2\} \leq \gamma.
\end{align*}

\end{definition}

\begin{theorem}
\label{theorem:l2_bound_chebyshev_lasso}
Suppose $\bX = \bSigma^{\frac{1}{2}}\bzeta$ where $\bzeta\in\mathbb{R}^p$ is an isotropic $\gamma$ sub-Gaussian vector. Let the predictors $\bX_i \sim \cL(\bX)$ be i.i.d., where $\cL(\bX)$ denotes the law of the random varaible $\bX$. Additionally we assume that the Gram matrix $\bSigma = \EE \bX \bX\T$ satisfies the $RE(\kappa,2,s)$ condition for a constant $\kappa>0$, $\|\bSigma\|_{\operatorname{op}}$ is bounded from above and $\bSigma_{jj} = 1$ for all $j \in [p]$. %Suppose also $\bSigma^{1/2}$ satisfies $\sup_{\vb: \|\vb\| = 1, \|\vb\|_0 \leq s}\|\bSigma^{1/2} \vb\|$ is upper bounded by a constant, where $\|\vb\|_0$ denotes the number of non-zero entries of $\vb$. 
If $s\leq p/2$, and $s ((\log (5ep/s) \vee \log p) \vee \log p (\log n)^2)) \leq p$ then for $\lambda = \frac{\kappa}{(4 + \epsilon)\sqrt{s}\log n}$ for any small $\epsilon > 0$, we have
\begin{align*}
    \|\hat\bbeta-\bbeta^*\|_1 \lesssim_{\gamma, \|\bSigma\|_{\operatorname{op}}, \kappa} \frac{s^{3/2} ((\log (5ep/s) \vee \log p) \vee \log p (\log n)^2))}{n}. %\leq \max\bigg(\frac{6 ( L + 1)a}{\lambda n^{1-\alpha}}, \frac{24 (L + 1)a\sqrt{s}\log n}{\kappa n^{1- \alpha}}\bigg),
\end{align*}
with probability converging to $1$, where $\lesssim_{\gamma, \|\bSigma\|_{\operatorname{op}}, \kappa}$ hides constants depending on $\gamma, \|\bSigma\|_{\operatorname{op}}, \kappa$.%, where $L$ is some sufficiently large absolute constant.
\end{theorem}

\begin{remark} The above theorem shows that Chebyshev's LASSO can be (much) more accurate than the regular LASSO under certain assumptions. Importantly, note that the optimal choice of $\lambda$ does not seem to depend on the parameter $a$, which will be affecting the LASSO tuning parameter (since the variance of a uniform distribution is $\frac{a^2}{3}$). However, it does depend on the (potentially) unknown sparsity $s$, and hence in practice some tuning will be required. This can be done with cross-validation, e.g. We do not provide a tight upper bound on the $\ell_2$ norm, but it should be clear that the $\|\hat \bbeta - \bbeta^* \| \leq \|\hat \bbeta - \bbeta^* \|_1$ and hence the above bound is valid in terms of the $\ell_2$ norm too. In addition, we would like to mention that the factor $\log n$ that appears in the upper bound on $\|\hat \bbeta - \bbeta^* \|_1$ and in the definition of $\lambda$ may be replaced by any slowly diverging sequence in $n$. Finally we give some intuition on why we obtain $s^{3/2}$ in the upper bound for $\|\hat \bbeta - \bbeta^* \|_1$. Recall that (dropping log factors) the rate of the Chebyshev estimator for Gaussian design is $p/n$, and we obtain $s^{3/2}/n$ bound for the $\ell_1$ norm in the high-dimensional setting (i.e. it is $\sqrt{s}$ more than the bound $s/n$). This is similar to how in the Gaussian design case in regression with Gaussian errors the upper bound under $\ell_2$ loss is $\sqrt{p/n}$ but the LASSO obtains a rate under $\ell_1$ loss equal to $s/\sqrt{n}$ (dropping log factors), i.e., we multiply by $\sqrt{s}$. Intuitively, this comes from the bound $\|\vb\|_1 \leq \sqrt{s} \|\vb\|$ for any $\vb$ with $\|\vb\|_0 \leq s$ where $\|\vb\|_0$ is the number of non-zero entries in $\vb$.
%\color{red} expand this remark}
%).

%{\color{red} this simply sounds fishy... need to recheck the proof.}
%Finally, as is evident upon close inspection of the proof, the $\log n$ terms can be replaced with any slowly diverging sequence with $n$ (such as $\log \log n$ e.g.). This change will affect the rate only slightly, but may have more practical implications in the condition $\kappa \sqrt{n^{\alpha}/\log p} \gg \gamma^{1/2} \|\bSigma\|^{1/4}_{\operatorname{op}}\sqrt{s}\log n$, since $\log n$ tends to be ``large'' in comparison to $n^\alpha$ for small values of $n$. 

%{\color{red} have to comment here}
\end{remark}

\section{Simulations}\label{sim:section}

In this section we provide several brief numerical experiments in support of our findings.

\subsection{Chebyshev estimator}

We begin with the Chebyshev estimator. We use three designs to construct our experiments --- standard Gaussian design, Rademacher design and uniform on the unit sphere design. We remind the reader that these three designs were considered as examples after Theorem \ref{lemma:wendel_v3}, and we know from our theorems that for the first two designs $\|\hat \bbeta - \bbeta^*\| \lesssim (p\log p)/n$ while for the last design we have $\|\hat \bbeta - \bbeta^*\| \lesssim (p\sqrt{p}\log p)/n$. We constructed datasets of multiple sizes, one for each pair $(n,p)$ where $n \in \{30,40,50,60,70,80,90,100,110\}$ and $p \in \{4,8,12,16,20\}$. Here we set $\bbeta^*$ to have its first $p/2$ entries equal to $1$ and the remaining entries equal to $-1$, while $a = 2$. For each dataset we computed the Chebyshev estimator $100$ times and averaged  $\|\hat \bbeta - \bbeta^*\|$. We then plotted these results against $p/n$ and $(p\sqrt{p})/n$ since we believe the extraneous $\log p$ factors  that we obtained are artifacts of the proof. Figure \ref{figure:one} illustrates our findings. We see a near perfect linear alignment. This empirical evidence suggests that our simple analysis is nearly tight for those designs. This is also corroborated by Theorem \ref{Chebyshev:lower:bound:thm} in the supplementary material.

\begin{figure}[t!]
\includegraphics[scale=.4]{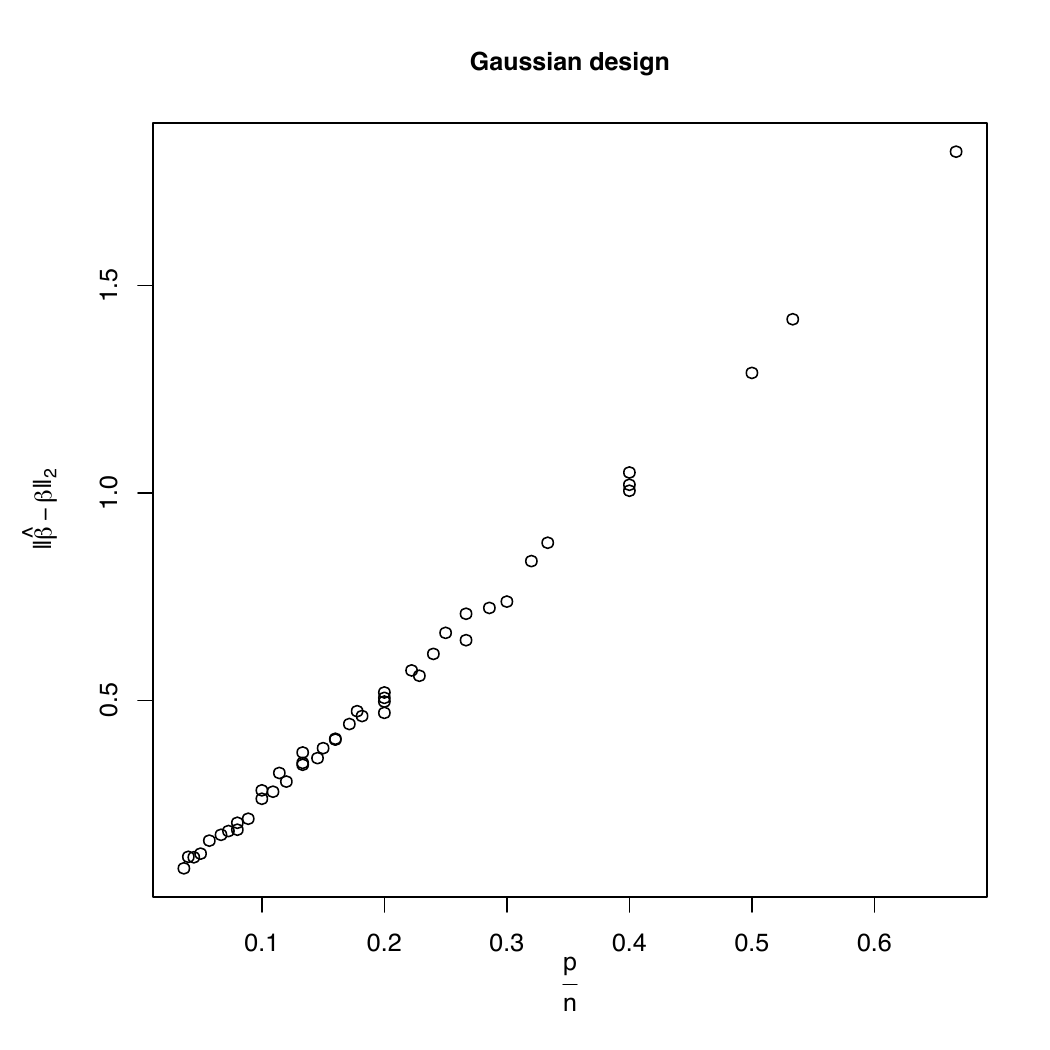}\hfill
\includegraphics[scale=.4]{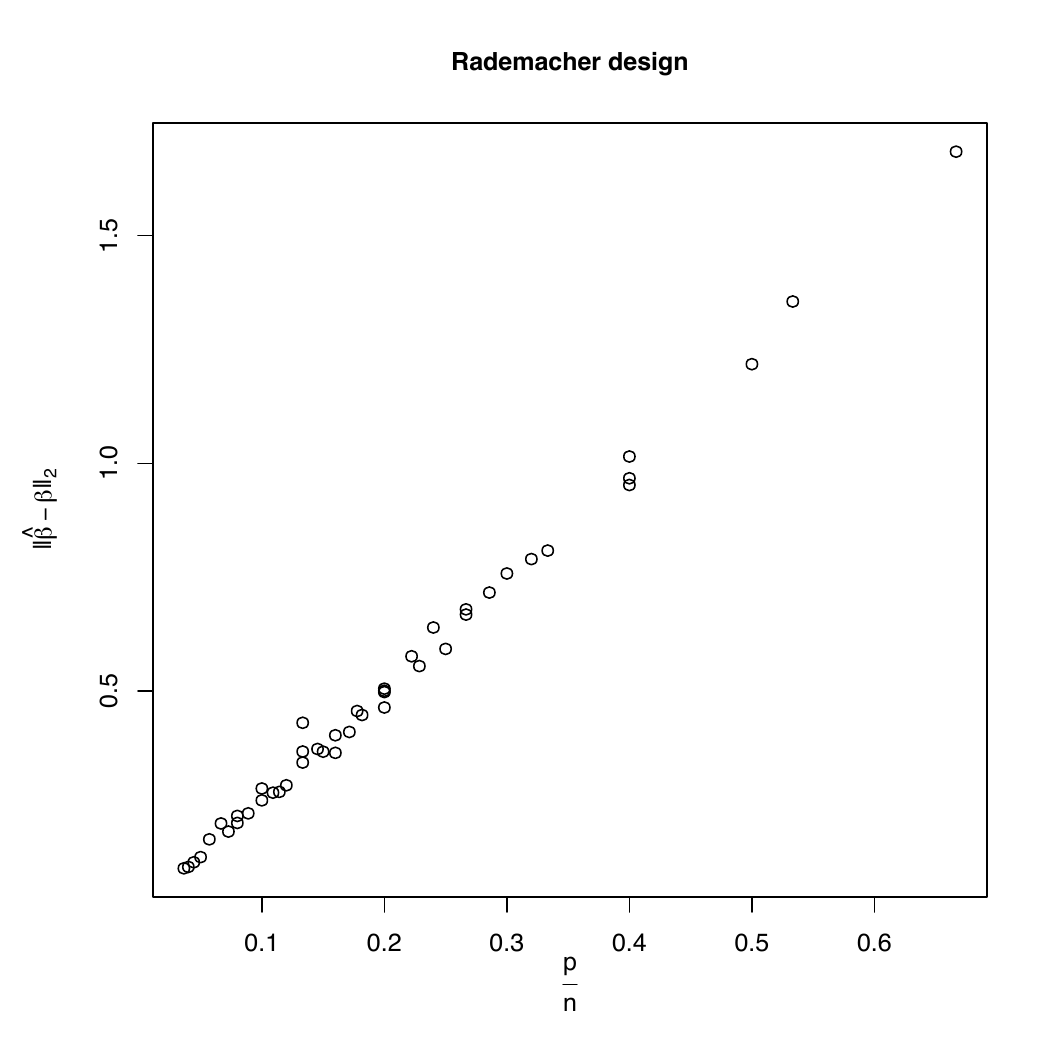}\hfill
\centering\includegraphics[scale=.4]{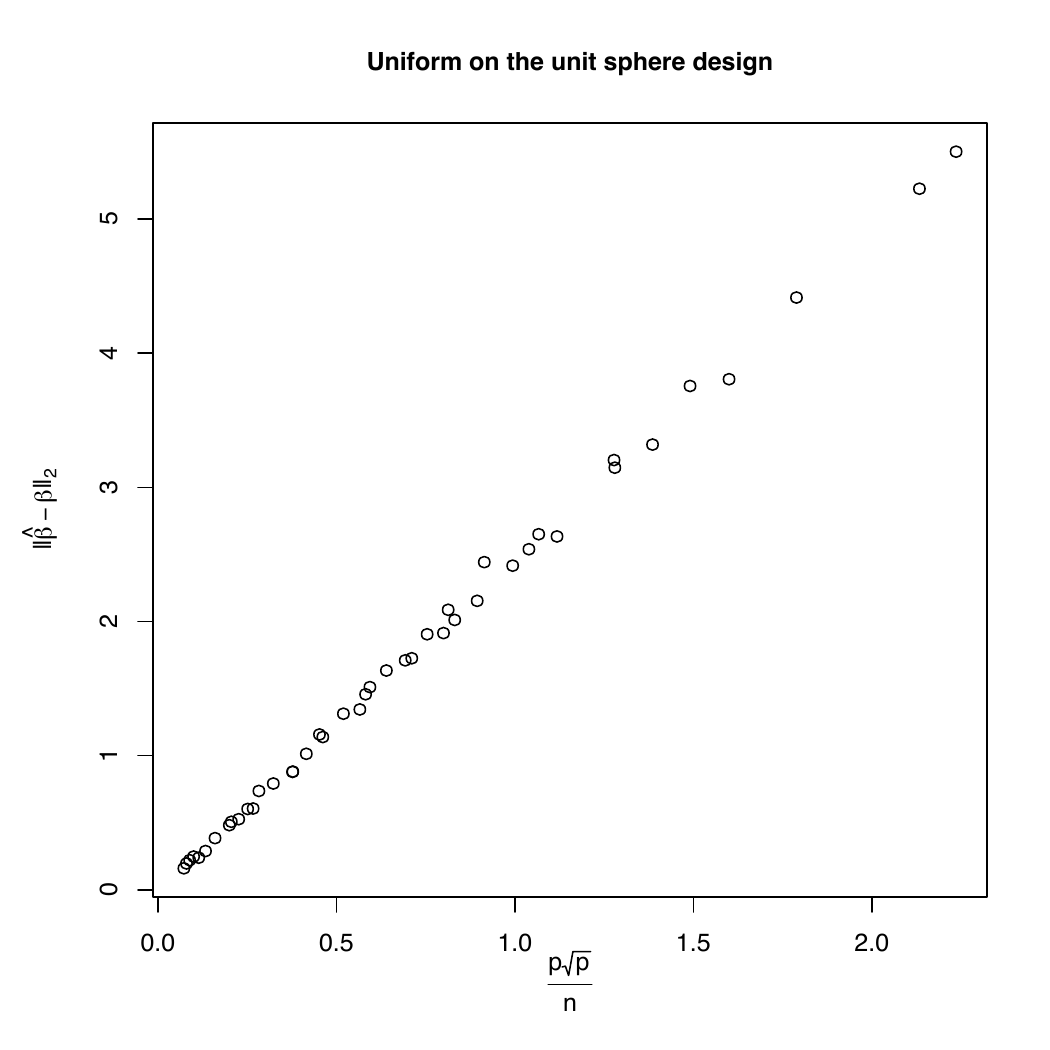}
\caption{From left to right: Gaussian design, Rademacher design, and uniform on the unit sphere design. We observe near perfect linear patterns which suggests that our analysis is nearly tight. }\label{figure:one}
\end{figure}

\subsection{Chebyshev's LASSO}

In order to illustrate the superiority of Chebyshev's LASSO vs the regular LASSO, we constructed examples where the $\bbeta^*$ vector is very sparse in comparison to the sample size. This is in order to make the requirement  $\sqrt{s\log p}\log n \ll \sqrt{p}$ hold at least approximately. We considered two possible sample sizes $n = 600, 800$, two possible values of the sparsity of $\bbeta^*$: $s = 4, 10$ and the ambient dimension is $p = 1000 + s$. Here $\bbeta^*$ has its first $s/2$ entries equal to $1$, the next $s/2$ entries equal to $-1$ and all remaining entries are $0$. We also set $a = 5$ throughout the simulations. We tuned both Chebyshev's LASSO and the regular LASSO. For the tuning of Chebyshev's LASSO we considered six equispaced values in the range $[.1\sqrt{\log p/n^{.4}}, 2\sqrt{\log p/n^{.4}}]$, and after we run the analysis we pick the value  $\hat \bbeta$ which is closest to the true $\bbeta^*$ in terms of the $\ell_1$ norm. Similarly, to tune the LASSO, we considered the default $\lambda$ values given by the \verb|cv.glmnet| function of the \verb|glmnet| package in \verb|R| (which are calculated from the data and are around 100), and used the one that gives the closest $\hat \bbeta$ to the true $\bbeta^*$ in $\ell_1$ norm. It is evident from the results of Table \ref{table:one} that Chebyshev's LASSO dominates the LASSO in all $4$ settings considered. %{\color{red}$a = 5$}
%
%Chebyshev's LASSO
%\begin{center}
%\begin{table}
%\begin{tabular}{ ccc } 
%% \hline
%  n,s & 4 & 10 \\ 
% 600 & 1.06 & 3.4 \\ 
%800 & 0.83 & 2.4 \\ 
%% \hline
%\end{tabular}
%\caption{Chebyshev's LASSO}
%\end{table}
%%\end{center}
% \quad
% \begin{table}
%%\begin{center}
%\begin{tabular}{ccc } 
%% \hline
%  n,s & 4 & 10 \\ 
% 600 & 1.53 & 3.7 \\ 
%800 & 1.32 & 3.11 \\ 
%% \hline
%\end{tabular}
%\caption{Chebyshev's LASSO}
%\end{table}
%\end{center}

\begin{table*}%[H]
           \centering
           \captionsetup[subtable]{position = below}
          \captionsetup[table]{position=top}
           \caption{Summary of simulation results}\label{table:one}
           \begin{subtable}{0.3\linewidth}
               \centering
               \begin{tabular}{|c|cc|}
               \hline
		  \diaghead(-1,1){aa}{n}{s}& 4 & 10 \\ 
		\hline
		 600 & 1.06 & 3.4 \\ 
		800 & 0.83 & 2.4 \\ 
		\hline
               \end{tabular}
               \caption{Chebyshev's LASSO $\|\hat \bbeta - \bbeta^*\|_1$ averaged over $100$ simulations.}
               \label{tab:dimFFT}
           \end{subtable}%
           \hspace*{4em}
           \begin{subtable}{0.3\linewidth}
               \centering
               \begin{tabular}{|c|cc|}
                   \hline
		  \diaghead(-1,1){aa}{n}{s}& 4 & 10 \\ 
		  \hline
 		600 & 1.53 & 3.7 \\ 
		800 & 1.32 & 3.11 \\ 
		\hline
               \end{tabular}
                \caption{LASSO $\|\hat \bbeta - \bbeta^*\|_1$ averaged over $100$ simulations.}
                 \label{tab:dimGMM}
           \end{subtable}
       \end{table*}

\section{Discussion}\label{discussion:section}

In this paper we presented some non-asymptotic bounds on the Chebyshev estimator in linear regression with uniform errors.  In addition we demonstrated that under certain assumptions Chebyshev's LASSO can strictly dominate the usual LASSO. As we remarked our approach is immediately extendible to symmetric bounded noise, and with a little more effort can be extended to asymmetric noise as well. There are a lot of interesting open questions. Unlike the asymptotic approach in \cite{knight2017asymptotic}, our analysis does not rest on the epi-convergence techniques; however, it is interesting whether such epi-convergence techniques could be directly turned into finite sample results. %It is also very interesting whether one can improve the dimension dependent constants we provided in our examples after Theorem \ref{lemma:wendel_v3}. 

Next, we discuss the lower bounds. As it stands, our Theorem \ref{theorem:lower_bound} does not produce a matching lower bound to the bound in Theorem \ref{ball:in:convex:hull:follows:rate} under standard Gaussian design, e.g. On the other hand in Appendix \ref{otimal:upper:bound:for:gaussian:design:case} we establish that the lower bound is sharp in the i.i.d. standard (sub-)Gaussian design case, at least when $p$ is not too large compared to $n$. Hence a question arises: is the Chebyshev estimator truly sub-optimal or the gap is this sub-optimality introduced by our imprecise analysis? This question is closed by Theorem \ref{Chebyshev:lower:bound:thm} which argues that the dimension dependence we obtain for the Chebyshev estimator with i.i.d. (sub-)Guassian design is optimal (up to logarithmic factors). This illustrates the interesting phenomenon that the Chebyshev estimator (which is the MLE in the unknown $a$ case and is an MLE in the known $a$ case) is provably sub-optimal in terms of the dimension dependence. One explanation of this is that it overfits, and in addition it does not utilize the knowledge of the constant $a$ whereas the optimal estimator estimator we develop in Theorem \ref{optimal:estimator:thm} relies on $a$ being known. An open question is whether one can improve the lower bound Theorem \ref{theorem:lower_bound} to capture the unknown $a$ case.

%Hence an exciting open question is to prove a lower bound for the performance of the Chebyshev estimator. This will show that the Chebyshev estimator is provably sub-optimal for Gaussian design. 

There are also a multitude of questions left in for the high-dimensional Chebyshev estimator. First it is not clear whether the rate that Theorem \ref{theorem:l2_bound_chebyshev_lasso} is optimal. In fact is likely suboptimal given the sub-optimality of the Chebyshev estimator in low dimensional situations. Second, deriving matching lower and upper bounds sounds like a challenging but interesting question for future research. 
%This opens the question of whether one can improve the lower bound of Theorem \ref{theorem:lower_bound}, or if not, do there exist better estimators than the Chebyshev estimator for such a design. 

Finally, if one is interested in inference, a possible approach that works for some non-regular models was recently proposed by \cite{wasserman2020universal}. Unfortunately, this approach has problems with models with uniform distributions (see the uniform distribution example before section 4 \citep{wasserman2020universal}), but there may exist smart ways of tweaking it to make it work in our setting. We defer this to future work.

%
%the lower bound for gaussian design.
%
%lower bound for the Chebyshev estimator or maybe any equivariant estimator?
%
%inference with ``universal inference'''. 

\section{Acknowledgements}

The authors would like to thank Sivaraman Balakrishnan for inspiring discussions on the topic, and in particular the lower bounds and his advice on the presentation of this work. The second author is also indebted to Alexandre Tsybakov and Tony Cai for communicating to him their belief that the lower bound is tight, and one should try to improve the upper bound in the Gaussian case. Finally the authors would like to express their gratitude to the AE and two anonymous referees for their insightful suggestions which led to substantial improvements of the manuscript.

\bibliographystyle{imsart-nameyear.bst}
\bibliography{yufeibib}

\begin{thebibliography}{50}
% BibTex style file: imsart-nameyear.bst, 2017-11-03
% Default style options (sort=1,type=nameyear).
% Used options (sort=1,type=nameyear).

\bibitem[\protect\citeauthoryear{Ak{\c{c}}ay, Hjalmarsson and
  Ljung}{1996}]{akccay1996choice}
\begin{barticle}[author]
\bauthor{\bsnm{Ak{\c{c}}ay},~\bfnm{H{\"u}seyin}\binits{H.}},
  \bauthor{\bsnm{Hjalmarsson},~\bfnm{H{\aa}kan}\binits{H.}} \AND
  \bauthor{\bsnm{Ljung},~\bfnm{Lennart}\binits{L.}}
(\byear{1996}).
\btitle{On the choice of norms in system identification}.
\bjournal{IEEE Trans. Automat. Contr.}
\bvolume{41}
\bpages{1367--1372}.
\end{barticle}
\endbibitem

\bibitem[\protect\citeauthoryear{Alecu et~al.}{2006}]{alecu2006wavelet}
\begin{barticle}[author]
\bauthor{\bsnm{Alecu},~\bfnm{Alin}\binits{A.}},
  \bauthor{\bsnm{Munteanu},~\bfnm{Adrian}\binits{A.}},
  \bauthor{\bsnm{Cornelis},~\bfnm{Jan~PH}\binits{J.~P.}} \AND
  \bauthor{\bsnm{Schelkens},~\bfnm{Peter}\binits{P.}}
(\byear{2006}).
\btitle{Wavelet-based scalable L-infinity-oriented compression}.
\bjournal{IEEE Trans. Image Process.}
\bvolume{15}
\bpages{2499--2512}.
\end{barticle}
\endbibitem

\bibitem[\protect\citeauthoryear{Appa and Smith}{1973}]{appa19731}
\begin{barticle}[author]
\bauthor{\bsnm{Appa},~\bfnm{Gautam}\binits{G.}} \AND
  \bauthor{\bsnm{Smith},~\bfnm{Cyril}\binits{C.}}
(\byear{1973}).
\btitle{On L 1 and Chebyshev estimation}.
\bjournal{Math. Program.}
\bvolume{5}
\bpages{73--87}.
\end{barticle}
\endbibitem

\bibitem[\protect\citeauthoryear{Armstrong and Kung}{1980}]{armstrong1980dual}
\begin{barticle}[author]
\bauthor{\bsnm{Armstrong},~\bfnm{Ronald~D}\binits{R.~D.}} \AND
  \bauthor{\bsnm{Kung},~\bfnm{David~S}\binits{D.~S.}}
(\byear{1980}).
\btitle{A dual method for discrete Chebychev curve fitting}.
\bjournal{Math. Program.}
\bvolume{19}
\bpages{186--199}.
\end{barticle}
\endbibitem

\bibitem[\protect\citeauthoryear{Beck and Eldar}{2007}]{beck2007regularization}
\begin{barticle}[author]
\bauthor{\bsnm{Beck},~\bfnm{Amir}\binits{A.}} \AND
  \bauthor{\bsnm{Eldar},~\bfnm{Yonina~C}\binits{Y.~C.}}
(\byear{2007}).
\btitle{Regularization in regression with bounded noise: A Chebyshev center
  approach}.
\bjournal{SIAM J. Matrix Anal. Appl.}
\bvolume{29}
\bpages{606--625}.
\end{barticle}
\endbibitem

\bibitem[\protect\citeauthoryear{Berenguer-Rico, Johansen and
  Nielsen}{2019}]{berenguer2019models}
\begin{barticle}[author]
\bauthor{\bsnm{Berenguer-Rico},~\bfnm{Vanessa}\binits{V.}},
  \bauthor{\bsnm{Johansen},~\bfnm{S{\o}ren}\binits{S.}} \AND
  \bauthor{\bsnm{Nielsen},~\bfnm{Bent}\binits{B.}}
(\byear{2019}).
\btitle{Models where the Least Trimmed Squares and Least Median of Squares
  estimators are maximum likelihood}.
\bjournal{Available at SSRN 3455870}.
\end{barticle}
\endbibitem

\bibitem[\protect\citeauthoryear{Bertsch, Sabbey and
  Uusn{\"a}kki}{2005}]{bertsch2005fitting}
\begin{barticle}[author]
\bauthor{\bsnm{Bertsch},~\bfnm{George~F}\binits{G.~F.}},
  \bauthor{\bsnm{Sabbey},~\bfnm{B}\binits{B.}} \AND
  \bauthor{\bsnm{Uusn{\"a}kki},~\bfnm{M}\binits{M.}}
(\byear{2005}).
\btitle{Fitting theories of nuclear binding energies}.
\bjournal{Phys. Rev. C}
\bvolume{71}
\bpages{054311}.
\end{barticle}
\endbibitem

\bibitem[\protect\citeauthoryear{Bickel et~al.}{2009}]{bickel2009simultaneous}
\begin{barticle}[author]
\bauthor{\bsnm{Bickel},~\bfnm{Peter~J}\binits{P.~J.}},
  \bauthor{\bsnm{Ritov},~\bfnm{Ya'acov}\binits{Y.}},
  \bauthor{\bsnm{Tsybakov},~\bfnm{Alexandre~B}\binits{A.~B.}} \betal{et~al.}
(\byear{2009}).
\btitle{Simultaneous analysis of Lasso and Dantzig selector}.
\bjournal{Ann. Stat.}
\bvolume{37}
\bpages{1705--1732}.
\end{barticle}
\endbibitem

\bibitem[\protect\citeauthoryear{Brenner}{2002}]{brenner2002aeroservoelastic}
\begin{barticle}[author]
\bauthor{\bsnm{Brenner},~\bfnm{Martin~J}\binits{M.~J.}}
(\byear{2002}).
\btitle{Aeroservoelastic model uncertainty bound estimation from flight data}.
\bjournal{J Guid Control Dyn}
\bvolume{25}
\bpages{748--754}.
\end{barticle}
\endbibitem

\bibitem[\protect\citeauthoryear{Castillo et~al.}{2009}]{castillo2009combined}
\begin{barticle}[author]
\bauthor{\bsnm{Castillo},~\bfnm{Enrique}\binits{E.}},
  \bauthor{\bsnm{Castillo},~\bfnm{Carmen}\binits{C.}},
  \bauthor{\bsnm{Hadi},~\bfnm{Ali~S}\binits{A.~S.}} \AND
  \bauthor{\bsnm{Sarabia},~\bfnm{Jos{\'e}~M}\binits{J.~M.}}
(\byear{2009}).
\btitle{Combined regression models}.
\bjournal{Comput. Stat.}
\bvolume{24}
\bpages{37--66}.
\end{barticle}
\endbibitem

\bibitem[\protect\citeauthoryear{Cline}{1972}]{cline1972rate}
\begin{barticle}[author]
\bauthor{\bsnm{Cline},~\bfnm{AK}\binits{A.}}
(\byear{1972}).
\btitle{Rate of convergence of Lawson's algorithm}.
\bjournal{Math. Comput.}
\bvolume{26}
\bpages{167--176}.
\end{barticle}
\endbibitem

\bibitem[\protect\citeauthoryear{Du et~al.}{2019}]{du2019optimal}
\begin{barticle}[author]
\bauthor{\bsnm{Du},~\bfnm{Juan}\binits{J.}},
  \bauthor{\bsnm{Cao},~\bfnm{Shanshan}\binits{S.}},
  \bauthor{\bsnm{Hunt},~\bfnm{Jeffrey~H}\binits{J.~H.}} \AND
  \bauthor{\bsnm{Huo},~\bfnm{Xiaoming}\binits{X.}}
(\byear{2019}).
\btitle{Optimal Shape Control via {$L_{\infty}$} Loss for Composite Fuselage
  Assembly}.
\bjournal{arXiv preprint arXiv:1911.03592}.
\end{barticle}
\endbibitem

\bibitem[\protect\citeauthoryear{Fournier and
  Printems}{2010}]{fournier2010absolute}
\begin{barticle}[author]
\bauthor{\bsnm{Fournier},~\bfnm{Nicolas}\binits{N.}} \AND
  \bauthor{\bsnm{Printems},~\bfnm{Jacques}\binits{J.}}
(\byear{2010}).
\btitle{Absolute continuity for some one-dimensional processes}.
\bjournal{Bernoulli}
\bvolume{16}
\bpages{343--360}.
\end{barticle}
\endbibitem

\bibitem[\protect\citeauthoryear{Gu{\'e}don et~al.}{2022}]{guedon2022geometry}
\begin{barticle}[author]
\bauthor{\bsnm{Gu{\'e}don},~\bfnm{Olivier}\binits{O.}},
  \bauthor{\bsnm{Krahmer},~\bfnm{Felix}\binits{F.}},
  \bauthor{\bsnm{K{\"u}mmerle},~\bfnm{Christian}\binits{C.}},
  \bauthor{\bsnm{Mendelson},~\bfnm{Shahar}\binits{S.}} \AND
  \bauthor{\bsnm{Rauhut},~\bfnm{Holger}\binits{H.}}
(\byear{2022}).
\btitle{On the geometry of polytopes generated by heavy-tailed random vectors}.
\bjournal{Commun. Contemp. Math.}
\bvolume{24}
\bpages{2150056}.
\end{barticle}
\endbibitem

\bibitem[\protect\citeauthoryear{Haagerup et~al.}{1978}]{haagerup1978best}
\begin{barticle}[author]
\bauthor{\bsnm{Haagerup},~\bfnm{Uffe}\binits{U.}} \betal{et~al.}
(\byear{1978}).
\btitle{The best constants in the Khintchine inequality}.
\bjournal{Stud. Math.}
\bvolume{70}
\bpages{231-283}.
\end{barticle}
\endbibitem

\bibitem[\protect\citeauthoryear{Hand and Sposito}{1980}]{hand1980using}
\begin{barticle}[author]
\bauthor{\bsnm{Hand},~\bfnm{ML}\binits{M.}} \AND
  \bauthor{\bsnm{Sposito},~\bfnm{VA}\binits{V.}}
(\byear{1980}).
\btitle{Using the least squares estimator in Chebyshev estimation: Using the
  least squares estimator in Chebyshev estimation}.
\bjournal{Commun. Stat. - Simul. Comput.}
\bvolume{9}
\bpages{43--49}.
\end{barticle}
\endbibitem

\bibitem[\protect\citeauthoryear{Ibragimov and
  Has'~Minskii}{2013}]{ibragimov2013statistical}
\begin{bbook}[author]
\bauthor{\bsnm{Ibragimov},~\bfnm{Ildar~Abdulovich}\binits{I.~A.}} \AND
  \bauthor{\bsnm{Has'~Minskii},~\bfnm{Rafail~Zalmanovich}\binits{R.~Z.}}
(\byear{2013}).
\btitle{Statistical estimation: asymptotic theory}
\bvolume{16}.
\bpublisher{Springer Science \& Business Media}.
\end{bbook}
\endbibitem

\bibitem[\protect\citeauthoryear{James}{1983a}]{james1983fitting}
\begin{barticle}[author]
\bauthor{\bsnm{James},~\bfnm{F}\binits{F.}}
(\byear{1983}a).
\btitle{Fitting tracks in wire chambers using the Chebyshev norm instead of
  least squares}.
\bjournal{Nucl. Instrum. Methods Phys. Res.}
\bvolume{211}
\bpages{145--152}.
\end{barticle}
\endbibitem

\bibitem[\protect\citeauthoryear{James}{1983b}]{james1983probability}
\begin{bincollection}[author]
\bauthor{\bsnm{James},~\bfnm{F}\binits{F.}}
(\byear{1983}b).
\btitle{Probability, statistics, and associated computing techniques}.
In \bbooktitle{Techniques and concepts of high-energy physics II}
\bpages{189--231}.
\bpublisher{Springer}.
\end{bincollection}
\endbibitem

\bibitem[\protect\citeauthoryear{Jaschke}{1997}]{jaschke1997arbitrage}
\begin{barticle}[author]
\bauthor{\bsnm{Jaschke},~\bfnm{Stefan~R}\binits{S.~R.}}
(\byear{1997}).
\btitle{Arbitrage bounds for the term structure of interest rates}.
\bjournal{Finance Stoch.}
\bvolume{2}
\bpages{29--40}.
\end{barticle}
\endbibitem

\bibitem[\protect\citeauthoryear{Jaschke and
  K{\"u}chler}{2001}]{jaschke2001coherent}
\begin{barticle}[author]
\bauthor{\bsnm{Jaschke},~\bfnm{Stefan}\binits{S.}} \AND
  \bauthor{\bsnm{K{\"u}chler},~\bfnm{Uwe}\binits{U.}}
(\byear{2001}).
\btitle{Coherent risk measures and good-deal bounds}.
\bjournal{Finance Stoch.}
\bvolume{5}
\bpages{181--200}.
\end{barticle}
\endbibitem

\bibitem[\protect\citeauthoryear{Jureckov{\'a} and
  Picek}{2009}]{jureckova2009minimum}
\begin{barticle}[author]
\bauthor{\bsnm{Jureckov{\'a}},~\bfnm{Jana}\binits{J.}} \AND
  \bauthor{\bsnm{Picek},~\bfnm{Jan}\binits{J.}}
(\byear{2009}).
\btitle{Minimum risk equivariant estimator in linear regression model}.
\bjournal{Stat. decis.}
\bvolume{27}
\bpages{37--54}.
\end{barticle}
\endbibitem

\bibitem[\protect\citeauthoryear{Knight}{2020}]{knight2017asymptotic}
\begin{btechreport}[author]
\bauthor{\bsnm{Knight},~\bfnm{Keith}\binits{K.}}
(\byear{2020}).
\btitle{On the asymptotic distribution of the $L_{\infty}$ estimator in linear
  regression}
\btype{Technical Report},
\bpublisher{Mimeo, http://www. utstat. utoronto.ca/keith/home. html}.
\end{btechreport}
\endbibitem

\bibitem[\protect\citeauthoryear{Koltchinskii and
  Mendelson}{2015}]{koltchinskii2015bounding}
\begin{barticle}[author]
\bauthor{\bsnm{Koltchinskii},~\bfnm{Vladimir}\binits{V.}} \AND
  \bauthor{\bsnm{Mendelson},~\bfnm{Shahar}\binits{S.}}
(\byear{2015}).
\btitle{Bounding the smallest singular value of a random matrix without
  concentration}.
\bjournal{Int. Math. Res}
\bvolume{2015}
\bpages{12991--13008}.
\end{barticle}
\endbibitem

\bibitem[\protect\citeauthoryear{Laurent and
  Massart}{2000}]{laurent2000adaptive}
\begin{barticle}[author]
\bauthor{\bsnm{Laurent},~\bfnm{Beatrice}\binits{B.}} \AND
  \bauthor{\bsnm{Massart},~\bfnm{Pascal}\binits{P.}}
(\byear{2000}).
\btitle{Adaptive estimation of a quadratic functional by model selection}.
\bjournal{Ann. Stat.}
\bpages{1302--1338}.
\end{barticle}
\endbibitem

\bibitem[\protect\citeauthoryear{Lawson}{1961}]{lawson1961contribution}
\begin{barticle}[author]
\bauthor{\bsnm{Lawson},~\bfnm{Charles~Lawrence}\binits{C.~L.}}
(\byear{1961}).
\btitle{Contribution to the theory of linear least maximum approximation}.
\bjournal{Ph. D. dissertation, Univ. Calif.}
\end{barticle}
\endbibitem

\bibitem[\protect\citeauthoryear{Mendelson}{2010}]{mendelson2010empirical}
\begin{barticle}[author]
\bauthor{\bsnm{Mendelson},~\bfnm{Shahar}\binits{S.}}
(\byear{2010}).
\btitle{Empirical processes with a bounded $\psi$ 1 diameter}.
\bjournal{GAFA}
\bvolume{20}
\bpages{988--1027}.
\end{barticle}
\endbibitem

\bibitem[\protect\citeauthoryear{Milanese and
  Belforte}{1982}]{milanese1982estimation}
\begin{barticle}[author]
\bauthor{\bsnm{Milanese},~\bfnm{Mario}\binits{M.}} \AND
  \bauthor{\bsnm{Belforte},~\bfnm{Gustavo}\binits{G.}}
(\byear{1982}).
\btitle{Estimation theory and uncertainty intervals evaluation in presence of
  unknown but bounded errors: Linear families of models and estimators}.
\bjournal{IEEE Trans. Automat. Contr.}
\bvolume{27}
\bpages{408--414}.
\end{barticle}
\endbibitem

\bibitem[\protect\citeauthoryear{Mourtada}{2022}]{mourtada2019exact}
\begin{barticle}[author]
\bauthor{\bsnm{Mourtada},~\bfnm{Jaouad}\binits{J.}}
(\byear{2022}).
\btitle{Exact minimax risk for linear least squares, and the lower tail of
  sample covariance matrices}.
\bjournal{Ann. Stat.}
\bvolume{50}
\bpages{2157-2178}.
\end{barticle}
\endbibitem

\bibitem[\protect\citeauthoryear{Neykov}{2022}]{neykov2022minimax}
\begin{barticle}[author]
\bauthor{\bsnm{Neykov},~\bfnm{Matey}\binits{M.}}
(\byear{2022}).
\btitle{On the minimax rate of the Gaussian sequence model under bounded convex
  constraints}.
\bjournal{IEEE Trans. Inf. Theory}.
\end{barticle}
\endbibitem

\bibitem[\protect\citeauthoryear{Petrov}{2007}]{petrov2007lower}
\begin{barticle}[author]
\bauthor{\bsnm{Petrov},~\bfnm{Valentin~V}\binits{V.~V.}}
(\byear{2007}).
\btitle{On lower bounds for tail probabilities}.
\bjournal{J Stat Plan Inference}
\bvolume{137}
\bpages{2703--2705}.
\end{barticle}
\endbibitem

\bibitem[\protect\citeauthoryear{Qi}{2015}]{qi2015theoretical}
\begin{barticle}[author]
\bauthor{\bsnm{Qi},~\bfnm{Chong}\binits{C.}}
(\byear{2015}).
\btitle{Theoretical uncertainties of the Duflo--Zuker shell-model mass
  formulae}.
\bjournal{JPhysG}
\bvolume{42}
\bpages{045104}.
\end{barticle}
\endbibitem

\bibitem[\protect\citeauthoryear{Rademacher}{2007}]{rademacher2007approximating}
\begin{binproceedings}[author]
\bauthor{\bsnm{Rademacher},~\bfnm{Luis~A}\binits{L.~A.}}
(\byear{2007}).
\btitle{Approximating the centroid is hard}.
In \bbooktitle{Proceedings of the twenty-third annual symposium on
  Computational geometry}
\bpages{302--305}.
\end{binproceedings}
\endbibitem

\bibitem[\protect\citeauthoryear{Robbins and Zhang}{1986}]{robbins1986maximum}
\begin{barticle}[author]
\bauthor{\bsnm{Robbins},~\bfnm{Herbert}\binits{H.}} \AND
  \bauthor{\bsnm{Zhang},~\bfnm{Cun-Hui}\binits{C.-H.}}
(\byear{1986}).
\btitle{Maximum likelihood estimation in regression with uniform errors}.
\bjournal{Lecture Notes-Monograph Series}
\bpages{365--385}.
\end{barticle}
\endbibitem

\bibitem[\protect\citeauthoryear{Rudelson and
  Vershynin}{2010}]{rudelson2010non}
\begin{binproceedings}[author]
\bauthor{\bsnm{Rudelson},~\bfnm{Mark}\binits{M.}} \AND
  \bauthor{\bsnm{Vershynin},~\bfnm{Roman}\binits{R.}}
(\byear{2010}).
\btitle{Non-asymptotic theory of random matrices: extreme singular values}.
In \bbooktitle{Proceedings of the International Congress of Mathematicians 2010
  (ICM 2010) (In 4 Volumes) Vol. I: Plenary Lectures and Ceremonies Vols.
  II--IV: Invited Lectures}
\bpages{1576--1602}.
\bpublisher{World Scientific}.
\end{binproceedings}
\endbibitem

\bibitem[\protect\citeauthoryear{Schechtman and
  Schechtman}{1986}]{schechtman1986estimating}
\begin{barticle}[author]
\bauthor{\bsnm{Schechtman},~\bfnm{Edna}\binits{E.}} \AND
  \bauthor{\bsnm{Schechtman},~\bfnm{Gideon}\binits{G.}}
(\byear{1986}).
\btitle{Estimating the parameters in regression with uniformly distributed
  errors}.
\bjournal{JSCS}
\bvolume{26}
\bpages{269--281}.
\end{barticle}
\endbibitem

\bibitem[\protect\citeauthoryear{Sielken~Jr and Hartley}{1973}]{sielken1973two}
\begin{barticle}[author]
\bauthor{\bsnm{Sielken~Jr},~\bfnm{RL}\binits{R.}} \AND
  \bauthor{\bsnm{Hartley},~\bfnm{HO}\binits{H.}}
(\byear{1973}).
\btitle{Two linear programming algorithms for unbiased estimation of linear
  models}.
\bjournal{JASA}
\bvolume{68}
\bpages{639--641}.
\end{barticle}
\endbibitem

\bibitem[\protect\citeauthoryear{Sklar and Armstrong}{1982}]{sklar1982least}
\begin{barticle}[author]
\bauthor{\bsnm{Sklar},~\bfnm{Michael~G}\binits{M.~G.}} \AND
  \bauthor{\bsnm{Armstrong},~\bfnm{Ronald~D}\binits{R.~D.}}
(\byear{1982}).
\btitle{Least absolute value and Chebychev estimation utilizing least squares
  results}.
\bjournal{Math. Program.}
\bvolume{24}
\bpages{346--352}.
\end{barticle}
\endbibitem

\bibitem[\protect\citeauthoryear{Srivastava and
  Vershynin}{2013}]{srivastava2013covariance}
\begin{barticle}[author]
\bauthor{\bsnm{Srivastava},~\bfnm{Nikhil}\binits{N.}} \AND
  \bauthor{\bsnm{Vershynin},~\bfnm{Roman}\binits{R.}}
(\byear{2013}).
\btitle{Covariance estimation for distributions with $2+\varepsilon$ moments}.
\bjournal{Ann. Probab.}
\bvolume{41}
\bpages{3081--3111}.
\end{barticle}
\endbibitem

\bibitem[\protect\citeauthoryear{Tong}{2012}]{tong2012multivariate}
\begin{bbook}[author]
\bauthor{\bsnm{Tong},~\bfnm{Yung~Liang}\binits{Y.~L.}}
(\byear{2012}).
\btitle{The multivariate normal distribution}.
\bpublisher{Springer Science \&amp; Business Media}.
\end{bbook}
\endbibitem

\bibitem[\protect\citeauthoryear{Tse, Dahleh and
  Tsitsiklis}{1993}]{tse1993optimal}
\begin{barticle}[author]
\bauthor{\bsnm{Tse},~\bfnm{David~NC}\binits{D.~N.}},
  \bauthor{\bsnm{Dahleh},~\bfnm{Munther~A}\binits{M.~A.}} \AND
  \bauthor{\bsnm{Tsitsiklis},~\bfnm{John~N}\binits{J.~N.}}
(\byear{1993}).
\btitle{Optimal asymptotic identification under bounded disturbances}.
\bjournal{IEEE Trans. Automat. Contr.}
\bvolume{38}
\bpages{1176--1190}.
\end{barticle}
\endbibitem

\bibitem[\protect\citeauthoryear{Vershynin}{2011}]{vershynin2011spectral}
\begin{barticle}[author]
\bauthor{\bsnm{Vershynin},~\bfnm{Roman}\binits{R.}}
(\byear{2011}).
\btitle{Spectral norm of products of random and deterministic matrices}.
\bjournal{Probab. Theory Relat. Fields}
\bvolume{150}
\bpages{471--509}.
\end{barticle}
\endbibitem

\bibitem[\protect\citeauthoryear{Vershynin}{2012}]{Vershynin2012Introduction}
\begin{binproceedings}[author]
\bauthor{\bsnm{Vershynin},~\bfnm{Roman}\binits{R.}}
(\byear{2012}).
\btitle{Introduction to the non-asymptotic analysis of random matrices}.
In \bbooktitle{Compressed Sensing: Theory and Applications}
(\beditor{\bfnm{Y.~C.}\binits{Y.~C.}~\bsnm{Eldar}} \AND
  \beditor{\bfnm{G.}\binits{G.}~\bsnm{Kutyniok}}, eds.).
\bpublisher{Cambridge University Press}.
\end{binproceedings}
\endbibitem

\bibitem[\protect\citeauthoryear{Vershynin}{2018}]{vershynin2018high}
\begin{bbook}[author]
\bauthor{\bsnm{Vershynin},~\bfnm{Roman}\binits{R.}}
(\byear{2018}).
\btitle{High-dimensional probability: An introduction with applications in data
  science}
\bvolume{47}.
\bpublisher{Cambridge university press}.
\end{bbook}
\endbibitem

\bibitem[\protect\citeauthoryear{Wainwright}{2019}]{wainwright2019high}
\begin{bbook}[author]
\bauthor{\bsnm{Wainwright},~\bfnm{Martin~J}\binits{M.~J.}}
(\byear{2019}).
\btitle{High-dimensional statistics: A non-asymptotic viewpoint}
\bvolume{48}.
\bpublisher{Cambridge University Press}.
\end{bbook}
\endbibitem

\bibitem[\protect\citeauthoryear{Wasserman, Ramdas and
  Balakrishnan}{2020}]{wasserman2020universal}
\begin{barticle}[author]
\bauthor{\bsnm{Wasserman},~\bfnm{Larry}\binits{L.}},
  \bauthor{\bsnm{Ramdas},~\bfnm{Aaditya}\binits{A.}} \AND
  \bauthor{\bsnm{Balakrishnan},~\bfnm{Sivaraman}\binits{S.}}
(\byear{2020}).
\btitle{Universal inference}.
\bjournal{PNAS}
\bvolume{117}
\bpages{16880--16890}.
\end{barticle}
\endbibitem

\bibitem[\protect\citeauthoryear{Yaskov}{2014}]{yaskov2014lower}
\begin{barticle}[author]
\bauthor{\bsnm{Yaskov},~\bfnm{Pavel}\binits{P.}}
(\byear{2014}).
\btitle{Lower bounds on the smallest eigenvalue of a sample covariance matrix.}
\bjournal{ECP}
\bvolume{19}
\bpages{1--10}.
\end{barticle}
\endbibitem

\bibitem[\protect\citeauthoryear{Yu}{1997}]{yu1997assouad}
\begin{bincollection}[author]
\bauthor{\bsnm{Yu},~\bfnm{Bin}\binits{B.}}
(\byear{1997}).
\btitle{Assouad, fano, and le cam}.
In \bbooktitle{Festschrift for Lucien Le Cam}
\bpages{423--435}.
\bpublisher{Springer}.
\end{bincollection}
\endbibitem

\bibitem[\protect\citeauthoryear{Zhou}{2009}]{zhou2009restricted}
\begin{barticle}[author]
\bauthor{\bsnm{Zhou},~\bfnm{Shuheng}\binits{S.}}
(\byear{2009}).
\btitle{Restricted eigenvalue conditions on subgaussian random matrices}.
\bjournal{arXiv preprint arXiv:0912.4045}.
\end{barticle}
\endbibitem

\bibitem[\protect\citeauthoryear{Zolghadri and
  Henry}{2004}]{zolghadri2004minimax}
\begin{barticle}[author]
\bauthor{\bsnm{Zolghadri},~\bfnm{Ali}\binits{A.}} \AND
  \bauthor{\bsnm{Henry},~\bfnm{David}\binits{D.}}
(\byear{2004}).
\btitle{Minimax statistical models for air pollution time series. Application
  to ozone time series data measured in Bordeaux}.
\bjournal{Environ. Monit. Assess}
\bvolume{98}
\bpages{275--294}.
\end{barticle}
\endbibitem

\end{thebibliography}

\newpage
%%%%%%%%%%%%%
% Appendix
%%%%%%%%%%%%%
\appendix

% Page layouts for supplements
%-------------------------------------
%\newenvironment{changemargin}[2]{%
%\begin{list}{}{%
%\setlength{\topsep}{0pt}%
%\setlength{\leftmargin}{#1}%
%\setlength{\rightmargin}{#2}%
%\setlength{\listparindent}{\parindent}%
%\setlength{\itemindent}{\parindent}%
%\setlength{\parsep}{\parskip}%
%}%
%\item[]}{\end{list}}

%\begin{changemargin}{-1cm}{-1.1cm}
%------------------------------------

\section{Examples}\label{appendix:examples}

\begin{example}\label{Gaussian:example}
The first application of Corollary \ref{cor:paley:zygmund} with $\alpha =1$ and $q = 2$ is for Gaussian design. Suppose $\bX_i \sim N(0, \bSigma)$, where $\bSigma$ has smallest eigenvalue $\lambda_{\min} > 0$. It follows that $\tilde \bX_i = \eta_i \bX_i \sim N(0,\bSigma)$. Since, $\vb\T \tilde\bX \sim N(0, \vb\T \bSigma \vb)$ we have 
\begin{align*}
\EE |\vb\T \tilde\bX| = \sqrt{2/\pi}\sqrt{\vb\T \bSigma \vb} \geq \sqrt{2\lambda_{\min}/\pi}, ~~~ \EE (\vb\T \tilde\bX)^2 = \vb\T \bSigma \vb.
\end{align*}
 Set $\theta = 1/2$ (so $\xi = \sqrt{\frac{\lambda_{\min}}{8\pi}}$) to obtain $\rho = c^{-1} + (1 - 1/(2\pi))/2$. Set $c = 8\pi$ to obtain $\rho = 1/2 - 1/(8\pi)$. We have
\begin{align*}
    \PP(\xi \mathbb{B}^{p}_2 \not \subset \operatorname{conv}(\tilde\bX_1, \ldots, \tilde\bX_m)) & \leq (1 + 32\sqrt{2}\pi^{3/2}\sqrt{\operatorname{tr}(\bSigma)}/\sqrt{\lambda_{\min}})^p (1 - 1/(8\pi))^m \\
& \leq  (1 + 32\sqrt{2}\pi^{3/2} \sqrt{\operatorname{tr}(\bSigma)}/\sqrt{\lambda_{\min} })^p \exp(-m/(8\pi)),
\end{align*}
where we used that by Jensen's inequality $\EE \|\tilde\bX\| \leq \sqrt{\operatorname{tr}(\bSigma)}$. Hence the design contains a sphere of constant radius with probability at least $1 - \gamma$, so long as $m > 8\pi p \log (1 + 32\sqrt{2}\pi^{3/2} \sqrt{\operatorname{tr}(\bSigma)}/\sqrt{\lambda_{\min} }) + 8\pi \log \gamma^{-1}$.

Applying Theorem \ref{ball:in:convex:hull:follows:rate} now gives that with probability at least $1 -\gamma - \exp(- L^2/(8L/3 + 2))$
\begin{align*}
\|\hat \bbeta - \bbeta^*\| \leq \frac{a (L + 1)(8\pi p \log (1 + 32\sqrt{2}\pi^{3/2}\sqrt{\operatorname{tr}(\bSigma)}/\sqrt{\lambda_{\min} }) + 8\pi \log \gamma^{-1} + 1)}{\xi n},
\end{align*}
where recall that $\xi = \sqrt{\frac{\lambda_{\min}}{8\pi}}$.
\end{example}

\begin{example} Our next application includes applying Corollary \ref{cor:paley:zygmund} with $\alpha = 1$ and $q = 2$ to Rademacher design. Let $\bX_{ij}$ be i.i.d. Rademacher random variables. In this example, the first variable can also optionally be an intercept. In any case, it follows that $\tilde \bX_i = \eta_i \bX_{i}$ are Rademacher vectors. 

By Khintchine's inequality \citep{haagerup1978best} 
%(http://matwbn.icm.edu.pl/ksiazki/sm/sm70/sm70121.pdf)
 we now have that for any $\vb \in \mathbb{S}^{p-1}$, $1/\sqrt{2} \leq \EE |\langle \vb, \tilde \bX_i\rangle|$. In addition, clearly $\EE (\langle \vb, \tilde \bX_i\rangle)^2 = 1$. It follows that (plugging in $\theta = 1/2$ hence $\xi = 1/(4\sqrt{2})$) 
\begin{align*}
    \rho = c^{-1} + (1 - 1/8)/2 = c^{-1} + 7/16.
\end{align*}
and hence for $c > 32$ we have $\rho = 15/32 < 1/2$. We conclude that 
\begin{align*}
    \PP(\xi \mathbb{B}^{p}_2 \not \subset \operatorname{conv}(\tilde \bX_1, \ldots, \tilde \bX_m)) \leq (1 + 256\sqrt{2}\sqrt{p})^p (1 - 1/32)^m \leq (1 + 256\sqrt{2}\sqrt{p})^p\exp(-m/32).
\end{align*}
It follows there will be a $\xi$-sphere with probability at least $1 - \gamma$ as long as $m > 32 (p \log (1 + 256\sqrt{2}\sqrt{p}) + \log \gamma^{-1})$. This combined with the result of Theorem \ref{ball:in:convex:hull:follows:rate} shows that:
\begin{align*}
\|\hat \bbeta - \bbeta^*\| \leq \frac{a(L + 1)(32 (p \log (1 + 256\sqrt{2}\sqrt{p}) + \log \gamma^{-1}) + 1)}{\xi n},
\end{align*}
with probability at least $1 - \gamma - \exp\bigg(\frac{- L^2/2}{\frac{4}{3}L + 1}\bigg)$.
\end{example}

\begin{example} Let $\bX_i$ have a uniform distribution on the sphere. Then $\tilde \bX_i \stackrel{d}{=} \bX_i$. Let $\gb$ be a standard Gaussian random vector, and observe that $\tilde \bX_i \stackrel{d}{=} \frac{\gb}{\|\gb\|}$. We have $\sqrt{\frac{2}{\pi}} = \EE |\vb\T \frac{\gb}{\|\gb\|}| \|\gb\| = \EE |\vb\T \frac{\gb}{\|\gb\|}|\EE \|\gb\|$, so that $\EE |\vb\T \frac{\gb}{\|\gb\|}| \geq \frac{\sqrt{\frac{2}{\pi}}}{\EE \|\gb\|}$. Now use $\EE \|\gb\| \leq \sqrt{p}$ hence $\EE |\vb\T \frac{\gb}{\|\gb\|}| \geq \frac{\sqrt{\frac{2}{\pi}}}{\sqrt{p}}$. Next $\EE (\vb\T \frac{\gb}{\|\gb\|})^2 = \frac{1}{\EE\|\gb\|^2} = p^{-1}$, and therefore $\frac{(\EE |\vb\T \tilde \bX|)^2}{\EE[(\vb\T \tilde\bX)^2]} \geq \frac{2}{\pi}$. Applying Corollary \ref{cor:paley:zygmund} with $\alpha = 1$, $q = 2$, $\theta = \frac{1}{2}$, $\xi = \sqrt{\frac{1}{8\pi p}}$, and $\rho = \frac{1}{8\pi} + \frac{1 - \frac{1}{2\pi}}{2} = \frac{1}{2} - \frac{1}{8\pi}$ we now have that $\xi \mathbb{B}^{p}_2 \subset \operatorname{conv}(\tilde \bX_1,\ldots, \tilde \bX_m)$ with probability at least $1-\gamma$ whenever $m \geq 8\pi(p \log (1 + \frac{16\pi}{\xi}) + \log \gamma^{-1})$. By Theorem \ref{ball:in:convex:hull:follows:rate} we now have that 
\begin{align*}
\|\hat \bbeta - \bbeta^*\| \leq \frac{a(L + 1)(8\pi(p \log (1 + \frac{16\pi}{\xi}) + \log \gamma^{-1}) + 1)}{\xi n},
\end{align*}
with probability at least $1 -\gamma - \exp\bigg(\frac{- L^2/2}{\frac{4}{3}L + 1}\bigg)$.
 \end{example}

\begin{example} In this example we analyze a centered elliptical distribution $\bX$. This generalizes two of our previous examples where we considered Gaussian and uniform on the unit sphere distributions. By a stochastic representation theorem for centered elliptical distributions \cite[see Proposition 4.1.2 of][e.g.]{tong2012multivariate} we know that one can generate a centered elliptical random variable as $\bX \stackrel{d}{=} R \Ab \bU$, where $R \geq 0$ is a non-negative random variable independent of $\bU$, $\bU \stackrel{d}{=} \frac{\gb}{\|\gb\|}$  is distributed uniformly over the unit sphere $\mathbb{S}^{p-1}$, and $\Ab \in \RR^{p \times p}$ is a constant matrix. Suppose $\bSigma = \Ab \Ab\T$ has smallest eigenvalue $\lambda_{\min}$ bounded away from $0$ and largest eigenvalue $\lambda_{\max}$ being bounded. We have $\tilde \bX \stackrel{d}{=} \bX$. In what follows we also assume $\EE R > 0$ and $\EE R^2 < \infty$.

We now evaluate for a unit vector $\vb$, $\EE |\vb\T R \Ab \bU|= \EE |R| \EE |\vb \T \Ab \bU| = \frac{\EE R \|\vb\T \Ab\|\sqrt{\frac{2}{\pi}}}{\EE \|\gb\|} \geq \frac{\EE R\lambda_{\min}^{\frac{1}{2}}\sqrt{\frac{2}{\pi}}}{\sqrt{p}}$. On the other hand, $\EE  (\vb\T R \Ab \bU)^2 = \frac{\EE R^2 \|\vb\T \Ab\|^2}{\EE \|\gb\|^2} = \frac{\EE R^2 \|\vb\T \Ab\|^2}{p}$. Hence $\frac{(\EE |\vb\T R \Ab \bU|)^2}{\EE  (\vb\T R \Ab \bU)^2} \geq \frac{\frac{2}{\pi}(\EE R)^2}{\EE R^2}$. Next we upper bound $\EE \|\tilde \bX\| \leq \sqrt{ \EE \|\tilde \bX\|^2 } \leq \sqrt{\EE R^2 \lambda_{\max}}$. 

Set $\theta = \frac{1}{2}$, $c = \frac{8 \pi \EE R^2}{(\EE R)^2}$ to obtain $\rho = c^{-1} + \frac{1 -  \frac{\frac{2}{\pi}(\EE R)^2}{4 \EE R^2}}{2} = \frac{1}{2} - \frac{(\EE R)^2}{8\pi \EE R^2}$. Then by Corollary \ref{cor:paley:zygmund} with $\alpha = 1$, $q = 2$ we obtain
\begin{align*}
    \PP(\xi \mathbb{B}^{p}_2 \not \subset \operatorname{conv}(\tilde\bX_1, \ldots, \tilde\bX_m)) \leq \bigg(1 + \frac{C' \sqrt{p} (\EE R^2)^{3/2} \sqrt{\lambda_{\max}}}{(\EE R)^{3}\sqrt{\lambda_{\min}}}\bigg)^p\bigg(1 - \frac{(\EE R)^2}{8 \pi \EE R^2}\bigg)^{m},%(1 + 2c \EE \|\tilde \bX\|/\xi)^p (\frac{1}{2} + \rho)^m.
\end{align*}
for an absolute constant $C'$. As before if $m > \frac{8 \pi \EE R^2}{(\EE R)^2} \bigg(p \log \bigg(1 + \frac{C' \sqrt{p} (\EE R^2)^{3/2} \sqrt{\lambda_{\max}}}{(\EE R)^{3}\sqrt{\lambda_{\min}}}\bigg) + \log \gamma^{-1}\bigg)$, we have 
\begin{align*}
    \PP(\xi \mathbb{B}^{p}_2 \not \subset \operatorname{conv}(\tilde\bX_1, \ldots, \tilde\bX_m)) \leq \gamma,%(1 + 2c \EE \|\tilde \bX\|/\xi)^p (\frac{1}{2} + \rho)^m.
\end{align*}
and hence by Theorem \ref{ball:in:convex:hull:follows:rate} we have that with probability at least $1 -\gamma - \exp\bigg(\frac{- L^2/2}{\frac{4}{3}L + 1}\bigg)$:
\begin{align*}
\|\hat \bbeta - \bbeta^*\| \leq \frac{a(L + 1) \bigg(\frac{8 \pi \EE R^2}{(\EE R)^2} \bigg(p \log \bigg(1 + \frac{C' \sqrt{p} (\EE R^2)^{3/2} \sqrt{\lambda_{\max}}}{(\EE R)^{3}\sqrt{\lambda_{\min}}}\bigg) + \log \gamma^{-1}\bigg) + 1\bigg)}{\xi n},
\end{align*}
where $\xi = \frac{\EE R\lambda_{\min}^{\frac{1}{2}}\sqrt{\frac{2}{\pi}}}{4\sqrt{p}}$.
%{\color{red} rotationally invariant distr, or spherically symmetric one?}
\end{example}

\begin{example} \label{most:important:example:appendix} We now give a general example which only assumes that $\inf_{\vb \in \mathbb{S}^{p-1}}\EE \vb\T \bX\bX\T \vb = \lambda_{\min} > 0$ and $\sup_{\vb\in \mathbb{S}^{p-1}}\EE (\vb\T \bX)^4 \leq C < \infty$. The latter happens in the case when the variables $\bX$ are sub-Gaussian e.g. (in other words we assume that $\EE \exp(t^{-2} (\vb\T \bX)^2) \leq 2$ for some $t \in \RR^+$ for any $\vb \in \mathbb{S}^{p-1}$ (see also Definition \ref{sub-Gaussian:variable:def} in Section \ref{LASSO:section} for a formal definition)). Indeed, this is so by Lemma 5.5 of \cite{Vershynin2012Introduction}.

Clearly, under these assumptions $\inf_{\vb \in \mathbb{S}^{p-1}}\EE \vb\T \tilde \bX\tilde \bX\T \vb = \lambda_{\min} > 0$ and $\sup_{\vb\in \mathbb{S}^{p-1}}\EE (\vb\T\tilde  \bX)^4 \leq C < \infty$. Using Corollary \ref{cor:paley:zygmund} with $\alpha = 2$, $q = 2$, $\theta = \frac{1}{2}$ we have that $\rho = c^{-1} + \frac{1 - \frac{1}{4}\frac{\lambda^2_{\min}}{C}}{2}$ and $\xi = \frac{\lambda^{\frac{1}{2}}_{\min}}{2 \sqrt{2}}$. Setting $c^{-1} = \frac{1}{16}\frac{\lambda^2_{\min}}{C}$, gives $\rho = \frac{1}{2} - \frac{1}{16}\frac{\lambda^2_{\min}}{C} < \frac{1}{2}$. Next we can roughly upper bound $\EE \|\tilde \bX\| \leq \sqrt{\EE \|\tilde \bX\|^2} \leq \sqrt{pC^{1/2}}$, where we used that for any random variable $X$ we have $\EE X^2 \leq \sqrt{\EE X^4}$.

By Corollary \ref{cor:paley:zygmund} we have
\begin{align*}
    \PP(\xi \mathbb{B}^{p}_2 \not \subset \operatorname{conv}(\tilde\bX_1, \ldots, \tilde\bX_m)) \leq \bigg(1 + 64\sqrt{2} \frac{C}{\lambda^{5/2}_{\min}}  \sqrt{pC^{1/2}}\bigg)^p (1 - \frac{1}{16}\frac{\lambda^2_{\min}}{C})^m.
\end{align*}
Hence whenever $m > \frac{16C}{\lambda^2_{\min}}\bigg(p \log\bigg(1 + 64\sqrt{2}\frac{C^{5/4}}{\lambda_{\min}^{5/2}}p^{1/2}\bigg) + \log \gamma^{-1}\bigg)$ we have
\begin{align*}
    \PP(\xi \mathbb{B}^{p}_2 \not \subset \operatorname{conv}(\tilde\bX_1, \ldots, \tilde\bX_m)) \leq \gamma.
\end{align*}
Hence by Theorem \ref{ball:in:convex:hull:follows:rate}
\begin{align*}
\|\hat \bbeta - \bbeta^*\| \leq \frac{a(L + 1)\bigg(\frac{16C}{\lambda^2_{\min}}\bigg(p \log\bigg(1 + 64\sqrt{2}\frac{C^{5/4}}{\lambda_{\min}^{5/2}}p^{1/2}\bigg) + \log \gamma^{-1}\bigg) + 1\bigg)}{\xi n},
\end{align*}
with probability at least $1 - \gamma - \exp\bigg(\frac{- L^2/2}{\frac{4}{3}L + 1}\bigg)$.

One can of course assume even less assumptions in which case the bounds will worsen a bit. For instance, instead of assuming $\sup_{\vb\in \mathbb{S}^{p-1}}\EE (\vb\T \bX)^4 \leq C < \infty$ one can simply assume that the coordinates $\bX^{(j)}$ for $j \in [p]$ have bounded $4$-th moments by some constant $C_0$. The same analysis as above can be applied in this situation upon noting that for any $\vb \in \mathbb{S}^{p-1}$:
\begin{align*}
\EE (\vb\T\bX)^4 & = \EE \sum_{i,j,m,l} v_i v_j v_m v_l \bX^{(i)} \bX^{(j)}\bX^{(m)}\bX^{(l)}  \\
& \leq \EE \sum_{i,j,m,l} v_i v_j v_m v_l \sqrt[4]{\EE \bX^{(i)4} \EE \bX^{(j)4} \EE\bX^{(m)4} \EE\bX^{(l)4}} \\
& \leq \|\vb\|_1^4 C_0 \leq p^2 C_0.%\label{bounding:4th:moment}
\end{align*}

Finally, if one is bothered by $4$-th moment assumptions, this too can be relaxed. One needs to use Corollary \ref{cor:paley:zygmund} with $\alpha = 2$ and $q = 1 + \frac{\epsilon}{2}$ (so that $q\alpha = 2 + \epsilon$) for some $\epsilon > 0$. In this way, it suffices to assume that $\sup_{\vb \in \mathbb{S}^{p-1}}\EE |\vb\T \bX|^{2 + \epsilon} < \infty$ which is even weaker than a 4-th moment assumption. 
\end{example}

\section{Proofs}
%In this section we introduce several preliminary knowledge which will be needed in the consequent proofs. 

%\vskip 1cm
%%%%%%%%% DEF
Let $P$ and $Q$ be two probability distributions defined on space $(\mathcal{F},\mu)$. The \textit{Total Variation Distance} between $P$ and $Q$ is given by
\begin{align*}
    \|P-Q\|_{\operatorname{TV}} = \sup_{A\in\mathcal{F}}|P(A)-Q(A)| = \frac{1}{2}\int |p(x)-q(x)|d\mu(x)
\end{align*}
The \textit{Hamming Distance} for two vectors $\ab$ and $\bb$ is defined as the number of coordinates where $\ab_i\neq\bb_i$.

The next result is called Gershgorin's Disk Theorem, which is used to bound the eigenvalues of a square matrix.
\begin{theorem}
\label{gershgorin}
Let $\Ab\in\mathbb{R}^{n\times n}$ be a complex matrix with entries $a_{ij}$, and let $\lambda$ be an eigenvalue of $\Ab$. Then at least for one $i\in[n]$ we have
\begin{align*}
    |\lambda - a_{ii}| \leq \sum_{j\neq i}|a_{ij}|
\end{align*}
\end{theorem}
%Next we state the more general version of Theorem \ref{lemma:wendel_v3} which does not require $\EE \|\tilde \bX\| < \infty$.

We will also need the following Lemma.%, which is used in addition to Wendel's theorem above.

\begin{proof}[Proof of Lemma \ref{lemma:critical_ineq}]
We can obtain two inequalities from the constraint $Y_i - a\leq \Xb_i\T \bbeta \leq Y_i + a$:
   \begin{align*}
       -(a-|\varepsilon_i|) &\leq \Xb_i\T\bbeta-\Xb_i\T\bbeta^* \leq a+|\varepsilon_i|, \quad \text{if } \varepsilon_i > 0\\
       -(a+|\varepsilon_i|) &\leq  \Xb_i\T\bbeta-\Xb_i\T\bbeta^* \leq a-|\varepsilon_i|, \quad \text{if } \varepsilon_i \leq 0.
   \end{align*}
   This implies the following two inequalities:
   \begin{align*}
       -\sign(\varepsilon_{i}) \Xb_{i}\T (\bbeta - \bbeta^*) &\leq a - |\varepsilon_{i}|\\
       \sign(\varepsilon_{i}) \Xb_{i}\T (\bbeta - \bbeta^*) & \leq a + |\varepsilon_{i}|
   \end{align*}
We finally only use the first one as a critical inequality since $a - |\varepsilon_{i}|$ is stricter and more interesting. Its right hand side is close to zero when $|\varepsilon_{i}|$ is close to $a$. Thus we get the critical inequality
\begin{align*}
    -\sign(\varepsilon_{i}) \Xb_{i}\T (\bbeta - \bbeta^*) &\leq a - |\varepsilon_{i}|.
\end{align*}
Note that the RHS of the critical inequality is actually independent of the LHS, and by the independence of $\varepsilon_i$ and $\bX_i$ we have that the $-\sign(\varepsilon_{i})$ and $|\varepsilon_i|$ are independent. Hence we can think of the critical inequality as
\begin{align*}
\eta_{i} \bX_{i}\T (\hat\bbeta - \bbeta^*) \leq a - |\varepsilon_{i}|,
\end{align*}
where $\eta_{i}$ is a Rademacher random variable. 
\end{proof}

\begin{lemma}\label{lemma:uniform:random:variables:concentration}
Let $\{\varepsilon_i\}_{i \in [n]}$ be i.i.d. $U([-a,a])$ random variables. Sort the errors $|\varepsilon_i| \sim U([0,a])$ in decreasing manner $|\bar \varepsilon_{(i)}|$, so that $a \geq |\bar \varepsilon_{(1)}| \geq\ldots \geq |\bar \varepsilon_{(n)}| \geq 0$. Suppose $K \leq n$ is a fixed positive integer. Then 
\begin{align}
    \PP\Big[ \frac{|\bar \varepsilon_{(K)}|}{a} < 1 - \frac{K( L + 1)}{n}\Big] \leq\exp\bigg(\frac{- KL^2/2}{\frac{4}{3}L + 1}\bigg) \leq \exp\bigg(\frac{- L^2/2}{\frac{4}{3}L + 1}\bigg).
\label{bernstein_K}
\end{align}
By summing up the first bound over $K$ one may establish that for $L$ large enough, with at least constant probability \eqref{bernstein_K} holds simultaneously for all $K$.
\end{lemma}

\begin{proof}[Proof of Lemma \ref{lemma:uniform:random:variables:concentration}]
Consider the inequality 
\begin{align*}
    a-|\bar \epsilon_{(K)}| \leq \theta \quad \Leftrightarrow \quad |\bar \epsilon_{(K)}| \geq a - \theta,
\end{align*}
for some $\theta$. Suppose now $\theta \leq a$. Denote the number of $|\bar \varepsilon_i|$ being in the interval $[a-\theta,a]$ with $Z$. If $Z > K$, then clearly $|\bar \varepsilon_{(K)}| \geq a - \theta$, hence the event $|\bar \varepsilon_{(K)}| < a - \theta$ is a subset of the event $Z \leq K$. Since $|\bar \varepsilon_i|\sim U([0,a])$, the probability for an individual $|\bar \varepsilon_i|$ falling into the interval $[a-\theta,a]$ is $\frac{\theta}{a}$. One can see $Z$ follows a binomial distribution $Bin(n, \frac{\theta}{a})$. By (a one-sided) Bernstein's inequality \citep[Theorem 2.8.4]{vershynin2018high} we have
\begin{align*}
    \mathbb{P}\bigg(Z\leq\frac{n\theta}{a}-t\bigg) \leq \exp\bigg(\frac{-t^2/2}{\frac{n\theta}{a}(1-\frac{\theta}{a})+\frac{t}{3}}\bigg) \leq  \exp\bigg(\frac{-t^2/2}{\frac{n\theta}{a}+\frac{t}{3}}\bigg).
\end{align*}
Observe that if $K(L + 1) > n$, there is nothing to prove (since the probability in \eqref{bernstein_K} is $0$). Hence assuming $K(L + 1) \leq n$, set $\theta = \frac{K(L + 1)a}{n}$ and $t = KL$. This yields $\frac{n\theta}{a} - t = K$, and
\begin{align*}
\PP(Z \leq K) \leq \exp\bigg(\frac{-t^2/2}{\frac{n\theta}{a}+\frac{t}{3}}\bigg) = \exp\bigg(\frac{- KL^2/2}{\frac{4}{3}L + 1}\bigg) \leq \exp\bigg(\frac{- L^2/2}{\frac{4}{3}L + 1}\bigg),
\end{align*}
which is what we wanted to show.

%so that
%\begin{align*}
%    \mathbb{P}(Z\geq K) \geq 1-\exp\big(\frac{-t^2/2}{\frac{n\theta}{a}(1-\frac{\theta}{a})+\frac{t}{3}}\big),\quad 0\leq K \leq \frac{n\theta}{a}-t
%\end{align*}
\end{proof}
%}

%\begin{lemma}[Extension to symmetric bounded distributions]\label{extension:nonuniform:noise} Let $\{\varepsilon_i\}_{i \in [n]}$ be i.i.d. symmetric about $0$ random variables, bounded on the interval $[-a,a]$ with continuous distribution and cdf equal to $F_{\varepsilon}$. Sort the errors $|\varepsilon_i|$ in decreasing manner $|\bar \varepsilon_{(i)}|$, so that $a \geq |\bar \varepsilon_{(1)}| \geq\ldots \geq |\bar \varepsilon_{(n)}| \geq 0$. Suppose $K \leq n$ is a fixed integer and $L > 0$ is also fixed. It then follows that:
%\begin{align*}
%    &\PP\Big[ |\bar \varepsilon_{(K)}| < a - a_n(K, L)\Big] \leq \frac{(n+1)^2}{n(n+2)L^2},
%\end{align*}
%where $a_n(K, L)$ is defined as
%\begin{align*}
%a_n(K, L) = \inf\{a_n \in [0,2a] : 2(n + 1)(1 - F_{\varepsilon}(a - a_n)) > K(L + 1)\}.
%\end{align*}
%\end{lemma}
%
%\begin{proof}
%Clearly, the cdf of $|\varepsilon|$ is $2F_{\varepsilon}(y) - 1$. We know that $ u_{(K)}= 2F_{\varepsilon}(|\bar \varepsilon_{(K)}|) - 1$, where $u_{(K)} \sim Beta(n - K + 1, K)$. By Lemma \ref{lemma:uniform:random:variables:concentration} we have
%\begin{align*}
%\PP\bigg(u_{(K)} < 1 - \frac{K(L + 1)}{n+1}\bigg) \leq \frac{(n+1)^2}{n(n+2)L^2}.
%\end{align*}
%Hence
%\begin{align*}
%\PP\bigg(F_{\varepsilon}(|\bar \varepsilon_{(K)}|) < 1 - \frac{K(L + 1)}{2(n+1)}\bigg) \leq \frac{(n+1)^2}{n(n+2)L^2}.
%\end{align*}
%Since by the definition of $a_n(K)$ if $F_{\varepsilon}(|\bar \varepsilon_{(K)}|) \geq 1 - \frac{K(L + 1)}{2(n+1)}$ it follows that $|\bar \varepsilon_{(K)}| \geq a - a_n(K, L)$ the conclusion follows.
%\end{proof}

\begin{lemma}[Extension to symmetric bounded distributions]\label{extension:nonuniform:noise} Let $\{\varepsilon_i\}_{i \in [n]}$ be i.i.d. symmetric about $0$ random variables, bounded on the interval $[-a,a]$ with continuous distribution and cdf equal to $F_{\varepsilon}$. Sort the errors $|\varepsilon_i|$ in decreasing manner $|\bar \varepsilon_{(i)}|$, so that $a \geq |\bar \varepsilon_{(1)}| \geq\ldots \geq |\bar \varepsilon_{(n)}| \geq 0$. Suppose $K \leq n$ is a fixed integer and $L > 0$ is also fixed. It then follows that:
\begin{align*}
    &\PP\Big[ |\bar \varepsilon_{(K)}| < a - a_n(K, L)\Big] \leq \exp\bigg(\frac{- L^2/2}{\frac{4}{3}L + 1}\bigg),
\end{align*}
where $a_n(K, L)$ is defined as
\begin{align*}
a_n(K, L) = \inf\{a_n \in [0,2a] : 2n(1 - F_{\varepsilon}(a - a_n)) > K(L + 1)\}.
\end{align*}
\end{lemma}

\begin{proof}
Clearly, the cdf of $|\varepsilon|$ is $2F_{\varepsilon}(y) - 1$. We know that $ u_{(K)}= 2F_{\varepsilon}(|\bar \varepsilon_{(K)}|) - 1$, where $u_{(K)} \sim Beta(n - K + 1, K)$. By Lemma \ref{lemma:uniform:random:variables:concentration} we have
\begin{align*}
\PP\bigg(u_{(K)} < 1 - \frac{K(L + 1)}{n}\bigg) \leq \exp\bigg(\frac{- L^2/2}{\frac{4}{3}L + 1}\bigg).
\end{align*}
Hence
\begin{align*}
\PP\bigg(F_{\varepsilon}(|\bar \varepsilon_{(K)}|) < 1 - \frac{K(L + 1)}{2n}\bigg) \leq \exp\bigg(\frac{- L^2/2}{\frac{4}{3}L + 1}\bigg).
\end{align*}
Since by the definition of $a_n(K)$ if $F_{\varepsilon}(|\bar \varepsilon_{(K)}|) \geq 1 - \frac{K(L + 1)}{2n}$ it follows that $|\bar \varepsilon_{(K)}| \geq a - a_n(K, L)$, the conclusion follows.
\end{proof}

\begin{proof}[Proof of Lemma \ref{density:lemma}] We are concerned with the object

\begin{align*}
\PP(|\vb\T \bX| \leq 2\xi) = \PP((\vb\T \bX)^2 \leq (2\xi)^2). %= 1 - \PP((\vb\T \bX)^2 > (2\xi)^2)
\end{align*}
We have:
\begin{align*}
 \PP((\vb\T \bX)^2 \leq (2\xi)^2) = \PP(\exp(-\lambda(\vb\T \bX)^2) \geq \exp(-\lambda (2\xi)^2)) \leq \exp(\lambda (2\xi)^2)\EE \exp(-\lambda(\vb\T \bX)^2),
\end{align*}
for $\lambda > 0$. By H\"{o}lder's inequality we have
\begin{align*}
\EE \exp(-\lambda(\vb\T \bX)^2) = \int \exp(-\lambda t^2) f_{\vb}(t) dt & \leq \bigg[\int \exp\bigg(-\lambda \frac{q}{q-1} t^2\bigg) dt\bigg]^{\frac{q-1}{q}}  \bigg[\int f^q_{\vb}(t) dt \bigg]^{\frac{1}{q}} \\
& \leq \bigg(\sqrt{\frac{\pi (q-1)}{\lambda q}}\bigg)^{\frac{q-1}{q}}C%\bigg(\frac{1}{\sqrt{1 + \frac{2\lambda q}{q-1}}}\bigg)^{\frac{q-1}{q}}C
\end{align*}
Pick $\lambda = \frac{\pi(q-1)(C\epsilon^{-1})^{\frac{2q}{q-1}}}{q}$ so that the above bound becomes:
\begin{align*}
\EE \exp(-\lambda(\vb\T \bX)^2) \leq \epsilon%\int \exp(-\lambda t^2) f_{\vb}(t) dt \leq \bigg[\int \exp\bigg(-\lambda \frac{q}{q-1} t^2\bigg)\bigg]^{\frac{q-1}{q}}  \bigg[\int f^q_{\vb}(t) dt \bigg]^{\frac{1}{q}} \leq \frac{1}{\sqrt{1 + \frac{2\lambda q}{q-1}}}C
\end{align*}
Now select $\xi = \frac{\lambda^{-1/2}}{2}$. With this choice we obtain
\begin{align*}
\PP(|\vb\T \bX| \leq 2\xi)  =  \PP((\vb\T \bX)^2 \leq (2\xi)^2) \leq \epsilon e < c_0,
\end{align*}
for any $\epsilon < \frac{c_0}{e}$. This completes the proof since $\vb \in \mathbb{S}^{p-1}$ was arbitrary. 
\end{proof}

\begin{proof}[Proof of Proposition \ref{theorem:alternative_bound}]
Since $\bbeta^*$ is a feasible point of \eqref{main:optimization:problem}, and $\hat\bbeta$ minimizes the least squares among all feasible points, we have $\|\bY-\Xb\hat\bbeta\|^2\leq\|\bY-\Xb\bbeta^*\|^2=\|\bvarepsilon\|^2$, and consequently the following basic inequality:
\begin{align*}
(\hat\bbeta - \bbeta^*)\T n^{-1}\sum_{i \in [n]} \bX_i \bX_i\T (\hat\bbeta - \bbeta^*) \leq 2 n^{-1}\sum_{i \in [n]} \varepsilon_i \bX_i\T (\hat\bbeta - \bbeta^*).
\end{align*}
The above inequality can be rewritten as
\begin{align}
    &\inf_{\vb\in\mathbb{S}^{p-1}}n^{-1}\sum_{i\in[n]}\vb\T\bX_i\bX_i\T\vb\,\|\hat\bbeta - \bbeta^*\|^2 \leq \bigg\|2 n^{-1}\sum_{i \in [n]} \varepsilon_i \bX_i \bigg\| \|\hat\bbeta - \bbeta^*\|\nonumber,\\
    \Rightarrow \quad & \|\hat\bbeta - \bbeta^*\| \leq \frac{\big\|2 n^{-1}\sum_{i \in [n]} \varepsilon_i \bX_i \big\|}{\inf_{\vb\in\mathbb{S}^{p-1}}n^{-1}\sum_{i\in[n]}\vb\T\bX_i\bX_i\T\vb}.
\label{middle_alternative}
\end{align}
It remains to upper bound the term $\big\|2 n^{-1}\sum_{i \in [n]} \varepsilon_i \bX_i \big\|$ and lower bound the term $$\displaystyle\inf_{\vb\in\mathbb{S}^{p-1}}n^{-1}\sum_{i\in[n]}\vb\T\bX_i\bX_i\T\vb.$$ We first consider
\begin{align*}
S^2 := \EE \big\|n^{-1}\sum_{i \in [n]} \varepsilon_i \bX_i \big\|^2 = \EE \sum_{j \in [p]} \bigg(\frac{\sum_{i \in [n]} \sign(\varepsilon_i)|\varepsilon_i| \bX^{(j)}_i}{n}\bigg)^2.
\end{align*}
Observe that $\sign(\varepsilon_i)$ is a Rademacher random variable independent of $|\varepsilon_i| \bX^{(j)}_i$. Thus we may apply Khintchine's inequality (conditionally) to argue that
\begin{align*}
\EE \sum_{j \in [p]} \bigg(\frac{\sum_{i \in [n]} \sign(\varepsilon_i)|\varepsilon_i| \bX^{(j)}_i}{n}\bigg)^2 \leq n^{-2} K^2_2 \sum_{j \in [p]} \sum_{i \in [n]} \EE \varepsilon_i^2 \bX_i^{2(j)} \lesssim \sqrt{C_0} a^2\frac{p}{n}.
\end{align*}
where we used that $\EE \bX_i^{2(j)} \leq \sqrt{C_0}$ and that $\EE \varepsilon_i^2 = a^2/3$, and where $K_2$ is an absolute constant from Khintchine's inequality. Next, we will evaluate the variance of this term. We have

\begin{align*}
\MoveEqLeft\operatorname{Var}\sum_{j \in [p]} \bigg(\frac{\sum_{i \in [n]} \sign(\varepsilon_i)|\varepsilon_i| \bX^{(j)}_i}{n}\bigg)^2 \leq \EE \bigg(\sum_{j \in [p]} \bigg(\frac{\sum_{i \in [n]} \sign(\varepsilon_i)|\varepsilon_i| \bX^{(j)}_i}{n}\bigg)^2\bigg)^2 \\
& = \EE \sum_{j,j'} \bigg(\frac{\sum_{i \in [n]} \sign(\varepsilon_i)|\varepsilon_i| \bX^{(j)}_i}{n}\bigg)^2 \bigg(\frac{\sum_{i \in [n]} \sign(\varepsilon_i)|\varepsilon_i| \bX^{(j')}_i}{n}\bigg)^2\\
& \leq 2^{-1}\EE \sum_{j,j'} \bigg(\frac{\sum_{i \in [n]} \sign(\varepsilon_i)|\varepsilon_i| \bX^{(j)}_i}{n}\bigg)^4 + \bigg(\frac{\sum_{i \in [n]} \sign(\varepsilon_i)|\varepsilon_i| \bX^{(j')}_i}{n}\bigg)^4\\
& = p \EE \sum_{j} \bigg(\frac{\sum_{i \in [n]} \sign(\varepsilon_i)|\varepsilon_i| \bX^{(j)}_i}{n}\bigg)^4
\end{align*}
Now we may apply Khintchine's inequality once again. We have
\begin{align*}
\MoveEqLeft p\EE \sum_{j} \bigg(\frac{\sum_{i \in [n]} \sign(\varepsilon_i)|\varepsilon_i| \bX^{(j)}_i}{n}\bigg)^4 \leq p n^{-4} K_4^4 \sum_{j \in [p]} \EE \bigg(\sum_{i \in [n]} (\varepsilon_i)^2 \bX^{2(j)}_i\bigg)^2\\
&=  p n^{-4} K_4^4 \sum_{j \in [p]} \EE \sum_{i,i' \in [n]} (\varepsilon_i)^2 \bX^{2(j)}_i (\varepsilon_{i'})^2 \bX^{2(j)}_{i'}\\
& \leq p n^{-3} K_4^4 \sum_{j \in [p]} \EE \sum_{i \in [n]} (\varepsilon_i)^4 \bX^{4(j)}_i\\
& \lesssim C_0 p^2 n^{-2} a^4.
\end{align*}

Hence by Chebyshev inequality 
\begin{align*}
\PP\bigg(S^2  \geq \frac{c \sqrt{C_0}p a^2}{n} + t\bigg) \leq \PP(S^2  \geq \EE S^2 + t) \leq \PP(|S^2 - \EE S^2| \geq t) \leq \frac{\Var S^2}{t^2} \lesssim \frac{C_0 a^4 p^2}{n^2 t^2}.
\end{align*}
Thus one can set $t = C\frac{\sqrt{C_0} a^2 p}{n}$ for some large constant $C$. Hence with probability $1 - 1/C^2$ we will have $S^2 \lesssim \frac{\sqrt{C_0}p a^2}{n}$. This completes the bound of the first term.
 
%To handle the first term, let $S = \frac{1}{\sqrt{n}} \sum_{i} \varepsilon_i \bX_i$. Note that $Cov(S)=\EE \varepsilon_i^2 \bX_i \bX_i\T = a^2/3 \bSigma$ is invertible. Let $Z\sim N(0, a^2/3\bSigma)$. By the multivariate Berry-Esseen Theorem \citep[Theorem 1.1]{raivc2019multivariate}, for any convex set $U\subset \mathbb{R}^p$ we have
%\begin{align*}
%|\PP(S \in U) - \PP(Z \in U)| \lesssim p^{1/4} \EE\Big[\|(a^2/3 \bSigma)^{-1/2} \bX_i/\sqrt{n}\|^3\Big].
%\end{align*}
%Suppose $\bSigma$ has bounded eigenvalues. If one takes $U$ to be the set $U = \{\vb\in\mathbb{R}^p: \|\vb\| \leq C\sqrt{p}\}$ for a constant $C>0$, the probability $\PP(Z \in U)$ would converge to $1$ since a $\chi^2_p$ random variable will concentrate on the order $O(p)$ \citep[Lemma 1]{laurent2000adaptive}. With probability converging to $1$ we will get %(maybe we can use a Dvoretzky-Milman type of theorem to claim this -- i.e. the convex hull of the Gaussians is contained with ) we will get that with high prob
%\begin{align*}
%\bigg \|2 n^{-1}\sum_{i \in [n]} \varepsilon_i \bX_i \bigg\| \lesssim \sqrt{\frac{p}{n}}. 
%\end{align*}

To lower bound the second term, we rewrite it as
\begin{align*}
\inf_{\vb\in\mathbb{S}^{p-1}}n^{-1}\sum_{i\in[n]}\vb\T\bX_i\bX_i\T\vb = \inf_{\vb\in\mathbb{S}^{p-1}} \frac{1}{n}\vb\T \bSigma^{1/2} \sum_{i\in[n]} \bSigma^{-1/2} \bX_{i}\bX_{i}\T\bSigma^{-1/2}\bSigma^{1/2}\vb
\end{align*}
Let $\Ab =\frac{1}{n}\sum_{i\in[n]}\bSigma^{-1/2}\bX_{i}\bX_{i}\T\bSigma^{-1/2}$. Note that
\begin{align*}
    \mathbb{E}\Ab = \Ib_p
\end{align*}
Let $m\in[p]$, $l\in[p]$, and $c>0$ be constants. By Chebyshev's inequality, for the $(m, l)$-th entry of $\Ab$ and $\Ib_p$ we have
\begin{align}
\label{chebyshev_indiv_alternative}
    \mathbb{P}\Big(\,\big|\Ab^{(m,l)}-\Ib_p^{(m,l)}\big| > cp\sqrt{\var(\Ab^{(m,l)})}\,\Big) \leq \frac{1}{c^2p^2}
\end{align}

Apply union bound to \eqref{chebyshev_indiv_alternative}, given the bound on $\var(\Ab^{(m,l)})$, with probability at least $1- c^{-2}$
\begin{align*}
    \|\Ab - \Ib_p\|_{\max} \leq cp\sqrt{\var(\Ab^{(m,l)})}
\end{align*}
It follows that the $\infty$ norm of the matrix is bounded as
\begin{align*}
\|\Ab - \Ib_p\|_{\infty} \leq cp^2\sqrt{\var(\Ab^{(m,l)})}
\end{align*}
From here one can show that
\begin{align*}
\Ab^{(m,m)} - \sum_{l \neq m} |\Ab^{(m,l)}| \geq 1 - cp^2\sqrt{\var(\Ab^{(m,l)})}
\end{align*}
From Gershgorin's Disk Theorem it follows that the eigenvalues of $\Ab$ are lower bounded as $1 - cp^2\sqrt{\var(\Ab^{(m,l)})}$. Let $\lambda_{\min}(\Ab)$ be the smallest eigenvalue of $\Ab$, then
\begin{align*}
\inf_{\vb\in\mathbb{S}^{p-1}} \frac{1}{n}\vb\T \bSigma^{1/2} \sum_{i\in[n]} \bSigma^{-1/2} \bX_{i}\bX_{i}\T\bSigma^{-1/2}\bSigma^{1/2}\vb &\geq \lambda_{\min}(\Ab) \|\bSigma^{1/2}\vb\|^2\\
& \geq (1 - cp^2\sqrt{\var(\Ab^{(m,l)})})\lambda_{\min}(\bSigma)
\end{align*}

The quantity $\var(\Ab^{(m,l)})$ can be upper bounded as $\var(\Ab^{(m,l)})\lesssim\frac{1}{n} \|\bSigma^{-1}\|_{\operatorname{op}}^2\,p^2C_0$, which we prove in Lemma \ref{lemma:var_4thmmt} below. If $c\lesssim \frac{\sqrt{n}}{2\sqrt{C_0}p^3\|\bSigma^{-1}\|_{\operatorname{op}}}$, we have $(1 - cp^2\sqrt{\var(\Ab^{(m,l)})})\geq\frac{1}{2}$, so that combine with \eqref{middle_alternative} to get
\begin{align*}
    \|\bbeta - \bbeta^*\| \lesssim \frac{\sqrt{p/n}}{\frac{1}{2}\lambda_{\min}(\bSigma)}\asymp \sqrt{\frac{p}{n}}\|\bSigma^{-1}\|_{\operatorname{op}},
\end{align*}
with probability converging to $1$ if $c\rightarrow+\infty$ as $n\rightarrow+\infty$.

%{\color{purple}
 To complete the second part we will use Corollary 3.1 of \cite{yaskov2014lower} (see also \cite[Theorem 1.5]{srivastava2013covariance}) which states that assuming $\EE |\vb\T \bX|^{2 + \alpha} < \infty$ there exists a constant $C_\alpha$ such that 
\begin{align*}
\lambda_{\min}(\Ab) \geq 1 - C_\alpha \bigg(\frac{p}{n}\bigg)^{2/(2 + \alpha)},
\end{align*}
with probability at least $1 - e^{-p}$. Hence when $n > C_\alpha' p$ the above can be made bigger than $\frac{1}{2}$, in which case the proof may continue in the same fashion as before. This completes the proof.
%}
\end{proof}

\begin{lemma}
\label{lemma:var_4thmmt}
For $\Ab =\frac{1}{n}\sum_{i\in[n]}\bSigma^{-1/2}\bX_{i}\bX_{i}\T\bSigma^{-1/2}$, if the 4-th moment of each coordinate in $\bX_i$ is bounded by $C_0$, then
\begin{align*}
    \var(\Ab^{(m,l)})\leq \frac{1}{n} \|\bSigma^{-1}\|_{\operatorname{op}}^2p^2C_0.
\end{align*}
\end{lemma}
\begin{proof}
Since $\bX_{i}$ are i.i.d., we have
\begin{align*}
    \var(\Ab^{(m,l)}) & = %\frac{1}{M}\var[\eb_m\T\bSigma^{'-1/2} \tilde\bX_{\ell^*}\tilde\bX_{\ell^*}\T\bSigma^{'-1/2}\eb_l\,\,\big|\ell\in S] \\
 \frac{1}{n}\var[\eb_m\T\bSigma^{'1/2} \bX_{i}\bX_{i}\T\bSigma^{-1/2}\eb_l]
\end{align*}
Therefore
\begin{align*}
   \MoveEqLeft \var[\eb_m\T\bSigma^{'1/2} \bX_{i}\bX_{i}\T\bSigma^{-1/2}\eb_l]\leq \mathbb{E}[(\eb_m\T\bSigma^{'1/2} \bX_{i})^2s(\bX_{i}\T\bSigma^{-1/2}\eb_l)^2] \\
    & \leq \|\eb_m\T\bSigma^{-\frac{1}{2}}\|^2_2\|\eb_l\T\bSigma^{-\frac{1}{2}}\|^2_2 \mathbb{E}[\|\bX_{i}\|^4] \leq \|\bSigma^{'-1}\|_{\operatorname{op}}^2 \mathbb{E}[\|\bX_{i}\|^4],
\end{align*}
%{\color{red} DOUBLE CHECK THESE INEQUALITIES}
It is now clear that $\mathbb{E}[\|\bX_{i}\|^4] = \sum_{k,l \in [p]} \EE \tilde\bX_{i}^{(k)2}\tilde\bX_{i}^{(l)2} \leq \sum_{k,l \in [p]} [\EE \tilde\bX_{i}^{(k)4}]^{\frac{1}{2}} [\EE \tilde\bX_{i}^{(l)4}]^{\frac{1}{2}} \leq p^2 C_0$. This completes the proof.
\end{proof}

%%%%%%%%%%%%
%%
%%%%%%%%%%%%%
%\section{}
\begin{proof}[Proof of Theorem \ref{theorem:lower_bound}]
The minimax risk is defined as
\begin{align*}
\inf_{\hat \bbeta} \sup_{\bbeta^* \in \RR^p} \EE_{\bbeta^*} \|\hat \bbeta - \bbeta^*\|^2.
\end{align*}

Let $\Rb$ be any $p \times p$ orthogonal matrix. Pick some $\delta>0$ and define $\bbeta_{\bnu}=\delta\Rb\bnu$ where $\bnu=\{+1,-1\}^p$. Define a probability distribution corresponding to $\bnu$ as $\mathcal{P}_{\bnu}^{\otimes n}=\bX_i\T\bbeta_{\bnu}+U([-a,a])$. This is the distribution of $Y_i$ for a linear model $Y_i = \bX_i\T \bbeta_{\bnu} + U_i$ where $U_i \sim U([-a,a])$. 
Notice that we have a $\delta$-Hamming separation for the loss function $\|\cdot\|^2$
\begin{align*}
\|\hat \bbeta - \bbeta_{\bnu}\|^2 = \|\Rb \Rb\T\hat \bbeta - \Rb\delta\bnu\|^2 = \|\Rb\T\hat \bbeta - \delta\bnu\|^2  \geq \delta^2 \sum_{j \in [p]} \mathbbm{1}(\sign((\Rb\T\hat \beta)_j) \neq \nu_j).
\end{align*}
We now need to repeat the proof of Assouad's lemma (in order to capture the expectation with respect to the covariates $\bX$). We have

\begin{align*}
    \sup_{\bbeta^*} \EE_{\Xb} \EE_{\bY, \bbeta^* | \Xb} \|\hat \bbeta - \bbeta^*\|^2 & \geq \frac{1}{2^p} \sum_{\bnu} \delta^2 \sum_{j \in [p]} \EE_{\Xb} \EE_{\bY, \bbeta_{\bnu} | \Xb} \mathbbm{1}(\sign((\Rb\T\hat \bbeta)_j) \neq \nu_j)\\
    % & \geq \frac{1}{2^p} \sum_{\bnu} \delta^2 \sum_{j \in [p]} \EE_{\Xb} \EE_{\bY, \bbeta_{\bnu} | \Xb} \mathbbm{1}(\sign(\beta_j) \neq \nu_j)\\
    & \geq \delta^2 \EE_{\Xb} \sum_{j \in [p]} \frac{1}{2} \bigg(\frac{1}{2^{p-1}} \sum_{\bnu: \nu_j = 1} \EE_{\bY, \bbeta_{\bnu} | \Xb} \mathbbm{1}(\sign((\Rb\T\hat \bbeta)_j) \neq 1) \\
    & ~~~~ + \frac{1}{2^{p-1}} \sum_{\bnu: \nu_j = -1} \EE_{\bY, \bbeta_{\bnu} | \Xb} \mathbbm{1}(\sign((\Rb\T\hat \bbeta)_j) \neq -1)\bigg)\\
    & \geq \delta^2 \EE_{\Xb} \sum_{j \in [p]}\frac{1}{2}\bigg(1 - \bigg\|\frac{1}{2^{p-1}} \sum_{\bnu: \nu_j = 1} \PP^{\otimes n}_{\bY, \bbeta_{\bnu} | \Xb} - \frac{1}{2^{p-1}} \sum_{\bnu: \nu_j = -1} \PP^{\otimes n}_{\bY, \bbeta_{\bnu}|\Xb} \bigg\|_{\operatorname{\operatorname{TV}}}\bigg).
\end{align*}

Taking $\inf$ over all estimators on the LHS concludes that:
\begin{align}
\label{assouad_expression}
\inf_{\hat \bbeta} \sup_{\bbeta^* \in \RR^p} \EE_{\bbeta^*} \|\hat \bbeta - \bbeta^*\|^2 \geq \frac{\delta^2}{2} \sum_{j = 1}^p [1 - \EE_{\Xb} \|P^{\otimes
 n}_{+j} - P^{\otimes n}_{-j}\|_{\TV}],
\end{align}
Notice that 
$$P^{\otimes n}_{+ j}=2^{-p}\sum_{\bnu\in\{+1,-1\}^p}\mathcal{P}^{\otimes n}_{\bnu:\nu_j=+1},
\quad P^{\otimes n}_{-j}=2^{-p}\sum_{\bnu\in\{+1,-1\}^p}\mathcal{P}^{\otimes n}_{\bnu:\nu_j=-1},$$ 
where by $\mathcal{P}^{\otimes n}_{\bnu:\nu_j=+1}$ we mean the probability under $\bnu$ where we set $\nu_j = +1$ regardless  of the value of $\nu_j$, and similarly for $\mathcal{P}^{\otimes n}_{\bnu:\nu_j=-1}$.
The total variation can be bounded as
\begin{align*}
\EE_{\Xb}\|P^{\otimes n}_{+j} - P^{\otimes n}_{-j}\|_{\TV} & \leq 2^{-p}\sum_{\bnu\in\{+1,-1\}^p}\EE_{\Xb}\|P^{\otimes n}_{\bnu:\nu_j=1} - P^{\otimes n}_{\bnu':\nu'_j=-1}\|_{\operatorname{TV}} \\
& \leq
\max_{j\in[p]} \EE_{\Xb}\|P^{\otimes n}_{\bnu:\nu_j=1} - P^{\otimes n}_{\bnu':\nu'_j=-1}\|_{\operatorname{TV}}\\
& \leq \max_{\bnu, \bnu': d_{H}(\bnu, \bnu') = 1} \EE_{\Xb}\|P^{\otimes n}_{\bnu} - P^{\otimes n}_{\bnu'}\|_{\TV},
\end{align*}
where $d_H$ is the Hamming distance. Now it is easy to see that if one has two uniforms $c_1 + U([-a,a])$ and $c_2 + U([-a,a])$ (call those $P_{c_1}$, $P_{c_2}$) we have
\begin{align*}
\|P_{c_1} - P_{c_2}\|_{\TV} = \frac{|c_1 - c_2|}{2a} \wedge 1.
\end{align*}
Thus for any fixed $\bnu$ and $\bnu'$ with $d_{H}(\bnu, \bnu') = 1$ we have that 
\begin{align*}
\EE_{\Xb}\|P^{\otimes n}_{\bnu} - P^{\otimes n}_{\bnu'}\|_{\TV} \leq \sum_{i \in [n]} \EE_{\Xb}\frac{|\bX_i\T(\bbeta_{\bnu} - \bbeta_{\bnu'})|}{2a} = \sum_{i \in [n]} \EE_{\Xb}\frac{|(\bX_i\T \Rb)_{j}| \delta}{a},
\end{align*}
where $j$ is the coordinate where $\nu_j \neq \nu_j'$. Hence
\begin{align}\label{equation:which:is:different}
\max_{\bnu, \bnu': d_{H}(\bnu, \bnu') \leq 1} \EE_{\Xb}\|P^{\otimes n}_{\bnu} - P^{\otimes n}_{\bnu'}\|_{\TV} \leq \max_{j\in[p]} \EE_{\Xb} \sum_{i \in [n]} \frac{|(\bX_i\T \Rb)_{j}| \delta}{a}.
\end{align}
Picking $\delta = \frac{a}{2 \inf_{\Rb \in \cO} \max_{j}\EE_{\Xb} \sum_{i\in [n]} |(\bX_i\T \Rb)_{j}|}$, by \eqref{assouad_expression} we have 
\begin{align*}
\inf_{\hat \bbeta} \sup_{\bbeta^* \in \RR^p} \EE_{\bbeta^*} \|\hat \bbeta - \bbeta^*\|^2 \geq p \delta^2/4,
\end{align*}
which completes the first part of the proof. 

For the next part suppose $\tilde \Rb$ achieves the $\min$ in the definition of $\delta$ (if the $\min$ cannot be achieved the same argument will go through by taking a sequence that converges to the $\inf$). We let $S = \{\delta \tilde \Rb\bnu: \bnu \in \{\pm1\}^p\}$ where $\delta$ is as above. We have

\begin{align*}
    \inf_{\hat \bbeta} \sup_{\bbeta^*} \PP(\|\hat \bbeta - \bbeta^*\| \geq r_n) \geq \inf_{\hat \bbeta} \sup_{\bbeta^* \in S} \PP(\|\hat \bbeta - \bbeta^*\| \geq r_n),
\end{align*}
where $r_n := \delta \sqrt{p}/(2\sqrt{2})$. Now observe that in the RHS above, instead of taking $\inf$ over all $\hat \bbeta$ it suffices to take $\inf$ over $\hat \bbeta$ which belong to the set $S' = \{\bgamma: \exists \bbeta \in S, \|\bgamma - \bbeta\| \leq \operatorname{diam}_{L_2}(S) \}$. This is so since if $\hat \bbeta$ achieves the $\inf$ we can always consider $\tilde \bbeta = \hat \bbeta$ if $\hat \bbeta \in S'$ and a random vector in $S$ otherwise. Clearly, by this definition $\|\tilde \bbeta - \bbeta^*\| \leq \|\hat \bbeta - \bbeta^*\|$ and hence
\begin{align*}
    \PP(\|\hat \bbeta - \bbeta^*\| \geq r_n) \geq \PP(\|\tilde \bbeta - \bbeta^*\| \geq r_n).
\end{align*}
Thus we have established
\begin{align*}
    \inf_{\hat \bbeta} \sup_{\bbeta^*} \PP(\|\hat \bbeta - \bbeta^*\| \geq r_n) \geq \inf_{\hat \bbeta \in S'} \sup_{\bbeta^* \in S} \PP(\|\hat \bbeta - \bbeta^*\| \geq r_n),
\end{align*}
where $\hat \bbeta \in S'$ means measurable functions which output values in the set $S'$. From Assouad's lemma above we know that for any $\hat \bbeta$, $\sup_{\bbeta^* \in S} \EE \|\hat \bbeta - \bbeta^*\|^2_2 \geq 2r_n^2$. Thus for any $\hat \bbeta$, there exists a $\bbeta' \in S$ such that 
\begin{align*}
    \sup_{\bbeta^* \in S} \PP(\|\hat \bbeta - \bbeta^*\|^2_2 \geq r^2_n) \geq \PP(\|\hat \bbeta - \bbeta'\|^2 \geq r^2_n) \geq \PP(\|\hat \bbeta - \bbeta'\|^2_2 \geq 1/2\EE \|\hat \bbeta - \bbeta'\|^2) \geq 1/4 \frac{(\EE \|\hat \bbeta - \bbeta'\|^2_2)^2}{\EE \|\hat \bbeta - \bbeta'\|^4_2}\\
    \geq \frac{r_n^4}{\sup_{\bbeta^* \in S} \EE \|\hat \bbeta - \bbeta^*\|^4_2}
\end{align*}
provided that $\EE \|\hat \bbeta - \bbeta^*\|^4_2$ exists, where the above follows by Paley-Zygmund's inequality. We now need to note that 
\begin{align*}
    \sup_{\hat \bbeta \in S'} \sup_{\bbeta^* \in S} \EE \|\hat \bbeta - \bbeta^*\|^4 \leq 4 \operatorname{diam}_{L_2}(S)^4 = 4 (\delta^2 p)^2.
\end{align*}
Observe that this is precisely of the same order as $r^4_n$ hence it shows that the probability $$\inf_{\hat \bbeta \in S'} \sup_{\bbeta^* \in S} \PP(\|\hat \bbeta - \bbeta^*\| \geq r_n),$$ is lower bounded by a constant ($1/2^8$).

\end{proof}

% In the above, taking expectation with respect to $\Xb$, yields a choice of $\delta \asymp 1/n$ (in the case when $\Xb$ comes from a Gaussian distribution).

% {\color{purple} If $\Xb$ follows an isotropic Gaussian design such that $X_{ij} \sim N(0,1)$, by a standard concentration result of maximum sub-Gaussian variables \citep[Exercise 2.5.10]{vershynin2018high}, with probability converging to one we have
% \begin{align*}
%     \max_{j}\sum_{i\in [n]} |X_{ij}|\lesssim n\sqrt{\log p}
% \end{align*}
% so that
% \begin{align*}
%     \delta \gtrsim \frac{1}{n\sqrt{\log p}}
% \end{align*}
% Besides the isotropic Gaussian design, if $\Xb$ has bounded entries, then we have $\delta\gtrsim1/n$. Thus for an isotropic Gaussian or bounded design, the minimax lower bound of the estimation error
% \begin{align*}
%     \inf_{\hat \bbeta} \sup_{\bbeta^* \in \RR^p} \EE_{\bbeta^*} \|\hat \bbeta - \bbeta^*\|^2 \gtrsim \frac{p}{n^2}
% \end{align*}
% }

%%%%%%%%%%%%%
%%
%%%%%%%%%%%%
%\section{Proof of Lemma \ref{l2_bound_chebyshev}}
%We first sort the critical inequalities in terms of the values of $|\varepsilon_{i}|$ from high to low and remove the top $f(n)$ critical inequalities, so there are $n-f(n)$ remaining. Then the analysis can be proceeded as Theorem \ref{theorem:l2bound_general} and Theorem \ref{theorem:l2bound_bdd} with sample size $n-f(n)$.

\begin{proof}[Proof of Proposition \ref{modded:lower:bound:theorem}] 
We will only indicate where the proof differs from the proof of Theorem \ref{theorem:lower_bound}. We set the matrix $\Rb$ to the any orthonormal matrix such that one of its rows contains the vector $\vb$ which minimizes $\inf_{\vb \in \mathbf{S}^{p-1}}\EE |\bX\T \vb|$. We then follow the proof of Theorem \ref{theorem:lower_bound} until equation \eqref{equation:which:is:different}:
\begin{align*}
\max_{\bnu, \bnu': d_{H}(\bnu, \bnu') \leq 1} \EE_{\Xb}\|P^{\otimes n}_{\bnu} - P^{\otimes n}_{\bnu'}\|_{\TV} \leq \max_{j\in[p]} \EE_{\Xb} \sum_{i \in [n]} \frac{|(\bX_i\T \Rb)_{j}| \delta}{a}.
\end{align*}
For the index $j$ corresponding to the vector $\vb$, we have 
\begin{align*}
 \EE_{\Xb} \sum_{i \in [n]} \frac{|(\bX_i\T \Rb)_{j}| \delta}{a} = \EE_{\Xb} \sum_{i \in [n]} \frac{|\bX_i\T \vb| \delta}{a} %\leq \max_{j\in[p]} n\frac{\delta}{a} \sqrt{\EE_{\Xb} ((\bX_i\T \Rb)_{j})^2} =\frac{n \delta}{a} \sqrt{\lambda_{\min} (\EE \bX\bX\T)}.
\end{align*}
Hence if one selects $\delta = \frac{a}{2n \EE|\bX\T \vb |} \geq \frac{a}{2n\sqrt{\lambda_{\min} (\EE \bX\bX\T)}}$ one would obtain that 
\begin{align*}
	\inf_{\hat \bbeta} \sup_{\bbeta^* \in \RR^p} \EE_{\bbeta^*} \|\hat \bbeta - \bbeta^*\|^2 \geq \frac{1}{4}\delta^2,
\end{align*} 
which shows the first part of the claim. For the second part the proof is identical to that of Theorem \ref{theorem:lower_bound}, except in the definition of $r_n$, $\sqrt{p}$ is equal to $1$ (i.e. it is not present). 

%We will consider two points $\bbeta^* = \mathbf{0}$ and $\bbeta^* = \delta \vb$, where $\vb$ is the  eigenvector corresponding to the smallest eigenvalue of the matrix $\EE \bX\bX\T$. Let $S = \{\mathbf{0}, \delta \vb\}$. 
%
%By the proof of Proposition 15.1 in \cite{wainwright...} we know that 
%
%\begin{align*}
%    \inf_{\hat \bbeta} \sup_{\bbeta^* \in \RR^p} \EE_{\bbeta^*} \|\hat \bbeta - \bbeta^*\|^2 &\geq \delta^2 \inf_{\hat \bbeta} \sup_{\bbeta^* \in \RR^p}  \PP_{\bbeta^*} (\|\hat \bbeta - \bbeta^*\| \geq \delta) \geq \delta^2 \inf_{\hat \bbeta} \sup_{\bbeta^* \in S}  \PP_{\bbeta^*} (\|\hat \bbeta - \bbeta^*\| \geq \delta) \\
%& \geq \delta^2 \frac{1 - \|\PP_{0} - \PP_{\delta \vb}\|_{\operatorname{TV}}}{2},
%\end{align*}
%where the last line follows from Le Cam's lemma (see (15.14) in \cite{wainwright...}), and we have denoted the data generating mechanisms under $\bbeta^* = \mathbf{0}$ and $\bbeta^* = \delta \vb$ with $\PP_{0}$  and $\PP_{\delta \vb}$ respectively. Now, following the same calculation as in the proof of Theorem \ref{theorem:lower_bound}, one can verify that  
%\begin{align*}
% \|\PP_{0} - \PP_{\delta \vb}\|_{\operatorname{TV}} \leq 
%\end{align*}
\end{proof}

%%%%%%%%%%%%%
%%
%%%%%%%%%%%%
%\section{}
\begin{proof}[Proof of Theorem \ref{theorem:l2_bound_chebyshev_lasso}]

Let us sort $\varepsilon_i$ in a decreasing manner in terms of their magnitude $|\bar \varepsilon_{(i)}|$, and keep the first $m$ terms. By the sharper bound in  Lemma \ref{lemma:uniform:random:variables:concentration} we know that $\PP(|\bar \varepsilon_{(m)}| \leq a(1 - \frac{m(L + 1)}{n}))  = \PP(|\bar \varepsilon_{(m)}| \leq a(1 - \frac{(L + 1) m}{n})) \leq  \exp\bigg(\frac{- mL^2/2}{\frac{4}{3}L + 1}\bigg) \rightarrow 0$ if $m \rightarrow \infty$. Hence $|\bar \varepsilon_{(m)}|\geq a(1-(L + 1)m/n)$ with probability at least $1 -\exp\bigg(\frac{- mL^2/2}{\frac{4}{3}L + 1}\bigg)$, where recall that $a$ is the parameter of uniform distribution of the noise $\varepsilon_i$. 

%Suppose we temporarily sort $|\varepsilon_i|$ in an increasing order. Then the $(n-n^{\alpha}+1)$-th order statistic is equivalent as $|\varepsilon_{n^{\alpha}}|$ in the decreasing order. Since the $k$-th order statistics of uniform $[0,1]$ distribution is beta distribution, one can see that the $\frac{|\varepsilon_{n^{\alpha}}|}{a}\sim Beta(n-n^{\alpha}+1, n^{\alpha})$, whose expectation is $\frac{n-n^{\alpha}+1}{n+1}$ and variance is $\frac{(n-n^{\alpha}+1)n^{\alpha}}{(n+1)^2(n+2)}$. Using Chebyshev's inequality for a constant $c>0$ we will get
%\begin{align*}
%    &\PP\Big[ \frac{|\varepsilon_{n^{\alpha}}|}{a} < \frac{n-n^{\alpha}+1}{n+1} - c \frac{(n-n^{\alpha}+1)n^{\alpha}}{(n+1)^2(n+2)}\Big] \leq \frac{1}{c^2}
%\end{align*}
%Since
%\begin{align*}
%    \frac{n-n^{\alpha}+1}{n+1} - c \frac{(n-n^{\alpha}+1)n^{\alpha}}{(n+1)^2(n+2)} > \frac{n+1-n^{\alpha}}{n+1}(1-\frac{c}{n+2}) > (1-\frac{1}{n^{1-\alpha}})(1-\frac{c}{n+2}),
%\end{align*}
%the above probability reduces to
%\begin{align*}
%    &\PP\Big[ |\varepsilon_{n^{\alpha}}| < a(1-\frac{c}{n+2})(1-\frac{1}{n^{1-\alpha}})
%    \Big] \leq \frac{1}{c^2},
%\end{align*}
%With $c = n^{\alpha/2}$ we can see that 
%\begin{align*}
%    |\varepsilon_{n^{\alpha}}|\geq a(1-\frac{n^{\alpha/2}}{n+2})(1-1/n^{1-\alpha})\geq a(1-1/n^{1 - \alpha/2})(1 - 1/n^{1-\alpha}) \geq a(1 - 2/n^{1-\alpha}),
%\end{align*}
%with high probability.

By keeping the first $m$ items of $|\bar \varepsilon_{(i)}|$, we have $m$ critical inequalities as
\begin{align*}
\eta_{(i)} \bX_{(i)}\T (\bbeta^* - \hat \bbeta) \leq \hat a - |\bar \varepsilon_{(i)}|,
\end{align*}
where $\eta_{(i)}$ and $\bX_{(i)}$ are the concomitant values to $|\bar \varepsilon_{(i)}|$ (and recall that they are independent from $|\bar \varepsilon_{(i)}|$). With a slight abuse of notation we will drop the $()$ brackets from the sub-indexing, and we will also write $\varepsilon_i$ for $\bar \varepsilon_{(i)}$.

Let $S$ be the support of $\bbeta^*$. First, we prove that either $\hat\bbeta-\bbeta^*\in \mathcal{C}(S, \gamma)$ for $\gamma = 1$ or $2$, or else $\lambda \|\bbeta^*_S - \hat \bbeta_S\|_1 < 2(L+ 1)am/n$. The definition of $\mathcal{C}(S, \gamma)$ can be seen in Definition \ref{re_pre}. We need the condition $\hat\bbeta-\bbeta^*\in \mathcal{C}(S, \gamma)$ in order to apply the RE condition (see Definition \ref{re_condition}) to bound $\|\hat\bbeta-\bbeta^*\|$ and consequently $\|\hat\bbeta-\bbeta^*\|_1$. Otherwise if $\lambda \|\bbeta^*_S - \hat \bbeta_S\|_1 < 2(L + 1)am/n$, the bound of $\|\hat\bbeta-\bbeta^*\|_1$ will be obtained immediately without any further derivation. We now consider several cases.
\vskip 0.5cm
\begin{enumerate}
    \item $\hat a > a$.\\
    From the optimization \eqref{chebyshev_lasso_est} we have the inequality
\begin{align*}
\hat a + \lambda\|\hat \bbeta\|_1 \leq  a + \lambda\| \bbeta^*\|_1,
\end{align*}
then by the fact $\|\bbeta^*\|_1=\|\bbeta^*_S\|_1$ and $a - \hat a<0$, we have
\begin{align*}
\lambda \|\hat \bbeta_{S^c}-\bbeta^*_{S^c}\|_1 = \lambda \|\hat \bbeta_{S^c}\|_1 \leq a - \hat a + \lambda\| \bbeta_S^*\|_1 - \lambda \|\hat \bbeta_S\|_1 \leq \lambda \| \bbeta_S^* - \hat \bbeta_S \|_1.
\end{align*}
Hence in this case $\hat\bbeta-\bbeta^*\in \mathcal{C}(S, 1)$.
    
    \item $\hat a \leq a$.\\
The critical inequalities can be written as
\begin{align*}
-\eta_i \bX_i\T (\bbeta^* - \hat \bbeta) \geq (|\varepsilon_i| - \hat a)
\end{align*}
First suppose that $a - \hat a > 2(L + 1) am/n$. It follows that $|\varepsilon_i| - \hat a \geq a(1-(L + 1)m/n)-\hat a \geq (a - \hat a)/2$.
By H\"{o}lder's inequality we have
\begin{align*}
   \bigg \|\frac{1}{m}\sum_{i \in [m]} \eta_i \bX_{i} \bigg\|_{\infty} \|\bbeta^* - \hat \bbeta\|_1 \geq \frac{1}{m}\sum_{i \in [m]} -\eta_i \bX_{i}\T (\bbeta^* - \hat \bbeta) \geq (a - \hat a)/2
\end{align*}
Now under assumption that $\bX_i$ is sub-Gaussian (actually a product of a matrix and a sub-Gaussian random vector), by Lemma \ref{bound_infnorm}, $\|\frac{1}{m}\sum_{i \in [m} \eta_i \bX_{i}\|_{\infty}\leq 2C'  \sqrt{\gamma \|\bSigma\|_{\operatorname{op}}^{1/2}\frac{\log p}{m}}$ in high probability, where $C'$ is an absolute constant. Denote with $R := C'\sqrt{\gamma \|\bSigma\|^{1/2}_{\operatorname{op}}}$. Hence we conclude that 
\begin{align}
\label{after_sg}
    \|\bbeta^*_{S^c} - \hat \bbeta_{S^c}\|_{1} \geq \frac{1}{4 R}\sqrt{\frac{m}{\log p}} (a - \hat a) - \|\bbeta^*_{S} - \hat \bbeta_{S}\|_{1}.
\end{align}

Suppose now $\|\bbeta^*_{S} - \hat \bbeta_{S}\|_{1} > \frac{1}{8 R}\sqrt{\frac{m}{\log p}} (a - \hat a)$. Hence for $\lambda\geq 8R \sqrt{\log p/m}$ we have $\lambda \|\bbeta^*_{S} - \hat \bbeta_{S}\|_{1} \geq a - \hat a$. Combine with the inequality
\begin{align*}
\lambda \|\hat \bbeta_{S^c}\|_1 \leq a - \hat a + \lambda \| \bbeta^*_S - \hat \bbeta_{S}\|_1,
\end{align*}
which can be deduced from the optimization \eqref{chebyshev_lasso_est}, we conclude
\begin{align*}
\| \bbeta^*_{S^c} - \hat \bbeta_{S^c}\|_1 = \|\hat \bbeta_{S^c}\|_1 \leq 2 \| \bbeta^*_S - \hat \bbeta_{S}\|_1.
\end{align*}

On the other hand if $\|\bbeta^*_{S} - \hat \bbeta_{S}\|_{1} < \frac{1}{8R}\sqrt{\frac{m}{\log p}} (a - \hat a)$, from \eqref{after_sg} we can deduct that
\begin{align*}
    \|\bbeta^*_{S^c} - \hat \bbeta_{S^c}\|_{1} \geq \frac{1}{8R}\sqrt{\frac{m}{\log p}} (a - \hat a).
\end{align*}
Again from the optimization \eqref{chebyshev_lasso_est} we have
\begin{align*}
 \|\hat \bbeta_{S^c}\|_1 \leq (a - \hat a)/\lambda + \| \bbeta^*_S - \hat \bbeta_{S}\|_1,
\end{align*}
so that for $\lambda\geq16R\sqrt{\log p/m}$,
\begin{align*}
    \| \bbeta^*_S - \hat \bbeta_{S}\|_1 \geq \frac{1}{8R}\sqrt{\frac{m}{\log p}} (a - \hat a) - (a - \hat a)/\lambda \geq (a - \hat a)/\lambda.
\end{align*}
This shows that in either case 
\begin{align*}
\| \bbeta^*_{S^c} - \hat \bbeta_{S^c}\|_1 = \|\hat \bbeta_{S^c}\|_1 \leq 2 \| \bbeta^*_S - \hat \bbeta_{S}\|_1.
\end{align*}

Then we will handle the case where $a - \hat a \leq 2(L + 1)am/n$. Suppose first that $\lambda \|\bbeta^*_S - \hat \bbeta_S\|_1 < 2(L + 1)am/n$. We can get the bound of $\|\hat\bbeta-\bbeta^*\|_1$ immediately since from \eqref{chebyshev_lasso_est} we can deduct $\lambda \|\hat \bbeta_{S^c}\|_1 \leq a - \hat a +\lambda \|\bbeta^*_S - \hat \bbeta_S\|_1 \leq 4(L + 1)am/n$. Adding the two bounds we conclude that $\|\bbeta^* - \hat \bbeta\|_1 \leq 6(L + 1)am/(\lambda n)$, which is a very fast rate provided that $\lambda$ is not too small. %since we will select $\lambda\sim1/\log n$ later.

Next assume that $\lambda \|\bbeta^*_S - \hat \bbeta_S\|_1 \geq 2(L + 1)am/n \geq a - \hat a$. Then again from \eqref{chebyshev_lasso_est} we can deduct
$$\lambda \|\bbeta^*_{S^c} - \hat \bbeta_{S^c}\|_1 = \lambda \|\hat \bbeta_{S^c}\|_1 \leq a - \hat a + \lambda \|\bbeta^*_S - \hat \bbeta_S\|_1 \leq 2 \lambda \|\bbeta^*_S - \hat \bbeta_S\|_1.$$
\end{enumerate}

From the discussions above, we can conclude that when $\lambda \|\bbeta^*_S - \hat \bbeta_S\|_1 \geq 2(L + 1)am/n$ we will have
\begin{align}
\label{re_pre_beta}
    \|\bbeta^*_{S^c} - \hat \bbeta_{S^c}\|_1 \leq 2 \|\bbeta^*_S - \hat \bbeta_S\|_1,
\end{align}
so that
\begin{align*}
    \bbeta^* - \hat \bbeta \in \mathcal{C}(S,2),
\end{align*}
where the definition of $\cC(S,2)$ can be found in Definition \ref{re_pre}.

\vskip 1cm
Next we will show that in the case when $\bbeta^* - \hat \bbeta \in \mathcal{C}(S,2)$, there are at least $m-m/\log n$ following inequalities
\begin{align}
\label{2way_bound_chebyshev_lasso}
    -C((L + 1)am/n+\lambda\|\hat\bbeta_S-\bbeta^*_S\|_1) \leq \eta_i\bX_i\T(\hat\bbeta-\bbeta^*) \leq (L +1)am/n+\lambda\|\hat\bbeta_S-\bbeta^*_S\|_1,
\end{align}
where $C =2 \log n$. 
%\vskip 1cm
The upper bound is quite simple.
From the optimization \eqref{chebyshev_lasso_est} we have the inequality
\begin{align*}
\hat a + \lambda\|\hat \bbeta\|_1 \leq  a + \lambda\| \bbeta^*\|_1,
\end{align*}
Then we have $\hat a \leq a + \lambda\| \bbeta_S^*\|_1 - \lambda \|\hat \bbeta_S\|_1-\lambda \|\hat \bbeta_{S^c}\|_1\leq a + \lambda \| \bbeta_S^* - \hat \bbeta_S \|_1$. Combine with $|\varepsilon_i|\geq a(1-(L + 1)m/n)$, our critical inequalities become
\begin{align*}
\eta_i \bX_i\T (\bbeta^* - \hat \bbeta) & \leq \hat a - |\varepsilon_i| \leq a - |\varepsilon_i| + \lambda \|\bbeta_S^* - \hat \bbeta_S\|_1\nonumber\\
& \leq  (L + 1)am/n+ \lambda \|\bbeta_S^* - \hat \bbeta_S\|_1.
\end{align*}

The lower bound $\eta_i \bX_i\T (\bbeta^* - \hat \bbeta) \geq -C( (L + 1)am/n+ \lambda \|\bbeta_S^* - \hat \bbeta_S\|_1)$ involves a bit more work. We will prove this by contradiction. Suppose that at least $m/\log n$ critical inequalities actually satisfy the following 
\begin{align*}
\eta_i \bX_i\T (\bbeta^* - \hat \bbeta) \leq -C( (L + 1)am/n+ \lambda \|\bbeta_S^* - \hat \bbeta_S\|_1),
\end{align*}
where we will fix $C$ later on. 

Consider the average
\begin{align*}
    \frac{1}{m}\sum_{i \in [m]} -\eta_i \bX_i\T (\bbeta^* - \hat \bbeta) &\geq \frac{1}{\log n}C( (L+ 1)am/n+ \lambda \|\bbeta_S^* - \hat \bbeta_S\|_1) - ( (L+ 1)am/n+ \lambda \|\bbeta_S^* - \hat \bbeta_S\|_1)\\
    &\geq (L+ 1)am/n+ \lambda \|\bbeta_S^* - \hat \bbeta_S\|_1,
\end{align*}
for any $C \geq 2 \log n$. By H\"{o}lder's inequality we have
\begin{align*}
    \bigg\|\frac{1}{m}\sum_{i \in [m]} \eta_i \bX_{i}\bigg\|_{\infty}\|\bbeta^* - \hat \bbeta\|_1 \geq \frac{1}{m}\sum_{i \in [m]} -\eta_i \bX_{i}\T (\bbeta^* - \hat \bbeta).
\end{align*}
Now under assumption that $\bX_i$ is sub-Gaussian, by Lemma \ref{bound_infnorm}, $\|\frac{1}{m}\sum_{i \in [m]} \eta_i \bX_{i}\|_{\infty}\leq 2 R\sqrt{\frac{\log p}{m}}$ in high probability, where $R = C' \sqrt{\gamma \|\bSigma\|^{1/2}_{\operatorname{op}}}$. Hence we conclude that 
\begin{align*}
    \|\bbeta^*_{S^c} - \hat \bbeta_{S^c}\|_{1} \geq ((L+ 1)am/n+ \lambda \|\bbeta_S^* - \hat \bbeta_S\|_1)\frac{1}{2 R}\sqrt{\frac{m}{\log p}} - \|\hat \bbeta_S - \bbeta_S^*\|_1.
\end{align*}
For $\lambda\geq6 R\sqrt{\log p/m}$, we have
\begin{align*}
    \|\bbeta^*_{S^c} - \hat \bbeta_{S^c}\|_{1} \geq \frac{(L + 1)am}{2Rn}\sqrt{\frac{m}{\log p}} + 2\|\bbeta_S - \bbeta_S^*\|_1 > 2\|\bbeta_S - \bbeta_S^*\|_1,
\end{align*}
which is a contradiction to \eqref{re_pre_beta}. Hence we conclude that at least $m - m/\log n$ inequalities satisfy the bound \eqref{2way_bound_chebyshev_lasso}, so they also satisfy
\begin{align*}
    [\eta_i\bX_i\T(\hat\bbeta-\bbeta^*)]^2\leq ((L+ 1)am/n+\lambda\|\hat\bbeta_S-\bbeta^*_S\|_1)^2 4\log^2n.
\end{align*}
Let $\cJ$ be the set of $i$ for which the above bounds hold. We have showed $|\cJ| \geq m - m/\log n$ in the case when $\bbeta^* - \hat \bbeta\in \cC(S,2)$.
%\vskip 1cm
Our next goal is to bound $\|\hat\bbeta-\bbeta^*\|$ by the RE condition. Recall that from the discussion above we either have $\|\bbeta^* - \hat \bbeta\|_1 \leq 6(L + 1)am/(\lambda n)$ or $\bbeta^* - \hat \bbeta\in \cC(S,2)$.

According to the RE condition on sub-Gaussian ensemble matrices \citep[Theorem 1.6]{zhou2009restricted}, given that the population covariance matrix $\bSigma$ satisfies the $RE(\kappa,2,s)$ condition, if $s\leq p/2$, $m\leq p$ and $m \gtrsim \|\bSigma\|_{\operatorname{op}}\gamma^4\kappa^{-2}(s\log(5ep/s) \vee \log p)$, with probability at least $1-2\exp(-cm/\gamma^4)$, we have $\Xb\T\Xb/m$ satisfies the $RE(\kappa(1-\theta),2,s)$ (for some fixed small $0 < \theta < 1$) condition. Here $c>0$ is an absolute constant. Notice that the $m$ inequalities are chosen in terms of $\varepsilon_i$, which is independent from $\bX_i$, so the theorem about RE condition in \cite{zhou2009restricted} applies to them. Now we need to ensure that if we select at least $m-m/\log n$ observations the sample covariance matrix will still satisfy the RE condition. This is easy since
\begin{align}
\label{RE_union_bound}
    \sum_{i = 0}^{m/\log n}{m \choose m - i} =     \sum_{i = 0}^{m/\log n}{m \choose i}  \leq (e\log n)^{m/\log n} \ll \exp\big(c(m-m/\log n)/\gamma^4\big),
\end{align}
where the first inequality is obtained by the fact $\sum_{i \leq k}{n \choose i} \leq (\frac{en}{k})^k$, and the second inequality is obtained by taking $\log$ on both sides. If we select $m-i$ observations for $i \leq m/\log n$, then with probability at most 
$$2\exp\big(-c(m-i)/\gamma^4\big) \leq 2\exp\big(-c(m-m/\log n)/\gamma^4\big)$$ 
the matrix $\Xb\T\Xb/(m-i)$ (where with a slight abuse of notation $\Xb$ denotes the selected matrix) doesn't satisfy the RE condition. Thus, by the union bound, the probability that for any $m-i$ for $i \leq m/\log n$ out of $m$ observations the sample covariance matrix will satisfy the RE is at least
\begin{align*}
    1 - 2\sum_{i = 0}^{m/\log n}{m\choose m - i}\exp\big(-c(m-m/\log n)/\gamma^4\big),
\end{align*}
which is close to $1$ according to the bound in \eqref{RE_union_bound}.

\vskip 0.5cm
Thus with probability converging to $1$ %{\color{red} $\kappa$ is $\kappa/2$}
\begin{align*}
\kappa (1 -\theta) \| \bbeta^*_S - \hat \bbeta_S\| & \leq \sqrt{\frac{1}{|\cJ|} \sum_{i \in \cJ} (\bX_i\T (\bbeta^* - \hat \bbeta))^2}\\
& \leq 2\log n\sqrt{((L+ 1)am/n+\lambda\|\hat\bbeta_S-\bbeta^*_S\|_1)^2} \\
& \leq 2\log n((L + 1)am/n+\lambda\sqrt{s}\|\hat\bbeta_S-\bbeta^*_S\|)\\
%& \leq 2\log n\sqrt{ \frac{8a^2}{n^{2-2\alpha}}+2\lambda^2s\|\hat\bbeta_S-\bbeta^*_S\|^2 },
\end{align*}
where $s=|S|$. Suppose that $\lambda \leq \frac{\kappa}{(4 + 4\theta/(1-\theta))\sqrt{s}\log n}$; we obtain that 
$$\|\bbeta^*_S - \hat \bbeta_S\| \leq \frac{4 (L + 1)am\log n}{(1-\theta)\kappa n}.$$ 
From here we immediately get a bound on the $\ell_1$ norm
\begin{align*}
\|\bbeta^*_S - \hat \bbeta_S\|_1 \leq \sqrt{s} \|\bbeta^*_S - \hat \bbeta_S\| \leq \frac{4(L + 1)am\sqrt{s}\log n}{(1-\theta)\kappa n}.
\end{align*}
And then combining with $\|\hat \bbeta_{S^c}\|_1 \leq 2\|\bbeta^*_S - \hat \bbeta_S\|_1$ we can get %and the fact $\|\bbeta^*_S - \hat \bbeta_S\|\leq\|\bbeta^*_S - \hat \bbeta_S\|_1$ we can get
%{\color{red} change this part of the proof since the theorem writing changed.}
\begin{align*}
\|\bbeta^* - \hat \bbeta\|_1 \leq 3\|\bbeta^*_S - \hat \bbeta_S\|_1 \leq \frac{12 a (L + 1)m \sqrt{s}\log n}{(1-\theta)\kappa n}.
\end{align*}
On the other hand, if $\|\bbeta^* - \hat \bbeta\|_1 \leq 6a(L + 1)m/(\lambda n)$ the conclusion directly follows. It remains to select $m$ as the minimum possible number from our requirements. We have required $m \geq \log n$, $m \leq p$, $m \gtrsim  \|\bSigma\|_{\operatorname{op}}\gamma^4\kappa^{-2}(s\log(5ep/s) \vee \log p)$ and $m$ satisfying $\kappa \sqrt{m/\log p} \gtrsim \gamma^{1/2} \|\bSigma\|^{1/4}_{\operatorname{op}}\sqrt{s}\log n$, where we mean . This completes the proof.
\end{proof}

\vskip 1cm
\begin{lemma}
\label{bound_infnorm}
If $\bX_{i}$ is sub-Gaussian as specified in Theorem \ref{theorem:l2_bound_chebyshev_lasso} we have
\begin{align*}
    \bigg\|\frac{1}{m}\sum_{i \in [m]} \eta_i \bX_{i}\bigg\|_{\infty} \leq 2C'  \sqrt{\gamma \|\bSigma\|^{1/2}_{\operatorname{op}}\frac{\log p}{m}},%2\sqrt{\frac{\log p}{n^{\alpha}}}
\end{align*}
where $C'$ is an absolute constant with high probability.
\end{lemma}
%{\color{red} what is $\eta_i$ here?}
%
%{\color{red} the constants here seem off!} 
\begin{proof}%{\color{red} This proof needs to be changed since we are assuming $\bX = \bSigma^{1/2} \bzeta$.}
Observe that $\eta\bX$ is a centered sub-Gaussian random variable with sub-Gaussian constant $\leq \gamma \|\bSigma\|^{1/2}_{\operatorname{op}}$. To see this set $\wb = \frac{\vb\T\bSigma^{1/2}}{\|\vb\T\bSigma^{1/2}\|}$, and note that
\begin{align*}
\EE \exp\bigg(\frac{\gamma^{-2}}{\|\bSigma\|_{\operatorname{op}}} (\vb\T \bX)^2\bigg) \leq \EE \exp(\gamma^{-2}(\wb\T \bzeta)^2) \leq 2.
\end{align*}

Since $\eta_i\bX_i$ is sub-Gaussian, each coordinate $\eta_i\bX_i^{(j)}$ is a one-dimensional sub-Gaussian variable with sub-Gaussian constant $\gamma\|\bSigma\|^{1/2}_{\operatorname{op}}$. 

By Lemma 5.9 of \cite{Vershynin2012Introduction}, we know that the random variable $\sum_{i} \eta_i\bX_i^{(j)}$ is sub-Gaussian with constant at most $C m \gamma\|\bSigma\|^{1/2}_{\operatorname{op}}$, where $C$ is an absolute constant. Next, using Lemma 5.5 (1) of \cite{Vershynin2012Introduction} and the union bound we conclude that for a constant $t>0$ we have
\begin{align*}
    \PP\Big( \|m^{-1}\sum_{i} \eta_i \bX_{i}\|_{\infty}\geq t \Big) \leq p e^{1-mt^2/(C' \gamma \|\bSigma\|^{1/2}_{\operatorname{op}})},
\end{align*}
with $C'$ being another absolute constant. Putting $t=2C' \gamma^{1/2} \|\bSigma\|^{1/4}_{\operatorname{op}} \sqrt{\frac{\log p}{m}}$ gives the desired result.
% {\color{red} Do you have a reference for the above result, because I don't understand why you have that?}
\end{proof}

\subsection{An optimal upper bound for standard Gaussian design}\label{otimal:upper:bound:for:gaussian:design:case}

In this subsection we make the case that when $\bX_i \sim N(0,\Ib_p)$ the lower bound of the order of $\frac{\sqrt{p}}{n}$ is tight, assuming that the noise variance $a$ is known. For simplicity let $a= 1$ (otherwise one can rescale $\bbeta^*$ to $\bbeta^*/a$). % Split the data (and for simplicity assume we have $2n$ observations). On the first part 
Construct $n!$ estimators which are linear regression based. In detail let $\cE = ((2k - n)/n)_{k \in [n]}\T$. For each of $n!$ permutations $\Pi$ construct the estimators:
\begin{align*}
\hat \bbeta_{\Pi} = (\Xb\T \Xb)^{-1} \Xb\T (\bY - \Pi \cE).
\end{align*}
Let $\hat \bbeta_C$ be the Chebyshev estimator. Let 
\begin{align}\label{cb:def}
\cB = \bigg\{\tilde\bbeta :\exists \Pi ~\mbox{s.t.} ~ \tilde\bbeta  = \hat \bbeta_{\Pi}, \|\tilde\bbeta - \hat \bbeta_C\| \leq \frac{2(L+1)(8\pi p \log(1 + 32\sqrt{2}\pi^{3/2} \sqrt{p}) + 8\pi \log\gamma^{-1} +1) }{\xi n}\bigg \},
\end{align}
where the constant in the bound is taken from Example \ref{Gaussian:example} (it is twice the constant in that example), and the constants $\gamma, \xi$ are specified there. Next consider ``playing'' a ``tournament'' for the (at most) $n!$ estimators in the set $\cB$. Specifically, for any estimator $\hat \bbeta \in \cB$, let $\cB_{\hat \bbeta} = \{\tilde \bbeta \in \cB: \|\tilde \bbeta- \hat \bbeta\| \geq 5\hat C \sqrt{p}/n\}$, for some sufficiently large constant $\hat C$. For any $\tilde \bbeta \in \cB_{\hat \bbeta}$ construct the numbers
\begin{align*}
\hat A = \frac{1}{n} \sum_{i \in [n]} (Y_i - \bX_i\T \tilde \bbeta)^2 - \frac{1}{n} \sum_{i \in [n]} (Y_i - \bX_i\T \hat \bbeta)^2  + \frac{1}{n}(\sum_i Y_i \bX_i - \bX_i \bX_i\T \hat \bbeta)\T (\tilde \bbeta  - \hat \bbeta),
\end{align*}
and 
\begin{align*}
\tilde A =  -\frac{1}{n} \sum_{i \in [n]} (Y_i - \bX_i\T \tilde \bbeta)^2 + \frac{1}{n} \sum_{i \in [n]} (Y_i - \bX_i\T \hat \bbeta)^2  + \frac{1}{n}(\sum_i Y_i \bX_i - \bX_i \bX_i\T \tilde \bbeta)\T (\hat \bbeta -\tilde \bbeta),
\end{align*}
which can be estimated from the data. If $\hat A > \tilde A$ say that $\hat \bbeta$ wins otherwise say that $\tilde \bbeta$ wins. Take any $\hat \bbeta$ that wins over all points in $\cB_{\hat \bbeta}$. We will prove that this succeeds to produce an estimation rate of $\sqrt{p}/n$ under the condition that $p^3 (\log p )^4 \ll n$. In more detail we have
\begin{theorem}\label{optimal:estimator:thm} Let $\hat \bbeta$ be the estimator selected by the above procedure. Suppose that $p^3 (\log p )^4 \ll n$. Then such an estimator exists, and in addition it satisfies that 
\begin{align}
\|\hat \bbeta- \bbeta^*\| \lesssim \frac{\sqrt{p}}{n},
\end{align}
with large probability (i.e. with a probability that can be made arbitrarily close to $1$ at the expense of increasing the constant in the bound, and in the algorithm).
\end{theorem}
\begin{proof} By the Dvoretky-Kiefer-Wolfowitz inequality we know that
\begin{align*}
\PP(\|\cE - \bvarepsilon^\uparrow\|_\infty \geq 2 t) \leq 2 e^{-2nt^2},
\end{align*}
where $\bvarepsilon^\uparrow$ is an increasing rearrangement of the error terms $\bvarepsilon$. Taking $t = C /\sqrt{n}$ for a large enough $C$ gives us that with high probability $\|\cE - \bvarepsilon^\uparrow\|_\infty \leq C/\sqrt{n}$, implying that $\|\cE - \bvarepsilon^\uparrow\| \leq C$. Next if one fits a regression on $\bY - \Pi \cE$ for $\Pi$ being any permutation one obtains
\begin{align*}
\hat \bbeta_{\Pi} = (\Xb\T \Xb)^{-1} \Xb\T (\bY - \Pi \cE) = \bbeta^* +  (\Xb\T \Xb)^{-1} \Xb\T (\varepsilon - \Pi \cE).
\end{align*}
Hence when $\Pi$ is $\hat \Pi$ which aligns $\hat \Pi \cE$ with $\bvarepsilon$ so that they are both sorted simultaneously we have that 
\begin{align*}
\|\hat \bbeta_{\hat \Pi} - \bbeta^*\| \leq \|(\Xb\T \Xb)^{-1}\|_{\operatorname{op}} \|\Xb\T (\bvarepsilon -\hat \Pi \cE)\| \leq 1/(\sqrt{n} - \sqrt{p} - t)^2 \sqrt{\chi^2(p)} C,
\end{align*}
with probability at least $1 -2\exp(-t^2/2)$ by Corollary 5.35 of \cite{Vershynin2012Introduction} and the $\chi^2(p) \lesssim 5p$ with probability at least $1 - \exp(-\sqrt{p})$, by Lemma 1 of \cite{laurent2000adaptive}. Hence at least one estimator $\hat \bbeta_{\hat \Pi}$ is $\sqrt{p}/n$ near the true point with high probability. Let the proportionality constant in the above bound be denoted by $\hat C$ (the bigger the $\hat C$ the bigger the probability of success can be made).

In addition fit the Chebyshev estimator on the data, and discard any estimator of the $n!$ ones that is more than $Cp\log p/n$ away from the Chebyshev one (the precise expression for $C$ is given in \eqref{cb:def}). We will now compare two estimators and $\hat \bbeta := \hat \bbeta_{\hat \Pi} \in \cB$ and $\tilde \bbeta \in \cB_{\hat \bbeta}$. Note that we have that $\hat \bbeta \in \cB$ by the triangle inequality. As described in the procedure we compare $\hat \bbeta$ and $\tilde \bbeta$ by comparing the two real numbers (which can be estimated from the data):
\begin{align*}
\hat A = \frac{1}{n} \sum_{i \in [n]} (Y_i - \bX_i\T \tilde \bbeta)^2 - \frac{1}{n} \sum_{i \in [n]} (Y_i - \bX_i\T \hat \bbeta)^2  + \frac{1}{n}(\sum_i Y_i \bX_i - \bX_i \bX_i\T \hat \bbeta)\T (\tilde \bbeta  - \hat \bbeta),
\end{align*}
and 
\begin{align*}
\tilde A =  -\frac{1}{n} \sum_{i \in [n]} (Y_i - \bX_i\T \tilde \bbeta)^2 + \frac{1}{n} \sum_{i \in [n]} (Y_i - \bX_i\T \hat \bbeta)^2  + \frac{1}{n}(\sum_i Y_i \bX_i - \bX_i \bX_i\T \tilde \bbeta)\T (\hat \bbeta -\tilde \bbeta)
\end{align*}

If $\hat A > \tilde A$ say that $\hat \bbeta$ wins otherwise say that $\tilde \bbeta$ wins. %Select any point that wins over all points outside of a radius $\kappa \sqrt{p}/n$ for some sufficiently large $\kappa$ to be specified later. 
%We claim that a selected point will be at least $\lesssim \sqrt{p}/n$ close to $\bbeta^*$. 
The true estimator $\hat \bbeta$ satisfies $\|\hat \bbeta - \bbeta^*\| \leq \hat C \sqrt{p}/n$, and any other estimator $\tilde \bbeta$ which satisfies $\|\tilde \bbeta - \bbeta^*\| \geq \kappa \sqrt{p}/n$ for some sufficiently large $\kappa \geq 4 \hat C$ (by the triangle inequality). We will show that with high probability $\hat A > \tilde A$. A simple calculation shows that 
\begin{align*}
\hat A = (\bbeta^* - \hat \bbeta)\T\hat \bSigma (\tilde \bbeta - \hat \bbeta) +  (\tilde \bbeta - \bbeta^*)\T\hat \bSigma (\tilde \bbeta - \bbeta^*) - (\hat \bbeta - \bbeta^*)\T\hat \bSigma (\hat \bbeta - \bbeta^*),
\end{align*}
where $\hat \bSigma = \Xb\T \Xb/n$. By Corollary 5.35 of \cite{Vershynin2012Introduction} we have that with high probability (at least $1 - 2\exp(-p/2)$) we have $1 - c\sqrt{p}/\sqrt{n}  \leq (1 - 2\sqrt{p}/\sqrt{n})^2 \leq \lambda_{\min}(\hat \bSigma) \leq \lambda_{\max}(\hat \bSigma) \leq (1 + 2\sqrt{p}/\sqrt{n})^2 \leq  1 + c\sqrt{p}/\sqrt{n}$ for some $c$. It follows by Cauchy-Schwartz that with high probability we have
\begin{align*}
\hat A &\geq -(1 + c \sqrt{p}/\sqrt{n})\|\bbeta^* - \hat \bbeta\| \|\tilde \bbeta - \hat \bbeta\| + (1 - c\sqrt{p}/\sqrt{n}) \|\tilde \bbeta - \bbeta^*\|^2 - (1 + c\sqrt{p}/\sqrt{n}) \|\hat \bbeta - \bbeta^*\|^2\geq\\
& -\|\bbeta^* - \hat \bbeta\| \|\tilde \bbeta - \hat \bbeta\| + (\|\tilde \bbeta - \bbeta^*\| - \|\hat \bbeta - \bbeta^*\|)\|\tilde \bbeta - \hat \bbeta\| - c\sqrt{p}/\sqrt{n} C p^2/n^2(\log p)^2\\
& \geq  (\|\tilde \bbeta - \bbeta^*\| - 2\|\hat \bbeta - \bbeta^*\|)\|\tilde \bbeta - \hat \bbeta\| - c\sqrt{p}/\sqrt{n} C p^2/n^2(\log p)^2,
\end{align*}
where we used that $\|\tilde \bbeta - \bbeta^*\| > \|\hat \bbeta - \bbeta^*\|$. On the other hand let us upper bound $\tilde A$. By a similar logic to before we can evaluate
\begin{align*}
\tilde A & = (\bbeta^* - \tilde \bbeta)\T\hat \bSigma (\hat \bbeta - \tilde \bbeta) +  (\hat \bbeta - \bbeta^*)\T\hat \bSigma (\hat \bbeta - \bbeta^*) - (\tilde \bbeta - \bbeta^*)\T\hat \bSigma (\tilde \bbeta - \bbeta^*)\\
& \leq (1 + c \sqrt{p}/\sqrt{n}) \|\bbeta^* - \tilde \bbeta\| \|\hat \bbeta - \tilde \bbeta\| + (1 + c \sqrt{p}/\sqrt{n}) \|\hat \bbeta - \bbeta^*\|^2 - (1 - c \sqrt{p}/\sqrt{n}) \|\tilde \bbeta - \bbeta^*\|^2\\
&\leq  \|\bbeta^* - \tilde \bbeta\| \|\hat \bbeta - \tilde \bbeta\| +  \|\hat \bbeta - \bbeta^*\|^2  -  \|\tilde \bbeta - \bbeta^*\|^2 + C c\sqrt{p}/\sqrt{n}p^2/n^2(\log p)^2\\
&\leq  \|\bbeta^* - \tilde \bbeta\| \|\hat \bbeta - \tilde \bbeta\|  - (\|\tilde \bbeta - \bbeta^*\| - \|\hat \bbeta - \bbeta^*\|) \|\hat \bbeta - \tilde \bbeta\|+ C c\sqrt{p}/\sqrt{n}p^2/n^2(\log p)^2\\
& \leq  \|\hat \bbeta - \bbeta^*\| \|\hat \bbeta - \tilde \bbeta\|+ Cc \sqrt{p}/\sqrt{n}p^2/n^2(\log p)^2,
\end{align*}
where we used the fact that $\|\tilde \bbeta - \bbeta^*\| \geq C \sqrt{p}/n \geq \|\hat \bbeta - \bbeta^*\|$. We conclude that if 
\begin{align*}
(\|\tilde \bbeta - \bbeta^*\| - 3\|\hat \bbeta - \bbeta^*\|)\|\tilde \bbeta - \hat \bbeta\| \geq 2c\sqrt{p}/\sqrt{n} C p^2/n^2(\log p)^2,
\end{align*}
$\hat \bbeta$ will always win. This is true however since $\|\tilde \bbeta - \bbeta^*\| \geq 4\|\hat \bbeta - \bbeta^*\| $ by assumption which also guarantees that $\|\tilde \bbeta - \hat \bbeta\| \geq \|\tilde \bbeta - \bbeta^*\| - \|\hat \bbeta - \bbeta^*\| \geq \sqrt{p}/n$. 

So when $p/n^2 >\sqrt{p}/\sqrt{n} C p^2/n^2(\log p)^2$ we will always have $\hat \bbeta$ win. That is the same as requiring $p^3 (\log p )^4 \ll n$. Hence we have established that $\hat \bbeta_{\hat \Pi}$ will win over all points outside of a radius $\kappa \sqrt{p}/n$. Therefore the selected point will not be more than $\kappa\sqrt{p}/n$ apart from  $\hat \bbeta_{\hat \Pi}$ which is at most $\lesssim \sqrt{p}/n$ from $\bbeta^*$. The proof is completed by the triangle inequality.
\end{proof}

\begin{remark} We would like to point out that this result remains valid for designs whose entries are i.i.d. centered sub-Gaussian random variables with variances equal to $1$ and sub-Gaussian parameter bounded by some constant $C < \infty$. This implies that the rows of $\Xb$ are i.i.d. sub-Gaussian isotropic random vectors. To see why the theorem extends to this setting, one needs to replace the application of \cite{laurent2000adaptive}'s Lemma 1 with a general sub-exponential bound such as the one offered by \cite{Vershynin2012Introduction}'s Proposition 5.16. In addition, the eigenvalue concentration of the matrix $\hat \bSigma$ can be deduced from Theorem 4.6.1 \citep{vershynin2018high}. The bound on the Chebyshev estimator can be taken from Example \ref{most:important:example}.
\end{remark}

\subsection{Lower bound for the Chebyshev estimator}

In this subsection we prove a general lower bound for the performance of the Chebyshev estimator. We start with a simple lemma on the order statistics of the error terms.

\begin{lemma}\label{lemma:uniform:random:variables:double:concentration}
Let $\{\varepsilon_i\}_{i \in [n]}$ be i.i.d. $U([-a,a])$ random variables. Sort the errors $|\varepsilon_i| \sim U([0,a])$ in decreasing manner $|\bar \varepsilon_{(i)}|$, so that $a \geq |\bar \varepsilon_{(1)}| \geq\ldots \geq |\bar \varepsilon_{(n)}| \geq 0$. Suppose $K \leq n$ is a fixed positive integer. Then 
\begin{align}
    \PP\Big[ \frac{|\bar \varepsilon_{(K)}|}{a} \geq 1 - \frac{K}{2n}\Big] \leq \exp\bigg(\frac{- 3K}{16}\bigg).
\label{bernstein_K:double}
\end{align}
%Thus by a union bound the inequalities above hold with constant probability over all $K$.
\end{lemma}

\begin{proof}[Proof of Lemma \ref{lemma:uniform:random:variables:double:concentration}]
Consider the inequality 
\begin{align*}
    a-|\bar \epsilon_{(K)}| \leq \theta \quad \Leftrightarrow \quad |\bar \epsilon_{(K)}| \geq a - \theta,
\end{align*}
for some $\theta$. Suppose now $\theta \leq a$. Denote the number of $|\bar \varepsilon_i|$ being in the interval $[a-\theta,a]$ with $Z$. If  $|\bar \varepsilon_{(K)}| \geq a - \theta$ then $Z \geq K$. Since $|\bar \varepsilon_i|\sim U([0,a])$, the probability for an individual $|\bar \varepsilon_i|$ falling into the interval $[a-\theta,a]$ is $\frac{\theta}{a}$. One can see $Z$ follows a binomial distribution $Bin(n, \frac{\theta}{a})$. By (a one-sided) Bernstein's inequality \citep[Theorem 2.8.4]{vershynin2018high} we have
\begin{align*}
    \mathbb{P}\bigg(Z\geq\frac{n\theta}{a}+t\bigg) \leq \exp\bigg(\frac{-t^2/2}{\frac{n\theta}{a}(1-\frac{\theta}{a})+\frac{t}{3}}\bigg) \leq  \exp\bigg(\frac{-t^2/2}{\frac{n\theta}{a}+\frac{t}{3}}\bigg).
\end{align*}
Set $\theta = \frac{Ka}{2n}$ and $t = K/2$. This yields $\frac{n\theta}{a} + t = K$, and
\begin{align*}
\PP(Z \geq K) \leq \exp\bigg(\frac{-t^2/2}{\frac{n\theta}{a}+\frac{t}{3}}\bigg) = \exp\bigg(\frac{- K^2/8}{K/2 + K/6}\bigg) \leq \exp\bigg(\frac{- 3K}{16}\bigg),
\end{align*}
which is what we wanted to show.

%so that
%\begin{align*}
%    \mathbb{P}(Z\geq K) \geq 1-\exp\big(\frac{-t^2/2}{\frac{n\theta}{a}(1-\frac{\theta}{a})+\frac{t}{3}}\big),\quad 0\leq K \leq \frac{n\theta}{a}-t
%\end{align*}
\end{proof}
%}

\begin{theorem}\label{Chebyshev:lower:bound:thm} Suppose the the matrix $\Xb$ has i.i.d. standard Gaussian entries. With at least a constant probability we have that the Chebyshev estimator $\hat \bbeta$ satisfies
\begin{align*}
\|\hat \bbeta - \bbeta^*\| \gtrsim a p/(n (\log n)^{3/2}),
\end{align*}
where the inequality $\gtrsim$ hides absolute constant factors.
\end{theorem}

\begin{proof}[Proof of Theorem \ref{Chebyshev:lower:bound:thm}]
Without loss of generality we assume $a = 1$. We have that $\hat a = \min_{\bbeta} \|\bY - \Xb\bbeta \|_{\infty} = \min_{\vb} \max_i |\varepsilon_i + \bX_i\T \vb| = \min_{\vb} \max_{\eb: \|\eb\|_1 \leq 1} \eb\T(\bvarepsilon + \Xb \vb)$. We can then write,
\begin{align*}
\hat a \leq \hat a(R) := \min_{\vb: \|\vb\| \leq  R} \max_{\eb: \|\eb\|_1 \leq 1} \eb\T(\bvarepsilon + \Xb \vb),
\end{align*}
for some $R > 0$. Applying the minimax theorem gives us that 
\begin{align*}
\hat a(R) = \max_{\eb: \|\eb\|_1 \leq 1} \min_{\vb: \|\vb\| \leq  R}  \eb\T(\bvarepsilon + \Xb \vb) = \max_{\eb: \|\eb\|_1 \leq 1} \eb\T\bvarepsilon  - R \|\eb\T \Xb\|
\end{align*}
Taking $R \rightarrow \infty$, shows that $\hat a \leq \max_{\eb: \|\eb\|_1 \leq 1, \eb\T \Xb = 0} \eb\T\bvarepsilon$. Next using Theorem 9.1.1 \citep{vershynin2018high}, we have that 
\begin{align*}
\EE \sup_{\eb: \|\eb\|_1 \leq 1} |\|\eb\T \Xb\| - \sqrt{p}\|\eb\|| \lesssim \sqrt{\log n}, 
\end{align*}
for some absolute constant, where we used that the Gaussian width of the $\ell_1$ ball is $\sqrt{\log n}$ up to constant factors. Hence by Chebyshev's inequality we can claim that with probability at least $.99$ for all $\eb: \|\eb\|_1 \leq 1$ we have $\|\eb\T \Xb\| \geq \sqrt{p}\|\eb\| - C \sqrt{\log n}$ for a sufficiently large absolute constant $C$. It follows that $\hat a \leq \max_{\eb: \|\eb\|_1 \leq 1, \|\eb\| \leq C \sqrt{\log n/p}} \eb\T\bvarepsilon$. %Now we write each $\varepsilon_i= 1 - \nu_i$ and we obtain $\hat a \leq 1 - \min_{\eb: \|\eb\|_1 \leq 1, \|\eb\| \leq C \sqrt{\log p/n}} \eb\T\bnu$. 
Let $s$ be the biggest integer smaller than $s \leq [C\sqrt{\log n /p}]^{-2}/4$. Then $\|\eb_S\|_1 \leq \sqrt{s}C \sqrt{\log n/p} \leq 1/2$, where $S$ is the support of the maximal $s$ coefficients of $\eb$ corresponding to the maximal $s$ values of $\bvarepsilon$ (which always needs to be the case due to the rearrangement inequality). By Lemma \ref{lemma:uniform:random:variables:double:concentration} we know that with at least a constant probability,
\begin{align*}
|\bar \varepsilon_{(s + 1)}| & \leq  1 - (s+1)/(2n)  \leq 1 - (p/(4C^2\log n))/(2n) \\
& \leq 1 - \kappa'(p/(n\log n)). 
\end{align*}
Since $\|\bvarepsilon\|_{\infty} \leq 1$ we have $\eb\T \bvarepsilon \leq \|\eb_S\|_1 + (1 - \|\eb_S\|_1)|\bar \varepsilon_{(s + 1)}|  \leq 1 - 1/2 \kappa'(p/n\log n)$. We conclude that
\begin{align*}
\hat a \leq 1 - \kappa''p/(n\log n).
\end{align*}
Now by Lemma \ref{lemma:uniform:random:variables:concentration} we know that with constant probability for $L$ large enough,
\begin{align*}
(1-i(L+1)/n)- ( 1 - \kappa''p/(n\log n))\leq |\bar \varepsilon_{(i)}| - \hat a \leq |\bX_{(i)}\T (\hat \bbeta - \bbeta^*)|,
\end{align*}
and the LHS is positive for the first $\approx \kappa'' p/((L+1)\log n)$ entries, and where $\bX_{(i)}$ are the concomitant $\bX_i$ values for the top order statistics. Squaring and adding these inequalities yields, 
\begin{align}
\sum_{i < \kappa'' p/((L+1)\log n)} (\kappa''p/(n\log n) - i(L+1)/n)^2 & \leq (\hat \bbeta - \bbeta^*)\T \sum \bX_{(i)} \bX_{(i)}\T(\hat \bbeta - \bbeta^*) \nonumber \\
& \leq (p/((L+1)\log n) + p + p) \|\hat \bbeta  - \bbeta^*\|^2 \label{upper:bound:for:beta},
\end{align}
with probability at least $1-2\exp(cp^2)$ for some absolute constant $c$, where we used Corollary 7.3.3 of \cite{vershynin2018high}. On the other hand we have
\begin{align*}
\sum_{i < \kappa'' p/((L+1)\log n)} (\kappa''p/(n\log n) - i(L+1)/n)^2 & \geq \sum_{i =1}^{\lfloor \kappa'' p/((L+1)\log n) \rfloor} (\bar i(L+1)/n - i(L+1)/n)^2\\
&= \frac{(L+1)^2}{n^2}\bigg[\sum_{i =1}^{\lfloor \kappa'' p/((L+1)\log n) \rfloor} (i^2 - \bar i^2)\bigg]\\
& \gtrsim \frac{(L+1)^2}{n^2} \lfloor \kappa'' p/((L+1)\log n) \rfloor^3,
\end{align*}
where $\bar i =  \sum_{i =1}^{\lfloor \kappa'' p/((L+1)\log n) \rfloor} i/\lfloor \kappa'' p/((L+1)\log n) \rfloor = (\lfloor \kappa'' p/((L+1)\log n) \rfloor+1)/2$. Dividing \eqref{upper:bound:for:beta} by $3p$ yields that 
\begin{align*}
\|\hat \bbeta - \bbeta^*\| \gtrsim p/(n (\log n)^{3/2}),
\end{align*}
with at least constant probability. 
\end{proof}

\begin{remark} The proof remains valid if one substitutes the entries of the design matrix $\Xb$ with i.i.d. centered sub-Gaussian random variables with variances equal to $1$ and sub-Gaussian parameter bounded by some $C < \infty$. This implies that the rows and columns of $\Xb$ are i.i.d. sub-Gaussian isotropic random vectors. \eqref{upper:bound:for:beta} needs to be replaced with the eigenvalue concentration of the matrix $\hat \bSigma$ which can be deduced from Theorem 4.6.1 \citep{vershynin2018high}.

\end{remark}

\end{document}